\documentclass[11pt, a4paper]{article}
\usepackage{amsmath,amssymb,amsfonts}
\usepackage{amsthm}
\usepackage{verbatim}
\usepackage{a4wide}
\usepackage{epsfig,epic,eepic,graphicx}
\usepackage{enumerate}

\usepackage{graphicx}
\usepackage[usenames,dvipsnames]{color}
\usepackage[colorlinks=true, pdfstartview=FitV, linkcolor=blue, citecolor=blue, urlcolor=blue]{hyperref}

\newtheorem{theorem}{Theorem}[section]
\newtheorem{lemma}[theorem]{Lemma}
\newtheorem{proposition}[theorem]{Proposition}

\newtheorem{remark}[theorem]{Remark}
\newtheorem{corollary}[theorem]{Corollary}

\numberwithin{equation}{section}
\makeatletter

\@addtoreset{equation}{section}
\makeatother

\newcommand{\T}{{\mathbb T}}
\newcommand{\N}{{\mathbb N}}
\newcommand{\Z}{{\mathbb Z}}

\newcommand{\R}{{\mathbb R}}
\newcommand{\C}{{\mathbb C}}
\newcommand{\pa}{{\partial}}
\newcommand{\na}{{\nabla}}

\newcommand{\eps}{{\varepsilon}}
\newcommand{\be}{\begin{equation}}
\newcommand{\ee}{\end{equation}}

\newcommand{\ba}{\begin{aligned}}
\newcommand{\ea}{\end{aligned}}

\def\Im{{\cal I}m \,}

\def\Re{{\rm Re}  \,}
\def\bul{{$\bullet$\hspace*{2mm}}}

\def\mspace{\medskip \noindent}

\title{Sobolev stability of Prandtl expansions \\
 for the steady Navier-Stokes equations}

\author{\footnote{Universit\'e Paris Diderot and IUF, Sorbonne Paris Cit\'e, Institut de Math\'ematiques de Jussieu-Paris Rive Gauche, UMR 7586, F- 75205 Paris, France} David Gerard-Varet \and \footnote{Department of Mathematics, Graduate School of Science, Kyoto University,
Kitashirakawa Oiwake-cho, Sakyo-ku, Kyoto 606-8502, Japan } Yasunori Maekawa}

\begin{document}
\maketitle

\begin{abstract}
We show the $H^1$ stability of shear flows of Prandtl type: $\displaystyle U^\nu = \left(U_s(y/\sqrt{\nu}), 0\right)$,
in the steady  two-dimensional Navier-Stokes equations, under the natural assumptions that $U_s(Y) > 0$ for $Y > 0$, $U_s(0) = 0$, and $U'_s(0) > 0$.  Our result is in sharp contrast with the unsteady ones, in which at most Gevrey stability can be obtained, even under global monotonicity and concavity hypotheses. It  provides the first positive answer to the  inviscid limit problem in Sobolev regularity  for a non-trivial class of steady Navier-Stokes flows with no-slip boundary condition. 
\end{abstract}
\section{Introduction and main result}\label{sec.intro}
Our concern in this paper is the vanishing viscosity limit of the  two-dimensional steady Navier-Stokes equations:
\begin{align}\label{eq.NS}
\begin{cases}
& v^\nu \cdot \nabla v^\nu - \nu \Delta v^\nu + \nabla q^\nu = g^\nu \,, \quad (x,y) \in \T_\kappa \times \R_+\,,\\
& {\rm div}\, v^\nu =0\,, \quad (x,y) \in \T_\kappa\times \R_+\,,\\
& v^\nu|_{y=0} =0.
\end{cases}
\end{align}
 Here $\T_\kappa=\R/ (2\pi \kappa) \Z$, $\kappa>0$, is a torus with periodicity $2\pi \kappa$, $\R_+=\{y\in \R~|~y>0\}$, while $v^\nu =(v^\nu_1,v^\nu_2)$ and $q^\nu$ are respectively the unknown velocity field and pressure field of the fluid. The positive constant $\nu$ is the viscosity coefficient.  The vector field $g^\nu$ is an external force, decaying fast enough at infinity. The usual no-slip condition is prescribed at $y=0$. 
 
\mspace
Understanding the behaviour of   $v^\nu$  for small $\nu$ is a classical and difficult problem: $\na v^\nu$ tends to blow-up near the boundary as $\nu \rightarrow 0$,  and the dynamics of this so-called boundary layer is uneasy to analyze. A main step forward  was made by L. Prandtl in 1904, who suggested asymptotics of the form 
\begin{equation} \label{asymptprandtl} 
\begin{aligned}
v^\nu(x,y) & \sim \bigl( V_1(x,y/\sqrt{\nu}), \sqrt{\nu} V_2(x,y/\sqrt{\nu})\bigr) \quad \text{near the boundary,} \\
v^\nu(x,y) & \sim v^0(x,y) \quad \text{away from the boundary}, 
\end{aligned}
\end{equation}
where  $V = (V_1, V_2)(x,Y)$ depends on a rescaled variable $Y = y /\sqrt{\nu}$. Hence, in the Prandtl model, the boundary layer has a characteristic scale $\sqrt{\nu}$. Moreover, it connects to an Euler solution $v^0$ as $Y \rightarrow +\infty$. By plugging the expansion in \eqref{eq.NS}, one obtains a kind of reduced Navier-Stokes system on $V$,  the Prandtl equation, now classical in fluid dynamics.  Nevertheless, as pointed out by Prandtl himself, this formal  asymptotics is expected to have  a limited range of validity, due to an instability phenomenon called boundary layer separation. This instability is typical of flows around obstacles. Roughly, under an adverse pressure gradient in the boundary layer, past a certain distance $x= x_*$ from the leading edge of the obstacle,  the stress $\pa_y v_1^\nu\vert_{y=0}$ may vanish. This leads to the appearance  of a reverse flow for $x > x_*$,  and detachment of the boundary layer streamlines: see \cite[page 39]{Schlichting} for a more detailed description of the underlying physics and illustrations. 

\mspace
Mathematically, the importance of this phenomenon has been well recognized in the analysis of  the steady Prandtl model.  On one hand, it is known from the works of Oleinik \cite{Ole} that given a horizontal velocity $V_1$ at $x=0$ satisfying $V_1\vert_{x=0} > 0$, $\pa_Y V_1\vert_{x=0,Y=0} > 0$, one can construct a local in $x$ smooth solution of the Prandtl equation. This result is based on the so-called Von Mises transform, which turns the Prandtl equation into a  nonlinear heat equation, with $x$ as an  evolution variable.  Moreover, this smooth solution exists as long as $V_1 > 0$ and 
$\pa_Y V_1\vert_{Y=0} > 0$. On the other hand, there exists  blowing-up solutions: it was established recently in \cite{DalMas}, see also \cite{Gol,MaShi,E}. Still, these results leave aside the behaviour of the full system \eqref{eq.NS}, and the justification of the Prandtl asymptotics \eqref{asymptprandtl} prior to separation. The purpose of the present paper is to contribute to fill in this gap.

\mspace
Let us stress that most recent mathematical results on the validity of the Prandtl asymptotics are actually related to the unsteady Navier-Stokes equations. In such  case, it is now well-understood that the justification of the Prandtl approach requires stringent assumptions on the data. The underlying reason is the presence of many hydrodynamic instabilities. We refer for instance to \cite{Cowley,HoHu, GaSaSci} for discussions and numerics around the various blow-up scenarios. Even to hope for short time  stability, one must impose either restrictions on the structure of the perturbations \cite{LoMaNu, MaTa}, or strong regularity assumptions. As regards the well-posedness of the Prandtl model, we refer to \cite{LoCaSa, KuVi,GeDo,AlWaXuYa, MaWo,XiZa, GeMa,Tong1} and citations therein. As regards the full Navier-Stokes model, a complete  justification of the Prandtl theory was obtained  for analytic data \cite{SaCa1,SaCa2,WaWaZh} and for the initial vorticity supported away from the boundary \cite{Ma, FeTaZha}. 
On the contrary, counterexamples to the $H^1$ stability of Prandtl expansions of shear flow type was provided by Grenier in \cite{Gre}, using boundary layer profiles with inflexion points.  Even in the favourable case of monotonic and concave boundary layer profiles, the boundary layer expansion \eqref{asymptprandtl} is  not stable in a Sobolev framework. This is due to a viscous instability mechanism, the so-called Tollmien-Schlichting wave. 
This instability, identified in the first half of the 20th century \cite{DrRe}, was examined in a nice article by Grenier, Guo and Nguyen \cite{GGN2014}. Properly rescaled, their analysis provides highly growing eigenmodes of the linearized Navier-Stokes system around a shear flow of Prandtl type. These eigenmodes have high $x$-frequency    $n \sim \nu^{-3/8}$, and associated growth rate $\sigma \sim n^{2/3} \sim  \nu^{-1/4}$. For arbitrary small  $\nu$, these high frequencies must have very small initial amplitude to be controlled on a time scale independent of $\nu$: namely, one can only hope for a short time stability result in functional spaces of  Gevrey class $3/2$ in $x$. A result in this direction was obtained recently by the authors and N. Masmoudi in \cite{GeMaMa}.  See \cite{GrNg} for related statements.

\mspace
In view of the complexity of the unsteady framework, one could think that justifying Prandtl expansions in the steady case should require stringent assumptions. Surprisingly, we will be able to show local in space Sobolev stability of shear flow solutions of Prandtl type: 
$U^\nu(x,y) = \left(U_s(y/\sqrt{\nu}), 0\right)$, 
under the main assumptions $U_s(Y) > 0$ for $Y > 0$, $U_s(0) = 0$, and  $U'_s(0) > 0$. These assumptions are somehow minimal in view of the previous discussion: they forbid reverse flow and boundary layer separation. As far as we know, this is the first boundary layer stability result  for the steady Navier-Stokes equations with the usual no-slip conditions\footnote{see Remark \ref{rem.Guo} for an update.}. {The only previous articles that we are aware of} are  \cite{GuNgu2017,Iyer2017}, dedicated to inhomogeneous Dirichlet conditions. For instance, Guo and Nguyen consider  in \cite{GuNgu2017} the steady Navier-Stokes equations in a half-plane, but with a positive Dirichlet datum for the horizontal velocity. They construct general boundary layer expansions for this problem and prove their Sobolev stability through the use of original energy functionals. Let us mention that similar ideas are encountered in the context of the non-stationary MHD equations, where Sobolev stability can be recovered if the magnetic field has a non-vanishing tangential component at the boundary: see \cite{GePr,Tong2,Tong3}. In the present paper, the analysis is of a different nature, and centered on  handling the degeneracy due to  the homogeneous Dirichlet condition.

\mspace
We now state precisely our main result. Let $U_s = U_s(Y) \in C^2 (\overline{\R_+})$ such that 
\begin{align}
& U_s (0) =0 \,, \qquad  U_s>0 ~~ {\rm in}~Y>0\,, \qquad \lim_{Y\rightarrow \infty} U_s (Y) = U_E > 0\,, \label{assume.1}\\
& \pa_Y U_s (0) >0 \,,\label{assume.2}\\
&  \sum_{k=1,2} \sup_{Y\geq 0}  (1+Y)^{3}  |\pa_Y^k U_s (Y) |  <\infty\,.\label{assume.3}
\end{align}
From the continuity and \eqref{assume.2} we have $\pa_Y U_s >0$ on $0\leq Y\leq 4 Y_0$ for some $Y_0\in (0,1]$.
This nondegeneracy near the boundary will be crucial. We then consider the shear flow  
\begin{align}\label{def.U_s'}
U^\nu = (U_s^\nu (y), 0) \,, \qquad U_s^\nu (y) = U_s (y/\sqrt{\nu}). 
\end{align} 
Obviously, \eqref{def.U_s'} can be seen as a solution of \eqref{eq.NS}, setting  $\displaystyle g^\nu =  - \nu \pa^2_y U^\nu$ and $q^\nu = 0$. The goal of the paper is to establish stability estimates for this solution of boundary layer type. Denoting  $\displaystyle u^\nu = v^\nu - U^\nu$ the perturbation  induced by  $\displaystyle f^\nu =  g^\nu + \nu \pa^2_y U^\nu$, we get 
\begin{align}\label{eq.pNS}
\begin{cases}
& U_s^\nu \pa_x u^\nu + u^\nu_2 \pa_y U_s^\nu {\bf e}_1 - \nu \Delta u^\nu + \nabla p^\nu = - u^\nu \cdot \nabla u^\nu + f^\nu\,,  \quad (x,y) \in \T_\kappa \times \R_+\,,\\
& {\rm div}\, u^\nu =0\,, \quad (x,y) \in \T_\kappa\times \R_+\,,\\
& u^\nu|_{y=0} =0.
\end{cases}
\end{align}
Here ${\bf e}_1=(1,0)$. We then have to specify a functional setting, with $2\pi \kappa$ periodicity in $x$.  Let $\mathcal{P}_n$, $n\in \Z$, be the orthogonal projection on the $n$-th Fourier  mode in variable $x$:
\begin{align}\label{def.P_n}
(\mathcal{P}_n u ) (x,y) = u_n (y) e^{i\tilde n x}\,, \qquad \tilde n = \frac{n}{\kappa}\,, \qquad u_n (y) = \frac{1}{2\pi \kappa} \int_0^{2\pi \kappa} u (x,y) e^{-i \tilde n x} d x\,,
\end{align}
The divergence-free and homogeneous Dirichlet conditions imply $u_0 =(u_{0,1},0)$. Setting  
\begin{align}\label{def.Q_n}
\mathcal{Q}_0 u = (I-\mathcal{P}_0) u\,,
\end{align}
where $I$ is the identity operator, we can identify $u$ with the couple $(u_{0,1}, \mathcal{Q}_0 u)$. With this identification we introduce 
\begin{align}\label{def.X}
\begin{split}
X & = \big \{ (u_{0,1}, \mathcal{Q}_0 u) \in  BC (\overline{\R_+}) \ \times W^{1,2}_0 (\T_\kappa \times \R_+)^2 ~|~   \quad \pa_y u_{0,1} \in L^2 (\R_+)\,, \qquad  u_{0,1}|_{y=0}  = 0\,,\\
& \quad \| u\|_X = \| u_{0,1}\|_{L^\infty(\R_+)} + \| \pa_y u_{0,1} \|_{L^2(\R_+)} + \sum_{n\ne 0} \| u_n \|_{L^\infty(\R_+)}  + \| \mathcal{Q}_0 u \|_{W^{1,2}(\T_\kappa \times \R_+)} <\infty \big \}\,,
\end{split}
\end{align}
where the Sobolev space $W^{1,2}_0 (\T_\kappa \times \R_+)$ is defined as the subspace of $W^{1,2} (\T_\kappa \times \R_+)$ with functions having the zero boundary trace on $y=0$.
For simplicity we assume that $f^\nu = \mathcal{Q}_0 f^\nu$ below, though it is not difficult to extend our result to a general case by imposing a suitable condition on $f^\nu_0 (y)$. 
\begin{theorem}\label{thm.main} There exist positive numbers $\kappa_0, \nu_0, \epsilon$ such that the following statement holds for $0<\kappa\leq \kappa_0$ and $0<\nu\leq \nu_0$.
If $f^\nu = \mathcal{Q}_0 f^\nu$ and $\| f^\nu \|_{L^2} \leq \epsilon \nu^\frac34 |\log \nu|^{-1}$ then there exists a unique solution $(u^\nu, \nabla p^\nu) \in \big ( X\cap W^{2,2}_{loc} (\T_\kappa\times \R_+)^2\big ) \times L^2 (\T_\kappa\times\R_+)^2$ to \eqref{eq.pNS} such that 
\begin{align}\label{est.thm.main.1}
\begin{split}
& \| u^\nu_{0,1} \|_{L^\infty} + \nu^\frac14 \| \pa_y u^\nu_{0,1}\|_{L^2}  \\
& \qquad + \sum_{n\ne 0} \| u_n^\nu \|_{L^\infty} + \nu^{-\frac14} \| \mathcal{Q}_0 u^\nu \|_{L^2} + \nu^\frac14 \| \nabla \mathcal{Q}_0 u^\nu \|_{L^2}  \leq \frac{C|\log \nu|^\frac12}{\nu^\frac{1}{4}} \| f^\nu \|_{L^2}\,,
\end{split}
\end{align}
Here $C$ is independent of $\nu$ and $\kappa$.
\end{theorem}

\begin{remark}{\rm The main structural assumptions of our stability theorems are \eqref{assume.1} and  \eqref{assume.2}, which are natural in view of the previous comments on boundary layer separation. Another important requirement is the smallness condition on $\kappa$: it means that our stability result is only local in space (although on a lengthscale independent of $\nu$). An interesting question is wether this locality requirement is only a technical restriction of our stability method, or if instabilities are possible at larger wavelengths.}   
\end{remark} 

\begin{remark}
{\rm The perturbation $u^\nu$ converges (at least in a weak sense) to a constant shear flow at infinity: 
\begin{equation} \label{eq.limvinfty} 
\lim_{y \rightarrow +\infty} v^\nu = (c^\nu,0). 
\end{equation}
First, the requirement $\mathcal{Q}_0 u^\nu \in W^{1,2}$ implies that  $\mathcal{Q}_0 u^\nu$ goes to zero at infinity. Then, as regards the $x$-average $u_0^\nu = (u_{0,1}^\nu,0)$, we deduce from the first line of \eqref{eq.pNS} and the fact  that $f^\nu_0 = 0$: 
$$ - \nu \pa^2_y u_{0,1}^\nu = - \pa_y (Q_0 u^\nu_2 Q_0 u^\nu_1)_0. $$
As $\pa_y u_{0,1} \in L^2$, we can integrate this identity from $y=+\infty$ to deduce 
$$  -\nu \pa_y u_{0,1}^\nu = - (Q_0 u^\nu_2 Q_0 u^\nu_1)_0 $$
Eventually, as the right-hand side belongs to $L^1$, we find \eqref{eq.limvinfty} with
$$ c^\nu = \frac{1}{\nu }\int_{\R_+}  (Q_0 u^\nu_2 Q_0 u^\nu_1)_0. $$
Note that this constant at infinity can not be prescribed. Moreover, it obeys the bound 
$$ |c^\nu| \le \frac{C|\log \nu|^\frac12}{\nu^\frac{1}{4}} \| f^\nu \|_{L^2}$$
as a consequence of estimate \eqref{est.thm.main.1}.
}
\end{remark}
\begin{remark}
{\rm Theorem \ref{thm.main} has an easy implication on the inviscid limit problem. Namely, any solution $v^\nu$ of \eqref{eq.NS} associated to a source term of the form  $g^\nu = - \nu \pa^2_y U^\nu + f^\nu$, where $U^\nu$ satisfies \eqref{assume.1}-\eqref{assume.2}-\eqref{assume.3}-\eqref{def.U_s'} and  $\| f^\nu \|_{L^2} = o(\nu^{3/4}/|\log\nu|^{-1})$, converges to the Euler solution $v^0 = (U_E,0)$  in $L^\infty_{loc} \cap H^1_{loc}$. This is the first non-trivial example of a class of steady Navier-Stokes solutions for which the  inviscid limit holds true in finite regularity.}
\end{remark}
{\begin{remark} \label{rem.Guo} {\rm {\em (Update 07/31/2018)} Just after our manuscript submission  on the arXiv, Y. Guo and S. Iyer have submitted the very interesting preprint \cite{GuoIyer2018}. They establish there the Sobolev stability of a subclass of Prandtl expansions,  the main example of which being the famous Blasius flow. We feel that this work and ours are complementary: the focus of \cite{GuoIyer2018} is the very important Blasius self-similar solution (with $x$-dependence), while our analysis is more elementary: it treats the simpler case of shear flows under minimal assumptions. In this perspective, we hope that our paper will be a first step in  the stability analysis of general steady  boundary layer expansions, able to cover both works\footnote{{As regards the unsteady case, extension of article \cite{GeMaMa} to general $x$-dependent expansions is in progress.}}. Moreover, beyond the context of boundary layer flows, we feel that our analysis  may provide new tools to investigate the stability of shear flows, which is a classical topic in hydrodynamics.}    
\end{remark}
}
\noindent
To conclude this introduction, we give the outline of the proof of Theorem \ref{thm.main}. After collecting a few useful estimates in Section \ref{sec.pre}, we will turn to the core of the proof, which is the analysis of the linearized system around $U^\nu$. Through a Fourier transform in $x$, it can be written
\begin{align}\label{eq.linearized}
\begin{cases}
& i \tilde n U_s^\nu u_n + u_{n,2} (\pa_y U_s^\nu ) {\bf e_1}  - \nu (\pa_y^2 - \tilde  n^2) u_n +
\begin{pmatrix}
i \tilde n p_n\\
\pa_y p_n
\end{pmatrix}
=  f_n  \,, \qquad  y>0\,,\\
& i\tilde n u_{n,1} + \pa_y u_{n,2} =0\,, \qquad y>0\,,\\
& u_n |_{y=0} = 0\,.
\end{cases}
\end{align}
We remind that   $u_n = u_n(y)$ is the $n$-th Fourier coefficient of the velocity, and $\tilde n=n/\kappa$. Note that $|\pm \tilde 1|$ is large if $\kappa$ is small. The zero mode does not raise any difficulty, and is estimated in Section \ref{sec.zero}. The difficult part is the derivation of good bounds for $\tilde n \neq 0$, see Theorem \ref{thm.linear}. For $\kappa$ small enough, we can always ensure that $|\tilde n| \gg 1$ for all $n$. Nevertheless, as $\nu \ll 1$, the tangential diffusion term $- \nu \tilde  n^2 u_n$ in the first line of \eqref{eq.linearized} is in general far too small to control the stretching term $u_{n,2} \pa_y U_s^\nu = O(\frac{1}{\sqrt{\nu}} |u_n|)$. 

\mspace
To obtain good bounds, we  distinguish between two regimes: $|\tilde n| \ll  \nu^{-3/4}$ and $|\tilde n| \gtrsim \nu^{-3/4}$. The regime $|\tilde n| \gtrsim \nu^{-3/4}$ is  handled in  Section \ref{sec.high}, through a direct analysis of system \eqref{eq.linearized}. The subcase $|\tilde n| \gg \nu^{-3/4}$ can be treated by simple energy estimates, {\it cf} item i) in Proposition \ref{prop.high}:  the diffusion term is enough to control stretching by the boundary layer velocity.  However, the regime where $|\tilde n| \sim \nu^{-3/4}$ is much more difficult, and requires new estimates. Such estimates,  in which the convection term is involved, are actually valid in the wider regime $ \nu^{-1/2 }\ll |\tilde n| \lesssim  \nu^{-3/4}$,  {\it cf} item ii) in Proposition \ref{prop.high}.   

\mspace
Stability in the  regime  $|\tilde n| \ll  \nu^{-3/4}$ is the most delicate to obtain. It is deduced from a careful analysis of the steady Orr-Sommerfeld system \eqref{eq.OS_start}, which is a reformulation  of \eqref{eq.linearized} in terms of the stream function and of the rescaled variable $Y = y/\sqrt{\nu}$. {It reads 
\begin{align*}
\begin{cases}
&  OS[\phi] :=  U_s (\pa_Y^2-\alpha^2) \phi  - U_s'' \phi   + i \eps (\pa_Y^2 - \alpha^2)^2  \phi
= - f_{2} - \frac{i}{\alpha} \pa_Y f_{1}  \,, \qquad  Y>0\,,\\
& \phi |_{Y=0} = \pa_Y \phi |_{Y=0} = 0\,.
\end{cases}
\end{align*}
where  parameters $\alpha$ and $\eps$ are related to the tangential frequency $\tilde n$ and the viscosity $\nu$. In short, the regime  $|\tilde n| \ll  \nu^{-3/4}$ corresponds to the case $\eps^{1/3} \alpha \ll 1$.} 

\mspace
The point is that we are not able to get direct estimates on this system. Instead, we construct the solution through an iterative process, reminiscent of splitting methods in numerical analysis.  More precisely, one main idea is to construct a solution to the Orr-Sommerfeld equation in the form of a series, where successive corrections solve alternatively: \\
\bul inviscid approximations of the equation, based on the so-called {\em Rayleigh equation}. \\
\bul viscous approximations of the equation, based on the so-called {\em Airy equation}. \\
This idea of a splitting method was already present  in our Gevrey stability study of the unsteady case \cite{GeMaMa}, and found its origin in article \cite{GGN2014}: the construction of an unstable eigenmode for the linearized Navier-Stokes equations was performed with a similar iteration, although more explicit and  specific to a narrower regime of parameters. Here and in \cite{GeMaMa}, the convergence of the iteration is rather shown by energy arguments, and adapted to the whole range $|\tilde n|  \ll \nu^{-3/4}$. But in the steady setting considered here, we must rely on estimates that are totally different from  the ones in \cite{GeMaMa}, in order to reach Sobolev stability. {Moreover, the implementation of the splitting method is  different.}

\mspace
{
The inviscid estimates are established in Section \ref{sec.Rayleigh}: they are mostly about the equation $Ray[\varphi] = f$, where the Rayleigh operator  $Ray := U_s (\pa^2_Y - \alpha^2) - U''_s$ corresponds to neglecting the diffusion in the Orr-Sommerfeld operator. 
Due to the degeneracy of $U_s$ at $Y=0$, the derivation of good bounds is uneasy, and provided in Proposition \ref{prop.Ray}. The most difficult case is when $\alpha \ll 1$: indeed taking $\alpha \rightarrow 0$ in the Rayleigh equation  yields a singular perturbation problem, {\it cf} Remark \ref{rem.prop.Ray}. Nevertheless, as shown in the estimates of Proposition \ref{prop.Ray}, a crucial point is that the singularity shows up only when the source term $f$ has {nonzero} average in $Y$.} 

\mspace
{After the inviscid analysis of Section \ref{sec.Rayleigh}, Section \ref{sec.Airy} collects various estimates on viscous equations of Airy type: they all involve the operator  $Airy := U_s + i \eps (\pa^2_Y - \alpha^2)$. Note that the Rayleigh and Airy operators are naturally involved within the full Orr-Sommerfeld operator through the identities (to be detailed later):
\begin{align*}
OS[\phi] & = Ray[\phi] + i \eps (\pa^2_Y - \alpha^2)^2 \phi = Ray[\phi] + i \eps (\pa^2_Y - \alpha^2) [\frac{1}{U_s} Ray[\phi] + \frac{U''_s}{U_s} \phi], \\
OS[\phi] & = (\pa_Y^2-\alpha^2 ) Airy[\phi] - 2 \pa_Y (U_s' \phi), \\
OS[\phi] & = Airy \big [\frac{1}{U_s} Ray [{\phi}] \big ] + i \eps  (\pa_Y^2-\alpha^2) \frac{U_s''}{U_s} {\phi} 
\end{align*}
These identities are at the basis of the splitting method alluded to above, which provides a solution to the Orr-Sommerfeld equation under the form of a converging series. This construction, called Rayleigh-Airy iteration, is described in Subsection \ref{subsec.iteration}. In this process,  a special attention is paid to the possible singularity generated by the Rayleigh equation when $\alpha \ll 1$, which could forbid the convergence of the series. In short, one has to ensure that each "Rayleigh step" is performed with a zero average source term. This major difficulty is new compared to the unsteady analysis in \cite{GeMaMa}, and leads to a different iteration.}

\mspace
{Moreover, the Rayleigh-Airy iteration is not enough to conclude: it provides a solution to the Orr-Sommerfeld equation with a given source term, but this solution does not satisfy both Dirichlet and Neumann conditions. Only the Dirichlet condition is maintained through the iteration. One must then combine it with two solutions of the homogeneous Orr-Sommerfeld equation (with an inhomogeneous Dirichlet condition $\phi\vert_{Y=0} = 1$). These special solutions $\phi_{slow}$ and $\phi_{fast}$ are called slow and fast modes, following a terminology of \cite{GGN2014}.  They are built in Subsection \ref{subsec.slow} (see also the preliminary results given in Corollay \ref{cor.prop.slow.Ray}  and Proposition \ref{prop.slow.Ray'})  and Subsection \ref{subsec.fast} respectively. Let us stress that the construction of the slow and fast modes can not be performed in an abstract way, like for the solution coming from  the Rayleigh-Airy iteration. They are rather obtained starting from an explicit approximation (of inviscid type for the slow mode, of viscous "boundary layer type" for the fast mode), which fulfills the inhomogeneous condition, but solves approximately the equation.  One can then add a corrector to get an exact solution, notably making use of the Rayleigh-Airy iteration developped earlier. The proof of the linear stability result in the regime $|\tilde n| \ll \nu^{-3/4}$ is then achieved in Paragraphs \ref{subsec.proof.thm.OS} and \ref{subsec.proof.linear}.}

\mspace
Once the linear estimates of Theorem \ref{thm.linear} are shown, the proof of our main Theorem \ref{thm.main} can be completed classically by a fixed point argument. This is done in Section \ref{sec.nonlinear}. 

\section{Preliminaries}\label{sec.pre}
We collect here a few estimates to appear in the next sections.
Assume that $U_s$ satisfies \eqref{assume.1}-\eqref{assume.3}, and let $Y_0\in (0,1]$ be a number such that $\pa_Y U_s>0$ holds on $0\leq Y \leq 4 Y_0$.

\mspace

\begin{proposition}\label{prop.pre}   {\rm (1)} The following inequalities hold: 
\begin{align*}
& \| \frac{Y \pa_Y U_s}{U_s} \|_{L^2_Y} +  \|\frac{Y \pa_Y^2U_s}{U_s}\|_{L^\infty_Y}  <\infty\,,\\
& | \pa_y^2 U_s^\nu (y) | \leq \frac{C}{\nu^\frac12} \pa_y U_s^\nu (y) \,, \qquad  0\leq y\leq 2 Y_0 \nu^\frac12\,.
\end{align*}

\noindent {\rm (2)} Set $\displaystyle G_s(Y) = U_s (Y) \int_Y^{Y_0} \frac{1}{U_s^2} \, d Y_1$. Then 
\begin{align*}
\pa_Y \big (U_s^2 \pa_Y ( \frac{G_s}{U_s}) \big ) =0 \quad {\rm for}~~Y>0\,, \qquad G_s (0) = \frac{1}{\pa_Y U_s(0)}\,,
\end{align*}
and 
\begin{align*}
\| \frac{G_s}{1+Y} \|_{L^\infty_Y} + \| \frac{\pa_Y G_s}{\log Y} \|_{L_Y^\infty (\{0 < Y \leq 1/2\})} + \| \pa_Y G_s \|_{L^\infty_Y  (\{ Y\geq 1/2\})} <\infty\,.
\end{align*}

\end{proposition}

\begin{proof} The results of (1) are straightforward, so we only give the proof of (2).
It is easy to see that $G_s$ satisfies the equation as in the claim.
Moreover, we see
\begin{align*}
G_s (Y) =  - U_s (Y) \int_Y^{Y_0} (\frac{1}{U_s})' \frac{d Y_1}{U_s'} = \frac{1}{U_s'(Y)} - \frac{U_s(Y)}{U_s(Y_0)U_s'(Y_0)} - U_s (Y) \int_Y^{Y_0} \frac{1}{U_s}  \frac{U_s''}{(U_s')^2} \, d Y_1\,.
\end{align*}
Thus, we have $G_s (0) = \frac{1}{U_s'(0)}$, and we also have 
\begin{align*}
\pa_Y G_s (Y) = - \frac{U_s'(Y)}{U_s(Y_0)U_s'(Y_0)} - U_s'(Y) \int_Y^{Y_0} \frac{1}{U_s}  \frac{U_s''}{(U_s')^2} \, d Y_1\,.
\end{align*}
Thus, the fact $U_s\sim U_s'(0) Y$ for $0<Y\ll 1$ implies $\| \frac{\pa_Y G_s}{\log Y} \|_{L^\infty_Y (\{0<Y<1/2\})}<\infty$. 
On the other hand, since $U_s\sim U_E$ for $Y\gg 1$ and $U_s>0$ in $Y>0$, it is easy to show that $\| \frac{G_s}{1+Y} \|_{L^\infty_Y} + \| \pa_Y G_s \|_{L^\infty_Y (\{Y\geq 1/2\})} <\infty$. The proof is complete.
\end{proof}

\mspace

\begin{proposition}\label{prop.pre.op} {\rm (1)} Let $\sigma[\cdot]$ be the linear operator defined by
\begin{align*}
\displaystyle \sigma [f] (Y) = \int_Y^\infty f \, d Y_1\,, \qquad f\in C_0^\infty (\overline{\R_+})\,.
\end{align*}
Then for $1\leq p<\infty$ and $k=0,1,\ldots$, 
\begin{align*}
\| Y^k \sigma [f] \|_{L^p_Y} \leq C_p \| Y^{k+1} f \|_{L^p_Y} \,.
\end{align*}

\noindent {\rm (2)} Let $L[\cdot]$ be the linear operator defined by 
\begin{align*}
\displaystyle L[f] (Y) = U_s (Y) \int_Y^\infty \frac{f}{U_s^2} \, d Y_1\,, \qquad f\in C_0^\infty (\overline{\R_+})\,.
\end{align*}
Then for $1<p<\infty$ { and $k=0,1,\ldots$}, 
\begin{align*}
{ \| Y^k L[f] \|_{L^p_Y} }& \leq { C \| Y^k (1+Y) f \|_{L^p_Y} \,,}\\
\| \pa_Y L[f] \|_{L^2_Y} & \leq C \Big ( \| f\|_{L^1_Y} + \| f\|_{L^2_Y} + \| \pa_Y f \|_{L^2_Y} \Big )\,.
\end{align*}
\end{proposition}

\begin{remark}{\rm In the estimate of $\pa_Y L[f]$ we do not need the condition $f|_{Y=0}=0$. Thus, the estimate is valid for any $f\in L^1(\R_+) \cap H^1(\R_+)$.
}
\end{remark}

\begin{proof} (1) Since $Y^k |\sigma[f] (Y)| \leq \sigma [|f_k|](Y)$ with $f_k (Y) = Y^k f (Y)$ pointwisely, 
it suffices to consider the case $k=0$. By changing the order of the integral we see $\|\sigma [f] \|_{L^1_Y} \leq \| Y f \|_{L^1_Y}$, while we have for $1<p<\infty$, $|\sigma [f] (Y)|\leq C_p Y^{-1/p} \| {Y} f\|_{L^p_Y}$, which implies $\|\sigma [f] \|_{L^{p,\infty}_Y} \leq C_p \| {Y} f\|_{L^p_Y}$. Thus the Marcinkiewicz interpolation theorem gives $\| \sigma[f] \|_{L^p_Y} \leq C \| {Y} f\|_{L^p_Y}$. The proof is complete.

\noindent (2) { It suffices to consider the case $k=0$.}
Since $U_s\sim U_s'(0) Y$ for $0<Y\ll 1$ and $U_s\sim U_E$ for $Y\gg 1$ we observe that 
\begin{align*}
|L[f](Y)| \leq \frac{C_p}{Y^\frac1p} \| (1+Y) f\|_{L^p_Y} \,, \qquad Y>0\,.
\end{align*}
Thus $\|L[f] \|_{L^{p,\infty}_Y} \leq C_p \| (1+Y) f\|_{L^p_Y}$ holds for $1<p<\infty$, and hence, the Marcinkiewicz interpolation theorem yields $\|L[f]\|_{L^p_Y} \leq C\| {(1+Y)} f\|_{L^p_Y}$ for $1<p<\infty$.
Next we see 
\begin{align*}
\pa_Y L[f] (Y) = U_s'(Y) \int_Y^\infty \frac{f}{U_s^2} \, d Y_1 - \frac{f}{U_s} =  U_s'(Y) \int_Y^{Y_0} \frac{f}{U_s^2} \, d Y_1 + U_s'(Y) \int_{Y_0}^\infty  \frac{f}{U_s^2} \, d Y_1- \frac{f}{U_s} \,.
\end{align*}
From this expression it is easy to see that 
\begin{align*}
\|\pa_Y L[f] \|_{L^2(\{Y\geq Y_0\})} \leq C  ( \| f \|_{L^1_Y} + \| f\|_{L^2_Y}  )\,,
\end{align*}
and also 
\begin{align*}
\| U_s'(Y) \int_{Y_0}^\infty  \frac{f}{U_s^2} \, d Y_1 \|_{L^2_Y} \leq C \| f\|_{L^1_Y} \,.
\end{align*}
Now it suffices to consider the estimate of $U_s'(Y) \int_Y^{Y_0} \frac{f}{U_s^2} \, d Y_1-\frac{f}{U_s}$ in $0<Y\leq Y_0$. We see 
\begin{align*}
U_s'(Y) \int_Y^{Y_0} \frac{f}{U_s^2} \, d Y_1-\frac{f}{U_s} & = - U_s'(Y) \int_Y^{Y_0} (\frac{1}{U_s})' \frac{f}{U_s'} \, d Y_1-\frac{f}{U_s} \\
& = - \frac{U_s'(Y) f(Y_0)}{U_s (Y_0) U_s'(Y_0)}  + U_s'(Y) \int_Y^{Y_0} \frac{1}{U_s} (\frac{f}{U_s'})' \, d Y_1\,.
\end{align*}
It is clear that 
\begin{align*}
\| \frac{U_s'(Y) f(Y_0)}{U_s (Y_0) U_s'(Y_0)}  \|_{L^2(\{0<Y\leq Y_0\})} \leq C (\| \pa_Y f\|_{L^2_Y} + \| f \|_{L^2_Y} )
\end{align*}
by applying the embedding inequality $\| f\|_{L^\infty_Y}\leq C\|\pa_Y f\|_{L^2_Y}^\frac12 \| f\|_{L^2_Y}^\frac12$ (note that the condition $f|_{Y=0}=0$ is not required).
Let us consider the estimate of $\| U_s'(Y) \int_Y^{Y_0} \frac{1}{U_s} (\frac{f}{U_s'})'\, d Y_1\|_{L^2(\{0<Y\leq Y_0\})}$.
We see 
\begin{align*}
|U_s'(Y) \int_Y^{Y_0} \frac{1}{U_s} (\frac{f}{U_s'})'\, d Y_1| \leq \frac{C}{Y^\frac1p} \| (\frac{f}{U_s'})' \|_{L^p(\{0<Y\leq Y_0\})} \,, \qquad 1<p<\infty\,, \quad 0<Y\leq Y_0\,.
\end{align*}
Hence the Marcinkiewicz interpolation theorem yields 
\begin{align*}
\| U_s'(Y) \int_Y^{Y_0} \frac{1}{U_s} (\frac{f}{U_s'})' \, d Y_1 \|_{L^2 (\{0<Y\leq Y_0\})} \leq C\| (\frac{f}{U_s'})' \|_{L^2(\{0<Y\leq Y_0\})} \leq C (\|\pa_Y f \|_{L^2_Y} + \| f\|_{L^2_Y})\,.
\end{align*}
The proof is complete.
\end{proof}

\mspace

\noindent
We end this short section with an interpolation  inequality that will be applied several times. 
\begin{proposition} There exists $C > 0$ such that 
\begin{equation}\label{proof.prop.Airy.5}
\| g \|^2_{L^2_Y} \leq C \| \sqrt{U_s} g \|_{L^2_Y}^\frac43 \| \pa_Y g \|_{L^2_Y}^\frac23 + C \| \sqrt{U_s} g \|_{L^2_Y}^2\,, \qquad g = g(Y) \in H^1 (\R_+)\,.
\end{equation}
\end{proposition}
\begin{proof}
This bound is an easy  consequence of the inequality 
\begin{align*}
 \| g \|^2_{L^2_Y} \le C \| \sqrt{Y} g \|_{L^2_Y} \| g \|_{L^\infty_Y} \leq C \| \sqrt{Y} g \|_{L^2_Y} \| \pa_Y g \|_{L^2_Y}^\frac12 \| g \|_{L^2_Y}^\frac12 \,, 
\end{align*}
which implies 
\begin{align*}
 \| g \|_{L^2_Y} \le C \| \sqrt{Y} g \|_{L^2_Y}^\frac23 \| \pa_Y g \|_{L^2_Y}^\frac13\,,
\end{align*}
and of the properties of $U_s$: $U_s (Y) \sim U_s'(0) Y$ around $Y=0$, $U_s (Y) \sim U_E>0$ around $Y=\infty$,  together with the assumption $U_s>0$ for any $Y>0$. The first inequality is obtained by writing that for all $A > 0$
\begin{align*}
 \| g \|_{L^2_Y}^2 \le \int_0^A |g|^2 + \int_A^{+\infty} |g|^2 \le A \| g\|_{L^\infty_Y}^2  \: + \: \frac{1}{A} \| \sqrt{Y} g \|_{L^2_Y}^2 
 \end{align*}
 and optimizing in $A$.  The proof is complete.
 \end{proof}
 \noindent
Note that an interpolation inequality of this type is also used in the work  \cite{GaHiMa} to construct a steady Navier-Stokes flow around a rotating disk.

\section{Linear result for the zero mode}\label{sec.zero}

When $n=0$ the linearized problem \eqref{eq.linearized} is reduced to a simple ODE: indeed, $u_{0,2}=0$ and  $\nu \pa_y^2 u_{0,1} = f_{0,1}$ with $u_{0,1}|_{y=0}=0$. Then $u_{0,1}(y) = -\frac{1}{\nu} \int_0^y \int_{y'}^\infty f_{0,1} (y'') \, d y'' d y'$.
The pressure $p_0$ is given by $p_0 (y) = -\int_y^\infty f_{0,2} (y') d y'$. Hence we have 
\begin{theorem}\label{thm.zero} Let $f_0 \in L^1 (\R_+^2)^2$ and $f_{0,1}=\pa_y F_{0,1}$ with $F_{0,1}\in L^1 (\R_+)\cap L^2 (\R_+)$. Then there exists a unique solution $u_0=(u_{0,1},0)^\top$ to \eqref{eq.linearized} with $\tilde n=0$ such that 
\begin{align}
\| u_{0,1} \|_{L^\infty} & \leq \frac{1}{\nu} \| F_{0,1} \|_{L^1}\,,\\
\| \pa_y u_{0,1} \|_{L^2} & = \frac{1}{\nu} \| F_{0,1} \|_{L^2}\,.
\end{align}
We also have $\displaystyle \lim_{y\rightarrow \infty} u_{0,1} = \frac{1}{\nu} \int_0^\infty F_{0,1} \, d y$.
\end{theorem}

\section{Linear result for the non-zero modes}\label{sec.linear}

In this section we state the main result for the linearized problem \eqref{eq.linearized} when $\tilde n\ne 0$.

\begin{theorem}\label{thm.linear} There exist positive numbers $\kappa_0$, $\nu_0$, and $\delta_*$ such that the following statement holds for any $0<\kappa\leq \kappa_0$, $0<\nu\leq \nu_0$, and $\tilde n\ne 0$. For any $f_n\in L^2 (\R_+)^2$ there exists a unique solution $u_n\in H^2(\R_+)^2\cap H^1_0 (\R_+)^2$ to \eqref{eq.linearized} satisfying the estimates stated below:

\noindent {\rm (i)} if $0<|\tilde n| \leq  \nu^{-\frac37}$ then
\begin{align}
\| u_n \|_{L^2} & \leq 
\begin{cases}
& \displaystyle \frac{C}{|\tilde n|^\frac12} \| f_n \|_{L^2}\,, \qquad 0<|\tilde n|\leq \nu^{-\frac38}\\
& \displaystyle \frac{C}{|\tilde n|^\frac{11}{6}\nu^\frac12} \| f_n \|_{L^2}\,, \qquad \nu^{-\frac38}\leq |\tilde n|\leq \nu^{-\frac37}\,,
\end{cases} \label{est.thm.linear.1'}\\
\| u_n \|_{L^\infty} & \leq \frac{C}{|\tilde n|^\frac12\nu^\frac14} \| f_n \|_{L^2}\,, \label{est.thm.linear.2'}\\
\| \pa_y u_n \|_{L^2} + |\tilde n | \| u_n \|_{L^2} & \leq \frac{C}{|\tilde n|^\frac13 \nu^\frac12} \| f_n \|_{L^2}\,.\label{est.thm.linear.3'}
\end{align}

\noindent {\rm (ii)} if $\nu^{-\frac37}\leq |\tilde n| \leq \delta_* \nu^{-\frac34}$ then
\begin{align}
\| u_n \|_{L^2} & \leq \frac{C}{|\tilde n|^\frac23} \| f_n \|_{L^2}\,, \label{est.thm.linear.1}\\
\| \pa_y u_n \|_{L^2} + |\tilde n | \| u_n \|_{L^2} & \leq \frac{C}{|\tilde n|^\frac13 \nu^\frac12} \| f_n \|_{L^2}\,.\label{est.thm.linear.2}
\end{align}

\noindent {\rm (iii)} if $|\tilde n| \geq \delta_* \nu^{-\frac34}$ then 
\begin{align}
\| u_n \|_{L^2} & \leq \frac{C}{|\tilde n|^2 \nu} \| f_n \|_{L^2} \,,\label{est.thm.linear.5}\\
\| \pa_y u_n \|_{L^2} + |\tilde n| \| u_n \|_{L^2} & \leq \frac{C}{|\tilde n| \nu} \| f_n \|_{L^2}\,.\label{est.thm.linear.6}
\end{align}

\end{theorem}
\begin{remark}
{\rm (1)  In Theorem \ref{thm.linear}, the associated pressure $p_n$ belongs to $H^1(\R_+)$.

\noindent (2) Estimate \eqref{est.thm.linear.2'} is not a consequence of the interpolation between \eqref{est.thm.linear.1'} and \eqref{est.thm.linear.3'}.
On the other hand, by the interpolation inequality $\|u_n \|_{L^\infty} \leq C \| \pa_y u_n \|_{L^2}^\frac12 \| u_n \|_{L^2}^\frac12$, we have from \eqref{est.thm.linear.1}, \eqref{est.thm.linear.2}, \eqref{est.thm.linear.5}, and \eqref{est.thm.linear.6}, 
\begin{align}
\| u_n\|_{L^\infty} \leq 
\begin{cases}
& \displaystyle \frac{C}{|\tilde n|^\frac12 \nu^\frac14} \| f_n \|_{L^2}\,,\qquad \nu^{-\frac37} \leq |\tilde n| \leq \delta_* \nu^{-\frac34} \,,\\
& \displaystyle \frac{C}{|\tilde n|^\frac32 \nu} \| f_n \|_{L^2}\,, \qquad |\tilde n|\geq \delta_*\nu^{-\frac34}\,.
\end{cases}
\end{align}
}
\end{remark}

\noindent
The proof of Theorem \ref{thm.linear} consists of several steps, and is given in Sections  \ref{sec.Rayleigh} - \ref{sec.high} below. 
The core part of the proof is the study of the Orr-Sommerfeld equation for the streamfunction.
For the moment we assume that $f_n$ is smooth enough, say, $f_n\in H^1 (\R_+)^2$.
The Orr-Sommerfeld equation is deduced from the equation for the vorticity $\omega_n = i\tilde n u_{n,2}-\pa_y u_{n,1}$:
\begin{align}\label{eq.vorticity}
i \tilde n U_s^\nu \omega_n - u_{n,2} \pa_y^2 U_s^\nu - \nu (\pa_y^2 - \tilde n^2) \omega_n = i \tilde n f_{n,2}-\pa_y f_{n,1}\,, \qquad y>0\,.
\end{align}
Let us introduce the streamfunction $\phi_n$ as the solution to the Poisson equation 
\begin{align*}
-(\pa_y^2-\tilde n^2) \phi_n = \omega_n\,, \qquad \phi_n|_{y=0} =0\,.
\end{align*}
The $n$th mode of the velocity $u_n$ is recovered from the formula 
\begin{align} \label{def.velocity_n}
u_{n,1} = \pa_y \phi_n\,, \qquad u_{n,2}  = - i \tilde n \phi_n\,.
\end{align}
Taking into account no-slip boundary condition on $u_{n,1}${,} we obtain the fourth order ODE
\begin{align*}
\begin{cases}
&- i \tilde n U_s^\nu(\pa_y^2-\tilde n^2) \phi_n  +i \tilde n \pa_y^2 U_s^\nu \phi_n   + \nu (\pa_y^2 - \tilde  n^2)^2  \phi_n
= i \tilde n f_{n,2}-\pa_y f_{n,1}  \,, \qquad  y>0\,,\\
& \phi_n |_{y=0} = \pa_y \phi_n |_{y=0} = 0\,.
\end{cases}
\end{align*}
One can check that this fourth order ODE is equivalent to \eqref{eq.linearized}, and in particular, if $\phi_n\in H^4(\R_+)$ is a solution to this ODE then the velocity $u_n$ defined by \eqref{def.velocity_n} solves \eqref{eq.linearized} with a suitable pressure $p_n\in H^1 (\R_+)$.
Next we introduce the rescaled variable $Y=y/\sqrt{\nu}$ and set 
\begin{align}\label{def.rescale}
\sqrt{\nu} \phi(Y)=\phi_n (y)\,, \qquad \frac{1}{\sqrt{\nu}} f (Y)=f_n(y) 
\end{align}
for simplicity. The rescaled unknown then satisfies
\begin{align*}
- i \tilde n U_s (\pa_Y^2-\alpha^2) \phi  +i \tilde n U_s'' \phi   +  (\pa_Y^2 - \alpha^2)^2  \phi
= i \tilde n  f_{2}- \frac{1}{\sqrt{\nu}} \pa_Y f_{1}  \,, \qquad  Y>0\,,
\end{align*}
where $U_s''=\pa_Y^2 U_s$ and we have set 
\begin{align*}
\alpha = \tilde n \sqrt{\nu}\,.
\end{align*}
By dividing by $-i\tilde n$ of the above equation for $\phi$ and by setting $\eps=1/\tilde n$,  we have arrived at the Orr-Sommerfeld equation
\begin{align}\label{eq.OS_start}
\begin{cases}
& U_s (\pa_Y^2-\alpha^2) \phi  - U_s'' \phi   + i \eps (\pa_Y^2 - \alpha^2)^2  \phi
= - f_{2} - \frac{i}{\alpha} \pa_Y f_{1}  \,, \qquad  Y>0\,,\\
& \phi |_{Y=0} = \pa_Y \phi |_{Y=0} = 0\,.
\end{cases}
\end{align}
We note that the constant $\eps=1/\tilde n=\kappa/n$ is small if $|n|\geq 1$ and $\kappa>0$ is small.
For simplicity of notations we set 
\begin{align*}
\| f\| = \| f\|_{L^2_Y} = \big ( \int_0^\infty | f(Y)|^2 d Y \big )^\frac12\,.
\end{align*}

\section{Rayleigh equation}\label{sec.Rayleigh}

In this section we consider the Rayleigh equation in the rescaled variable, where
the Rayleigh operator is defined as $Ray[\varphi] = U_s (\pa_Y^2-\alpha^2) \varphi - U_s'' \varphi$.
Without loss of generality we may assume $\alpha>0$. The system under study is
\begin{align}\label{eq.Ray}
\begin{cases}
& Ray[\varphi] = f \,, \qquad Y>0\,,\\
& \varphi|_{Y=0} = 0\,.
\end{cases}
\end{align}

\begin{proposition}\label{prop.Ray} Let $f/U_s\in L^2 (\R_+)$. Then there exists a unique solution $\varphi \in H^2(\R_+)\cap H^1_0 (\R_+)$ to \eqref{eq.Ray} such that 

\noindent {\rm (i)} when $\alpha \geq 1$, 
\begin{align}
\| \pa_Y\varphi \| +\alpha \| \varphi\| & \leq C \min\big \{ \| \frac{Y}{U_s} f \|, \, \frac{1}{\alpha} \| \frac{f}{U_s} \| \big \} \,, \label{est.prop.Ray.1}\\
\| (\pa_Y^2-\alpha^2) \varphi \| & \leq C  \min \big \{ \| \frac{Y}{U_s} f \|, \, \frac{1}{\alpha} \| \frac{f}{U_s} \| \big \} + \| \frac{f}{U_s} \|\,.\label{est.prop.Ray.2}
\end{align}
\noindent {\rm (ii)} when $0<\alpha\leq 1$, if $(1+Y) \sigma [f]\in L^2 (\R_+)$ with $\sigma [f] (Y)=\int_Y^\infty f \, d Y_1$ in addition, 
\begin{align}
\alpha \| \varphi \| & \leq C \alpha \| (1+Y) \sigma [f] \|  + \frac{C}{\alpha^\frac12} |\int_0^\infty f \, d Y | \,, \label{est.prop.Ray.1'}\\
\| \pa_Y \varphi \| & \leq C \Big ( \| (1+Y)  \sigma [f] \|  + \| f \|\Big )  + \frac{C}{\alpha} |\int_0^\infty f \, d Y | \,, \label{est.prop.Ray.2'}\\
\| (\pa_Y^2-\alpha^2) \varphi \| & \leq C \Big ( \| (1+Y) \sigma [ f] \|  + \| \frac{f}{U_s} \| \Big ) +  \frac{C}{\alpha} |\int_0^\infty f \, d Y |  \,.\label{est.prop.Ray.3'}
\end{align}
\end{proposition}

\begin{remark}\label{rem.prop.Ray}{\rm (1) {We obtain solutions even for $0 < \alpha < 1$, which is non-trivial. Note that the existence and uniqueness of $H^2$ solutions  requires the  condition $f/U_s\in L^2 (\R_+)$. Nevertheless, it is seen from the proof that the existence and uniqueness of a weak solution in $H^1_0(\R_+)$ is valid under the milder condition $f\in L^2 (\R_+)$. In such a case, the only estimates that are retained are}  
{
\noindent {\rm (i)} when $\alpha \geq 1$, 
\begin{align*}
\| \pa_Y\varphi \| +\alpha \| \varphi\| & \leq C \| \frac{Y}{U_s} f \|,  \,
\end{align*}
\noindent {\rm (ii)} when $0<\alpha\leq 1$, 
\begin{align*}
\alpha \| \varphi \| & \leq C \alpha \| (1+Y) \sigma [f] \|  + \frac{C}{\alpha^\frac12} |\int_0^\infty f \, d Y | \,, \\
\| \pa_Y \varphi \| & \leq C \Big ( \| (1+Y)  \sigma [f] \|  + \| f \|\Big )  + \frac{C}{\alpha} |\int_0^\infty f \, d Y |.
\end{align*}
Moreover}, the solution is locally $H^2$ for $Y>0$, and satisfies 
\begin{align}\label{est.rem.prop.Ray}
\| U_s (\pa_Y^2-\alpha^2) \varphi \| \leq C \| \pa_Y \varphi \| + \| f\|\,,
\end{align}
which is verified from the equation and the Hardy inequality $\| U_s'' \varphi \| \leq C \| \pa_Y \varphi\|$,  in virtue of the sufficient decay of $U_s''$ {and $\varphi|_{Y=0}=0$}.

\noindent (2) {The singularity in $\alpha$, seen  in  \eqref{est.prop.Ray.1'}-\eqref{est.prop.Ray.2'}-\eqref{est.prop.Ray.3'} when $0<\alpha\ll 1$, can not be avoided in general. It is due to the fact that $\alpha \rightarrow 0$ is a singular limit:  indeed the formal limit system
$$  U_s \pa_Y^2 \varphi - U''_s \varphi = f$$
has in general no solution satisfying $\varphi\vert_{Y=0} = 0$. However, as we will show below, it has a solution under the additional condition that $f$ has zero average. This explains that the  singular factor $\alpha^{-\frac{1}{2}}$ or $\alpha^{-1}$ is only in front of   $|\int_0^\infty f \, d Y| $. This will be used crucially in our analysis.} 

\noindent (3) Estimate \eqref{est.prop.Ray.1'} is optimal in view of local regularity; roughly speaking, the $L^2$ norm of $\varphi$ is estimated in terms of $H^{-1}$ norm of $f$. We note that the Marcinkiewicz interpolation theorem yields the bound $\| (1+Y) \sigma [f] \| \leq C \| (1+Y) Y f \|$ as stated in Proposition \ref{prop.pre.op} (1).
}
\end{remark}

The proof of Proposition \ref{prop.Ray}  is a  consequence of the next lemma.
We consider the problem 
\begin{align}\label{eq.Ray''}
\begin{cases}
&  (\pa_Y^2 - \alpha^2 ) \varphi - \frac{U_s''}{U_s} \varphi = h  \,, \qquad Y>0\,,\\
& \varphi|_{Y=0} = 0\,.
\end{cases}
\end{align}

\begin{lemma}\label{lem.Ray} For any $h\in L^2 (\R_+)$ there exists a unique solution $\varphi \in H^2(\R_+) \cap H^1_0 (\R_+)$ and satisfies 

\noindent {\rm (i)} when $\alpha \geq 1$,
\begin{align}
\| \pa_Y \varphi \| + \alpha \| \varphi \| & \leq C \min\big \{ \| Y h \|, ~ \frac{1}{\alpha} \| h \| \big \}\,, \label{est.lem.Ray.1}\\
\| (\pa_Y^2-\alpha^2) \varphi \| & \leq C \min\big \{ \| Y h \|, ~ \frac{1}{\alpha} \| h \|\big \}   + \| h \|\,,\label{est.lem.Ray.2}
\end{align}
\noindent {\rm (ii)} when $0<\alpha\leq 1$, if $(1+Y) \sigma [U_s h] \in L^2 (\R_+)$with $\sigma [U_s h](Y)=\int_Y^\infty U_s h \, d Y_1$ in addition, 
\begin{align}
\alpha \| \varphi \| & \leq C \alpha \| (1+Y) \sigma [U_s h] \|  + \frac{C}{\alpha^\frac12} \big | \int_0^\infty U_s h \, d Y \big |  \,, \label{est.lem.Ray.1'}\\
\| \pa_Y \varphi \| & \leq C \Big ( \| (1+Y) \sigma [U_s h] \| +  \| U_s h \| \Big ) + \frac{C}{\alpha}  \big | \int_0^\infty U_s h \, d Y \big |  \,, \label{est.lem.Ray.2'}\\
\| (\pa_Y^2-\alpha^2) \varphi \| & \leq C \Big ( \| (1+Y) \sigma [U_s h] \|  + \| h \| \Big ) + \frac{C}{\alpha} \big | \int_0^\infty U_s h \, d Y \big |\,.\label{est.lem.Ray.3'}
\end{align}
\end{lemma}

\begin{proof} Let $T$ be the operator $T=\pa_Y^2-\alpha^2 - U_s''/U_s$, which is realized in $L^2 (\R_+)$ with the domain $H^2 (\R_+)\cap H^1_0 (\R_+)$. Note that $T$ is relatively compact to $\pa_Y^2-\alpha^2$ whose domain is taken in the same way. This is because $\displaystyle \lim_{Y\rightarrow \infty} \frac{U_s''}{U_s} =0$ by our assumptions and because the Hardy inequality implies 
\begin{align*}
\| \frac{U_s''}{U_s} \varphi \|\leq \| \frac{YU_s''}{U_s} \|_{L^\infty} \| \frac{\varphi}{Y}\|\leq 2 \| \frac{YU_s''}{U_s} \|_{L^\infty} \|\pa_Y \varphi \|\,, \qquad \varphi \in H^1_0 (\R_+)\,,
\end{align*}
and thus the term $\frac{U_s''}{U_s}\varphi$ is a lower order operator both in view of regularity and spatial decay. Since the spectrum of $\pa_Y^2-\alpha^2$ is contained in $\{\lambda\in \C~|~\Re \lambda\leq -\alpha^2\}$ the spectrum of $T$ outside the set $\{\lambda\in \C~|~\Re \lambda\leq -\alpha^2\}$  consists only of isolated eigenvalues with finite multiplicities.
Therefore, to show the invertiblity of $T$ in $L^2 (\R_+)$ it suffices to show the injectivity of $T$.
{To this end we observe the identity 
\begin{align*}
T\varphi = \frac{1}{U_s} \pa_Y \big ( U_s^2 (\pa_Y (\frac{\varphi}{U_s}) \big ) -\alpha^2 \varphi
\end{align*}
and then, from the computation for $\varphi\in H^2 (\R_+)\cap H^1_0 (\R_+)$,
\begin{align*}
\lim_{\delta\downarrow 0} \int_\delta^\infty \frac{1}{U_s} \pa_Y \big ( U_s^2 \pa_Y (\frac{\varphi}{U_s} ) \big ) \, \overline{\varphi} \,  d Y & =  \lim_{\delta\downarrow 0} \Big ( - \int_\delta^\infty  U_s^2 |\pa_Y (\frac{\varphi}{U_s} )|^2 \,  d Y -  \pa_Y ( \frac{\varphi}{U_s} ) (\delta) U_s (\delta) \overline{\varphi}  (\delta) \Big ) \\
& = - \| U_s \pa_Y (\frac{\varphi}{U_s}) \|^2\,,
\end{align*}
we verify the identity
\begin{align}\label{proof.lem.Ray.5'} 
\| U_s \pa_Y (\frac{\varphi}{U_s} ) \| ^2 + \alpha^2 \| \varphi \|^2 = - \Re \langle h, \varphi \rangle\,.
\end{align}}
Equality \eqref{proof.lem.Ray.5'} implies the injectivity of $T$, and thus, $T$ is invertible in $L^2 (\R_+)$ as explained above. In particular, there exists a unique solution $\varphi\in H^2(\R_+) \cap H^1_0 (\R_+)$ to $T\varphi = h$.

\mspace
Next we observe that the inner product with $\varphi$ in the equation $T\varphi = h$ also provides
\begin{align}\label{proof.lem.Ray.10} 
\| \pa_Y \varphi\|^2 + \alpha^2 \| \varphi\|^2  = - \Re \langle \frac{U_s''}{U_s} \varphi, \varphi\rangle - \Re \langle h, \varphi \rangle \,.
\end{align}
The first term in the right-hand side of \eqref{proof.lem.Ray.10} is estimated as
\begin{align*}
| \langle \frac{U_s''}{U_s}\varphi ,  \varphi \rangle | \leq \| \frac{Y(1+Y)U_s''}{U_s} \|_{L^2}  \| \frac{\varphi}{Y} \| \|\frac{1}{1+Y} \varphi \|_{L^\infty} \leq C \| \pa_Y \varphi\| \| \frac{\varphi}{1+Y} \|_{L^\infty} \,.
\end{align*}
Thus from \eqref{proof.lem.Ray.5'} and \eqref{proof.lem.Ray.10}, we obtain 
\begin{align}\label{proof.lem.Ray.11} 
\begin{split}
\| U_s \pa_Y (\frac{\varphi}{U_s} ) \|^2 + \alpha^2 \| \varphi \|^2 & \leq  |\Re \langle h, \varphi \rangle |\,, \\
\| \pa_Y \varphi \|^2 & \leq C \| \frac{\varphi}{1+Y}  \|_{L^\infty}^2 + C |\Re \langle h, \varphi \rangle |\,.
\end{split}
\end{align}

\mspace
(i) When $\alpha \geq 1$:
In this case estimate \eqref{est.lem.Ray.1} easily follows from \eqref{proof.lem.Ray.11} by applying $\| \frac{\varphi}{1+Y} \|_{L^\infty}^2 \leq C \| \varphi\| \| \pa_Y \varphi \|$.
Finally, the estimate of $(\pa_Y^2-\alpha^2) \varphi$ is obtained from \eqref{eq.Ray''}, for 
\begin{align*}
\| (\pa_Y^2-\alpha^2) \varphi \| \leq \| \frac{U_s''}{U_s} \varphi \| + \| h \| \leq 2 \| \frac{Y U_s''}{U_s} \|_{L^\infty} \| \frac{\varphi}{Y} \| + \| h \| \leq C \| \pa_Y \varphi \| + \| h\| \,.
\end{align*} 
The proof is complete.

\mspace
(ii) When $0<\alpha \leq 1$: This case requires a more delicate analysis.
We decompose $h$ as
\begin{align}\label{proof.lem.Ray.13} 
h = h_1 + h_2\,, \qquad h_2 = (\int_0^\infty U_s h \, d Y) \rho\,, \qquad \int_0^\infty U_s h_1 \, d Y=0\,,
\end{align}
where $\rho\in C_0^\infty (\R_+)$, $\int_0^\infty U_s \rho \, d Y =1$, is taken independently of $\alpha$ and $h$, and we may assume that ${\rm supp}\, \rho\subset \{1\leq Y\leq 2\}$.
According to the decomposition of $h=h_1+h_2$ as above, we decompose $\varphi$ as $\varphi=\varphi_1+ \varphi_2$, where 
\begin{align*}
T\varphi_j = h_j\,, \qquad \varphi_j |_{Y=0} =0\,.
\end{align*}

\noindent {\it Step 1 (Estimate of $\varphi_1$).} 
Let $\varphi_{1,1}$ be the function defined by 
\begin{align}\label{proof.lem.Ray.14} 
\varphi_{1,1} =  U_s \int_Y^\infty \frac{\sigma  [U_s h_1]}{U_s^2}  \, d Y_1 = L \big [\sigma [U_s h_1]\big ]\,,
\end{align}
where $L$ is studied in Proposition \ref{prop.pre.op} (2), and $\varphi_{1,1}$ satisfies $ \frac{1}{U_s} \pa_Y \Big ( U_s^2 \pa_Y (\frac{\varphi_{1,1}}{U_s}) \Big ) = h_1$ in $Y>0$.
Moreover, in virtue of $\int_0^\infty U_s h_1 \, d Y =0$, we can write $\sigma [U_s h_1] (Y_1) = - \int_0^{Y_1} U_s h_1\, d Y_2$, which ensure the boundary condition $\varphi_{1,1}|_{Y=0}=0$ as well.
From Proposition \ref{prop.pre.op} (2) we have the estimate of $\varphi_{1,1}$ as follows. 
\begin{align}\label{proof.lem.Ray.15} 
\| \varphi_{1,1} \|_{L^2}  = \| L\big [\sigma [U_s h_1] \big ] \|  \leq C \| (1+Y) \sigma [U_s h_1] \| \,.
\end{align} 
We look for the solution $\varphi_1$ of the form $\varphi_1 = \varphi_{1,1} + \varphi_{1,2}$, and thus, $\varphi_{1,2}$ is the solution to $T \varphi_{1,2} = \alpha^2 \varphi_{1,1}$ in $Y>0$ and $\varphi_{1,2}|_{Y=0} =0$. From \eqref{proof.lem.Ray.11} we have $\| \varphi_{1,2} \|^2 \leq C  \| \varphi_{1,1} \|^2$,
which gives 
\begin{align}\label{proof.lem.Ray.16} 
\| \varphi_1 \| \leq \| \varphi_{1,1} \| + \| \varphi_{1,2} \| \leq C \| (1+Y) \sigma [U_s h_1] \|\,.
\end{align}
Then, since $T \varphi_1 = h_1$, we have again from \eqref{proof.lem.Ray.11} and $h_1 = -\frac{1}{U_s} \pa_Y \sigma [U_s h_1]$ and $\sigma [U_s h_1]|_{Y=0} = \int_0^\infty U_s h \, d Y =0$,
\begin{align*}
\| U_s \pa_Y (\frac{\varphi_1}{U_s}) \|^2 \leq |\Re \langle h_1, \varphi_1\rangle | \leq \| \frac{\sigma [U_s h_1]}{U_s} \| \, \| U_s \pa_Y (\frac{\varphi_1}{U_s}) \|\,,
\end{align*}
that is, $\| U_s \pa_Y (\frac{\varphi_1}{U_s}) \| \leq \| \frac{\sigma [U_s h_1]}{U_s} \|$,
and thus, \eqref{proof.lem.Ray.11} yields
\begin{align}\label{proof.lem.Ray.17} 
\| \pa_Y \varphi_{1} \|^2 \leq C \| \pa_Y \varphi _{1} \| \| \varphi_{1} \|  + C |\Re \langle h_1, \varphi_{1}\rangle | &\leq C \| \varphi_1\|^2 + C \| \frac{\sigma [U_s h_1]}{U_s} \| \, \| U_s \pa_Y (\frac{\varphi_1}{U_s})  \|  \nonumber \\
& \leq C \| (1+Y) \sigma [U_s h_1] \|^2 + \| \frac{\sigma [U_s h_1]}{U_s}  \|^2 \nonumber \\
& \leq C  \| (1+Y) \sigma [U_s h_1] \|^2 + C \| U_s h_1  \|^2\,.
\end{align}
Here we have used the Hardy-type  inequality in the last line: $\| \frac{\sigma [U_s h_1]}{U_s} \| \leq C \big ( \| \pa_Y \sigma [U_s h_1] \| + \| \sigma [U_s h_1] \| \big )$. 
The $H^2$ estimate is then obtained from the equation and  \eqref{proof.lem.Ray.17} as
\begin{align}\label{proof.lem.Ray.18} 
\| (\pa_Y^2-\alpha^2) \varphi_1 \| \leq \| \frac{U_s''}{U_s} \varphi_1 \| + \| h_1 \| & \leq C \| \pa_Y \varphi_1 \| + \| h_1 \|  \nonumber \\
& \leq C \| (1+Y) \sigma [U_s h_1] \| + C \| U_s h_1 \| + \| h_1 \| \,.
\end{align}

\noindent {\it Step 2 (Estimate of $\varphi_2$).}
Next we consider the estimate of $\varphi_2$, which is the solution to $T \varphi_2 = h_2$ with $\varphi_2|_{Y=0} =0$, where $h_2=(\int_0^\infty U_s h \, d Y) \rho$.
First we set 
\begin{align}\label{proof.lem.Ray.19} 
\begin{split}
\varphi_{2,1,1} & = U_s \int_Y^\infty \frac{1}{U_s^2} \int_{Y_1}^\infty U_s h_2 \, d Y_2 \, d Y_1 =  L \big [ \sigma [U_s \rho] \big ] \int_0^\infty U_s h \, d Y\,,\\
\varphi_{2,1,2} & = - U_s \int_Y^{Y_0} \frac{1}{U_s^2} \, d Y_1 \, e^{-\alpha Y} \int_0^\infty U_s h \, d Y = - G_s  \, e^{-\alpha Y} \int_0^\infty U_s h \, d Y\,.
\end{split}
\end{align}
Here the function $G_s$ is studied in Proposition \ref{prop.pre}.
Then, since $\rho\in C_0^\infty (\R_+)$,  Proposition \ref{prop.pre.op} (2) for $L$ and Proposition \ref{prop.pre} (2) for $G_s$ imply 
\begin{align}\label{proof.lem.Ray.21} 
\| \pa_Y \varphi_{2,1,1} \| + \| \varphi_{2,1,1} \| + \| \varphi_{2,1,1}\|_{L^\infty}  + \| \frac{\varphi_{2,1,2}}{1+Y} \|_{L^\infty} & \leq C |\int_0^\infty U_s h\, d Y |\,.
\end{align}
Here $C$ is independent of $\alpha$.
Moreover, the same computation as in the proof of Proposition \ref{prop.pre} (2) leads to $\varphi_{2,1,1} (0) = \frac{1}{U_s'(0)} \int_0^\infty U_s h \, d Y$ thanks to $\int_0^\infty U_s \rho \, d Y=1$, and thus, we have 
\begin{align*}
(\varphi_{2,1,1} +\varphi_{2,1,2} )|_{Y=0} = 0\,.
\end{align*}
In particular, together with the estimates in Propositions \ref{prop.pre} and \ref{prop.pre.op} we see that $\varphi_{2,1} = \varphi_{2,1,1}+\varphi_{2,1,2}$ belongs to $H^1_0 (\R_+)$. 
Moreover, $\varphi_{2,1}$ satisfies for $Y>0$,
\begin{align*}
T\varphi_{2,1} & =  h_2 + 2 \alpha \pa_Y G_s  \, e^{-\alpha Y} \int_0^\infty U_s h \, d Y -\alpha^2 \varphi_{2,1,1} \nonumber \\
& = : h_2 + g_1 \,.
\end{align*}
To correct the error term $g_1$ we take $\varphi_{2,2}$ as the solution to $T\varphi_{2,2} = - g_1$ with $\varphi_{2,2}|_{Y=0}=0$. 
Let us decompose $g_1$ as $g_1=g_1\chi + g_1 (1-\chi)$, where $\chi$ is a smooth cut-off such that $\chi=1$ for $0\leq Y\leq 1$ and $\chi=0$ for $Y\geq 2$, and let $\varphi_{2,2,1}$ and $\varphi_{2,2,2}$ be respectively the solutions in $H^2(\R_+)\cap H^1_0 (\R_+)$ to $T \varphi_{2,2,1}=-g_1 \chi$ and $T \varphi_{2,2,2}=-g_1 (1-\chi)$.
From the formula $T = \frac{1}{U_s} \pa_Y (U_s^2 \pa_Y (\frac{\cdot}{U_s}))-\alpha^2 $ we have the estimate 
\begin{align}\label{proof.lem.Ray.25'} 
\| \frac{\varphi_{2,2,2}}{U_s} \|_{L^\infty} {\leq \frac{C}{\alpha^2} \| \frac{g_1 (1-\chi)}{U_s} \|_{L^\infty} \le \frac{C}{\alpha^2} \| g_1 (1-\chi) \|_{L^\infty} \,.}
\end{align}
Indeed, for $l\in \N$ we compute $\langle T\varphi_{2,2,2}, \varphi_{2,2,2} (\frac{\varphi_{2,2,2}}{U_s})^{2l}\rangle = - \langle g_1 (1-\chi),  \varphi_{2,2,2} (\frac{\varphi_{2,2,2}}{U_s})^{2l}\rangle$, which gives 
\begin{align*}
& (2 l + 1) \int_0^\infty | \pa_Y (\frac{\varphi_{2,2,2}}{U_s}) |^2 |\frac{\varphi_{2,2,2}}{U_s} |^{2l}  U_s^2 \, d Y + \alpha^2 \int_0^\infty |\frac{\varphi_{2,2,2}}{U_s}|^{2(l+1)} U_s^{{2}} \, d Y \\
& \leq \int_0^\infty {\bigl|\frac{g_1 (1-\chi)}{U_s}\bigr| \, \bigl|\frac{\varphi_{2,2,2}}{U_s} \bigr|^{2l+1} U_s^2} \, d Y\,,
\end{align*}
which gives 
\begin{align*}
\Big (\int_0^\infty |\frac{\varphi_{2,2,2}}{U_s} |^{2(l+1)} U^{{2}}_s \, d Y\Big )^\frac{1}{2(l+1)} \leq \frac{1}{\alpha^2} \Big ( \int_0^\infty {\big|\frac{g_1 (1-\chi)}{U_s}\bigr|^{2(l+1)} U_s^2}\, d Y\Big )^\frac{1}{2(l+1)}\,. 
\end{align*}
Then taking the limit $l\rightarrow \infty$ yields \eqref{proof.lem.Ray.25'}.
Thus we have from \eqref{proof.lem.Ray.21} and from Proposition \ref{prop.pre} (2) for the $L^\infty$ bound of $\pa_Y G_s$ in $Y\geq 1$,
\begin{align}\label{proof.lem.Ray.25} 
\| \frac{\varphi_{2,2,2}}{U_s} \|_{L^\infty} \leq \frac{C}{\alpha} |\int_0^\infty U_s h\, d Y |\,.
\end{align}
As for $\varphi_{2,2,1}$, we have from \eqref{proof.lem.Ray.11} that 
\begin{align*}
\| \varphi_{2,2,1}\|\leq \frac{C}{\alpha^2} \| g_1\chi\|\leq \frac{C}{\alpha} |\int_1^2 U_s h_2\, d Y|\,,
\end{align*}
and thus, again from \eqref{proof.lem.Ray.11} and the interpolation inequality we have 
\begin{align*}
\| \pa_Y \varphi_{2,2,1} \|^2 \leq C \| \varphi_{2,2,1} \|^2 + C\| Y g_1\chi \|^2 \leq \frac{C}{\alpha^2} |\int_0^\infty U_s h\, d Y|\,.
\end{align*}
Collecting these, we obtain 
\begin{align*}
\| \varphi_{2,2}\|_{L^\infty} \leq \| \varphi_{2,2,1} \|_{L^\infty} + \| \varphi_{2,2,2} \|_{L^\infty} \leq \frac{C}{\alpha} |\int_0^\infty U_s h\, d Y|\,.
\end{align*}
Thus, $\varphi_{2}=\varphi_{2,1}+\varphi_{2,2}=\varphi_{2,1,1} + \varphi_{2,1,2} + \varphi_{2,2}$ satisfies, by \eqref{proof.lem.Ray.21},
\begin{align}\label{proof.lem.Ray.27} 
\|\frac{\varphi_2}{1+Y} \|_{L^\infty} \leq \frac{C}{\alpha} |\int_0^\infty U_s h \, d Y|\,.
\end{align}
Then the fact $T\varphi_2 =  h_2$ with $\varphi_2|_{Y=0}=0$ and \eqref{proof.lem.Ray.11} with \eqref{proof.lem.Ray.27} yield
\begin{align}\label{proof.lem.Ray.28} 
\| \pa_Y \varphi_2 \|^2 & \leq \frac{C}{\alpha^2} |\int_0^\infty U_s h \, d Y|^2 + C \| Y h_2 \|^2\leq \frac{C}{\alpha^2}   |\int_0^\infty U_s h \, d Y|^2 \,,
\end{align}
and therefore,  again from \eqref{proof.lem.Ray.11},
\begin{align}\label{proof.lem.Ray.29} 
\alpha^2 \|\varphi_2 \|^2 \leq C\| Y h_2 \| \|\pa_Y \varphi_2 \| \leq \frac{C}{\alpha} |\int_0^\infty U_s h\, d Y|^2\,.
\end{align} 
Finally the $H^2$ estimate of $\varphi_2$ is obtained from the equation as 
\begin{align}\label{proof.lem.Ray.30} 
\| (\pa_Y^2-\alpha^2) \varphi_2 \| \leq \| \frac{U_s''}{U_s} \varphi_2 \| + \| h_2 \| \leq C \| \pa_Y \varphi_2 \| + \| h_2 \| \leq \frac{C}{\alpha} |\int_0^\infty U_s h\, d Y|\,.
\end{align}
Collecting \eqref{proof.lem.Ray.16}, \eqref{proof.lem.Ray.17}, \eqref{proof.lem.Ray.18} with $h_1=h-(\int_0^\infty U_s h \, d Y) \rho$, and \eqref{proof.lem.Ray.28}, \eqref{proof.lem.Ray.29}, and \eqref{proof.lem.Ray.30},  we obtain the estimates of $\varphi=\varphi_1+\varphi_2$. The proof is complete.
\end{proof}

\


The next proposition will be used to construct a slow mode in the case $0<\alpha\ll 1$, the boundary corrector for the Orr-Sommerfeld equation. 
For later use let us introduce the operator 
\begin{align}\label{def.op.K}
K[f] (Y) = U_s \int_Y^\infty \frac{1}{U_s^2} \int_{Y_1}^\infty  f \, d Y_2 \, d Y_1 = L\big [ \sigma [f] \big ] (Y) \,, \qquad  f\in C_0^\infty (\R_+)\,.
\end{align} 
The estimate of $K$ will be derived from the estimates of $L$ and $\sigma$ in Proposition \ref{prop.pre.op}.
\begin{proposition}\label{prop.slow.Ray} For any  {$0<\alpha\leq 1$}, there exists a function $\varphi\in H^1 (\R_+)$ satisfying 
\begin{align*}
Ray[\varphi] =0\qquad Y>0
\end{align*}
in the sense of distributions and the following properties: $\varphi = \varphi_0 + \varphi_1 + \varphi_2$, where 
\begin{align}
& \varphi_0 = U_s e^{-\alpha Y}\,, \qquad \varphi_1|_{Y=0} = \frac{U_E^2}{U_s'(0)} \alpha + O (\alpha^2) \,, \label{est.prop.slow.Ray.1}\\
& \|  \pa_Y \varphi_1 \| +   \| \varphi_1 \| \leq C \alpha\,, \label{est.prop.slow.Ray.2}\\
& \|  \pa_Y \varphi_2 \| +  \alpha \|  \varphi_2 \|  \leq C \alpha^{{3/2}}\,. \label{est.prop.slow.Ray.3}
\end{align}
Here $C$ is independent of $\alpha$. 
If $\frac{U_s''}{U_s}\in L^2 (\R_+)$ in addition, then $\varphi_1$ and $\varphi_2$ belong to $H^2 (\R_+)$.
\end{proposition}

\begin{proof} {We look for a  solution $\varphi$ to $Ray[\varphi]=0$ of the form $\varphi= U_s e^{-\alpha Y} +  \varphi_1 + \varphi_2$. We should have $Ray[\varphi_1+\varphi_2] = 2 \alpha U_s U_s' e^{-\alpha Y}$.}  

{We first consider the problem 
\begin{align}\label{proof.prop.slow.Ray.1}
U_s \pa_Y^2 \varphi_1 - U_s''\varphi_1 =  2 \alpha U_s U_s' e^{-\alpha Y}\,, \qquad \lim_{Y\rightarrow \infty} \varphi_1 =0\,.
\end{align}}
Its solution is given by 
\begin{align}\label{proof.prop.slow.Ray.2}
\varphi_1 = 2\alpha K [U_s U_s' e^{-\alpha Y}]\,,
\end{align}
where $K$ is defined by \eqref{def.op.K}. From Proposition \ref{prop.pre.op} we have 
\begin{align*}
\| \varphi_1 \| = 2 \alpha \| L\big [ \sigma [U_s U_s' e^{-\alpha Y}] \big ] \|& \leq C \alpha \| (1+Y) \sigma [U_s U_s' e^{-\alpha Y}] \| \\
& \leq C \alpha (\|  (1+Y)^2 U_s U_s' e^{-\alpha Y} \|    \leq C \alpha\,.
\end{align*}
Here we have used the decay condition  $|U_s'(Y)|\leq C (1+Y)^{-3}$.
Similarly, we have 
\begin{align*}
\| \pa_Y \varphi_1 \| & \leq C \alpha \big ( \| \sigma [U_s U_s' e^{-\alpha Y}]  \|_{L^1} + \| \sigma [U_s U_s' e^{-\alpha Y}] \| + \| \pa_Y \sigma [U_s U_s' e^{-\alpha Y}] \| \big ) \\
& \leq C \alpha \big ( \| Y U_s U_s' e^{-\alpha Y} \|_{L^1} + \| Y U_s U_s' e^{-\alpha Y} \|  + \| U_s U_s' e^{-\alpha Y} \| \big ) \\
& \leq C \alpha \,.
\end{align*}
Estimate \eqref{est.prop.slow.Ray.2} is proved. 

{Eventually, we introduce the solution $\varphi_2$ of 
\begin{align*}
Ray[\varphi_2] =\alpha^2 U_s \varphi_1 = 2 \alpha^3 U_s K [U_s U_s' e^{-\alpha Y}], \quad \varphi_2\vert_{Y=0} = 0, 
\end{align*}
that can be estimated using case ii) of Proposition \ref{prop.Ray}. { We use again Proposition \ref{prop.pre.op} and  the bound  $|U_s'(Y)|\leq C (1+Y)^{-3}$ to compute:
\begin{align*} 
\|\pa_Y \varphi_2 \| + \alpha \| \varphi_2\| & \le C \alpha^3  \| (1+Y)^2 K [U_s U_s' e^{-\alpha Y}] \| + C \alpha^2 | \int_0^{+\infty}  K [U_s U_s' e^{-\alpha Y}](Y) \, d Y | \\
& \le C \alpha^3  \| (1+Y)^4 U_s U_s' e^{-\alpha Y} \| + C \alpha^2  \int_0^{+\infty} (1+Y)^{-1} e^{-\alpha Y} \, d Y  \\
& \le C \alpha^3  \| (1+Y) e^{-\alpha Y} \| + C \alpha^2 |\ln \alpha |  \le C \alpha^{3/2} + \alpha^2 |\ln \alpha|\, . 
\end{align*}
Here, note that the bound $|K[U_sU_s' e^{-\alpha Y}] (Y)| \leq C (1+Y)^{-1} e^{-\alpha Y}$ used in the second line is proved by the following observation: 
\begin{align*}
|K[U_sU_s'e^{-\alpha Y}](Y)|  &\leq C U_s e^{-\alpha Y} \int_Y^\infty \frac{1}{U_s^2} \int_{Y_1}^\infty {(1+Y_2)^{-3}}\, d Y_2 \, d Y_1 \\
& \leq C U_s e^{-\alpha Y} \int_Y^\infty \frac{1}{U_s^2 (1+Y_1)^2} \, d Y_1  \leq C (1+Y) e^{-\alpha Y}\,. 
\end{align*}
In the last line we have also used the argument as in the computation of $G_s$ in Proposition \ref{prop.pre} (2) when $Y$ is small. The details are omitted here.}
The proof is complete.}
\end{proof}

\mspace

Let $\varphi=\varphi_0+\varphi_1 + \varphi_2 \in H^1(\R_+)$ be the function obtained in Proposition \ref{prop.slow.Ray},
and set 
\begin{align}
\varphi_{slow,Ray} = \frac{c_E}{\alpha} \varphi\,, \qquad c_E = \frac{\alpha}{\varphi_1 (0)} = \frac{U_s'(0)}{U_E^2}+ O (\alpha)\,, \quad  0<\alpha\leq  1 \,.
\end{align}
As a direct consequence of Proposition \ref{prop.slow.Ray}, we have 
\begin{corollary}\label{cor.prop.slow.Ray} {For any  $0<\alpha\leq 1$}, there exists a function $\varphi_{slow,Ray}\in H^1 (\R_+)$ satisfying 
\begin{align*}
Ray[\varphi_{slow,Ray}] =0\qquad Y>0
\end{align*}
in the sense of distributions and the following properties: $\varphi_{slow,Ray} = \varphi_{sRay,0} + \varphi_{sRay,1} + \varphi_{sRay,2}$, where 
\begin{align}
& \varphi_{sRay,0} =\frac{c_E}{\alpha} U_s e^{-\alpha Y}\,, \qquad \varphi_{sRay,1} (0) = 1\,, \label{est.cor.prop.slow.Ray.1}\\
& \|  \pa_Y \varphi_{sRay,1} \|   +  \|  \varphi_{sRay,1} \| \leq C\,, \label{est.cor.prop.slow.Ray.2}\\
& \|  \pa_Y \varphi_{sRay,2} \|  + \alpha \| \varphi_{sRay,2} \|  \leq C \alpha^{1/2}\,. \label{est.cor.prop.slow.Ray.3}
\end{align}
Here $C$ is independent of $\alpha$. In particular, we have 
\begin{align}\label{est.cor.prop.slow.Ray.4}
{\varphi_{slow,Ray} (0) = 1\,.}
\end{align}
If $\frac{U_s''}{U_s}\in L^2 (\R_+)$ in addition, then $\varphi_{sRay,1}$ and $\varphi_{sRay,2}$ belong to $H^2 (\R_+)$.
\end{corollary}

\mspace

When {$\alpha \geq 1$} we have 
\begin{proposition}\label{prop.slow.Ray'} {If $\alpha \geq 1$} then there exists a function $\varphi_{slow,Ray}\in H^1 (\R_+)$ satisfying 
\begin{align*}
Ray[\varphi_{slow,Ray}] =0\qquad Y>0\,, \qquad  \varphi_{slow,Ray}|_{Y=0} = 1\,,
\end{align*}
and the following properties:  $\varphi_{slow,Ray} = e^{-\alpha Y} + \tilde \varphi_{slow,Ray}$ with $\tilde \varphi_{slow,Ray}\in H^1_0 (\R_+)$, where
\begin{align}\label{est.prop.slow.Ray'.1} 
\| \pa_Y \tilde \varphi_{slow,Ray} \| + \alpha \| \tilde \varphi_{slow,Ray} \| \leq  C \min \{ 1, \alpha^{-\frac12} \}\,.
\end{align}
{Here $C$ is independent of $\alpha \ge 1$.}
If $\frac{U_s''}{U_s}\in L^2 (\R_+)$ in addition, then $\tilde \varphi_{slow,Ray}$ belongs to $H^2 (\R_+)$.
\end{proposition}

\begin{proof} The function $\tilde \varphi_{slow,Ray}$ is constructed as the weak solution to 
\begin{align*}
Ray [\tilde \varphi_{slow,Ray}] = U_s'' e^{-\alpha Y}\,, \qquad \tilde \varphi_{slow,Ray}|_{Y=0} =0\,.
\end{align*}
Proposition \ref{prop.Ray} shows that 
\begin{align*}
\| \pa_Y \tilde \varphi_{slow,Ray} \| + \alpha \| \tilde \varphi_{slow,Ray} \| \leq C {\| (1+Y) U_s'' e^{-\alpha Y} \| }\leq C \min \{1, \alpha^{-\frac12} \}\,.
\end{align*}
From $Ray = U_s (\pa_Y^2-\alpha^2 ) - U_s''$ it is straightforward that $\tilde \varphi_{slow,Ray}$ belongs to $H^2 (\R_+)$ if $\frac{U_s''}{U_s}\in L^2(\R_+)$. The proof is complete. 
\end{proof}

\section{Airy equation}\label{sec.Airy}

Set $Airy[\psi] = U_s \psi + i \eps (\pa_Y^2-\alpha^2) \psi$ with $\eps = 1/\tilde n$. 
In this section we consider the Airy equation
\begin{align}\label{eq.Airy}
\begin{cases}
& Airy[\psi] =  \eps f \,, \qquad Y>0\,,\\
& \psi|_{Y=0} = 0\,.
\end{cases}
\end{align}

\begin{proposition}\label{prop.Airy}
Let $f\in L^2 (\R_+)$. Then there exists a unique solution $\psi \in H^2(\R_+)\cap H^1_0 (\R_+)$ to \eqref{eq.Airy} such that 
\begin{align}
\| U_s \psi \| + \eps^\frac16 \| \sqrt{U_s} \psi \| + \eps^\frac13 \| \psi \| + \eps^\frac23 \big (\| \pa_Y \psi \| + \alpha \| \psi \|  \big ) + \eps \| (\pa_Y^2 -\alpha^2 ) \psi \| & \leq C \eps \| f\|\,,\label{est.prop.Airy.0}
\end{align}
and also 
\begin{align}
\| U_s Y \psi \| \leq C \eps \|Y f \| + C \eps^\frac43 \| f\|  \label{est.prop.Airy.-2}
\end{align}
if $(1+Y) f\in L^2 (\R_+)$ in addition.
Moreover, if $f$ is replaced by $\pa_Y f$ or $\frac{f}{Y}$, then 
\begin{align}
\eps^\frac12 \| \sqrt{U_s} \psi \| + \eps^\frac23 \|\psi \| + \eps \big ( \| \pa_Y \psi \| + \alpha \| \psi \|  \big ) & \leq C \eps \| f\|\,. \label{est.prop.Airy.4}
\end{align}
In the case when $f$ is replaced by $\frac{f}{Y}$ we also have 
\begin{align}
\| U_s \psi \| \leq  C\eps^\frac23 \| f\|\,.\label{est.prop.Airy.6}
\end{align}
\end{proposition}

\begin{remark}{\rm From the proof one can check that the unique existence of the weak solution in $H^1_0 (\R_+)$ is valid even when $f$ is replaced by $\pa_Y f$ or $\frac{f}{Y}$.
}
\end{remark}

\begin{proof} We focus on the a priori estimates. We first take the inner product with $\psi$ in the equation $Airy[\psi]=\eps f$, and then the real part and the imaginary part respectively give
\begin{align}
\| \sqrt{U_s} \psi \|^2 & = \eps \Re \langle f, \psi\rangle\,,\label{proof.prop.Airy.1}\\
\| \pa_Y \psi \|^2 + \alpha^2 \| \psi \|^2 & = -  \Im \langle f, \psi \rangle\,.\label{proof.prop.Airy.2}
\end{align}
Similarly, we take the inner product with $(\pa_Y^2-\alpha^2)\psi$ in the equation $Airy[\psi]=\eps f$, and then the imaginary part lead to 
\begin{align}
\eps \| (\pa_Y^2 - \alpha^2) \psi \|^2 & = \Im \langle U_s' \psi,\pa_Y \psi \rangle + \eps \langle f, (\pa_Y^2-\alpha^2) \psi \rangle\,.\label{proof.prop.Airy.4}
\end{align}
To obtain the estimate of $\|\psi\|$, we apply the interpolation inequality \eqref{proof.prop.Airy.5} together with \eqref{proof.prop.Airy.1}-\eqref{proof.prop.Airy.2}. They imply
\begin{align*}
\| \psi \|^2 & \leq C \| \sqrt{U_s} \psi \|^\frac43 \| \pa_Y \psi \|^\frac23  + C \| \sqrt{U_s} \psi \|^2 \\
& \leq C (\eps \Re \langle f, \psi\rangle )^\frac23 (|\Im \langle f, \psi\rangle|)^\frac13 + C \eps \Re \langle f, \psi \rangle\\
& \leq C \eps^\frac23 \| f\| \| \psi \| + \eps \| f\| \| \psi \|\\
& \leq C \eps^\frac43 \| f\|^2\,.
\end{align*}
Here $C$ is a universal constant. This proves the estimate of $\|\psi\|$.
The $H^1$ estimate of $\psi$ then follows from the estimate of $\|\psi\|$ and \eqref{proof.prop.Airy.2}.
The $H^2$ estimate easily follows from \eqref{proof.prop.Airy.4} and the $H^1$ estimate of $\psi$, while the estimate $\| \sqrt{U_s} \psi \|\leq C \eps^\frac56 \| f\|$ is obtained from the estimate of  $\|\psi\|$ and \eqref{proof.prop.Airy.1}. The details are omitted here.
Next we take the inner product with $U_s \psi$ in the equation $Airy [\psi] = \eps f$ and take the real part, which gives 
\begin{align*}
\| U_s\psi \|^2 + \eps \Im \langle \pa_Y \psi, U_s' \psi \rangle = \eps \Re \langle f, U_s \psi \rangle\,.
\end{align*}
Thus we have 
\begin{align*}
\| U_s \psi \|^2 \leq \eps \| U_s' \|_{L^\infty}  \| \pa_Y \psi \| \| \psi \| + \eps \| f\| \| U_s \psi \| \leq C \eps^2 \| f \|^2\,.
\end{align*}
Hence the estimate {of} $\|U_s \psi \|$ holds. When $(1+Y) f\in L^2(\R_+)$ it is not difficult to show that $(1+Y)\psi \in H^2 (\R_+) \cap H^1_0(\R_+)$. Indeed, we first take the inner product with $Y^2\chi_R^2 \psi$ to $Airy[\psi]=\eps f$, where $\chi_R$ is a smooth cut-off such that $\chi_R=1$ for $0\leq Y\leq R$ and $\chi_R=0$ for $Y\geq 2R$ with $\|\pa_Y^k\chi_R\|_{L^\infty} \leq C R^{-k}$. Then taking the limit $R\rightarrow \infty$ verifies $Y\psi, Y \pa_Y\psi\in L^2 (\R_+)$, from which it is also easy to see that $(1+Y)\psi \in H^2 (\R_+)$ by using the elliptic regularity. Now we observe that
\begin{align*}
Airy [Y\psi] = \eps Y f + i 2 \eps \pa_Y \psi \,.
\end{align*}
Thus we have
\begin{align*}
\| U_s Y \psi \|\leq C \eps \| Y f \| + 2 \eps \| \pa_Y \psi \| \leq C \eps \| Y f \| + C \eps^\frac43 \| f\| \,.
\end{align*} 
This proves \eqref{est.prop.Airy.-2}.
Finally, let us consider the case when $f$ is replaced by $\pa_Y f$ or $\frac{f}{Y}$.
We give the proof only for the case $\frac{f}{Y}$, for the argument of the case $\pa_Y f$ is the same by applying the integration by parts in the inner product $\langle \pa_Y f, \psi\rangle = -\langle f, \pa_Y \psi \rangle$.
The energy equalities \eqref{proof.prop.Airy.1}-\eqref{proof.prop.Airy.2} are replaced by 
\begin{align}
\| \sqrt{U_s} \psi \|^2 & = \eps \Re \langle f, \frac{\psi}{Y} \rangle\,,\label{proof.prop.Airy.6}\\
\| \pa_Y \psi \|^2 + \alpha^2 \| \psi \|^2 & =  -\Im \langle f, \frac{\psi}{Y} \rangle\,.\label{proof.prop.Airy.7}
\end{align}
Equality \eqref{proof.prop.Airy.7} gives the $H^1$ estimate by the Hardy inequality $\|\frac{\psi}{Y}\|\leq C \| \pa_Y \psi \|$, and then the estimate of $\|\sqrt{U_s} \psi \|$ follows from \eqref{proof.prop.Airy.6} and the estimate of $\| \pa_Y \psi \|$ by applying the Hardy inequality for the term $\frac{\psi}{Y}$ again. The estimate of  $\|\psi \|$ then follows from the interpolation inequality \eqref{proof.prop.Airy.5} and the estimates of $\|\sqrt{U_s} \psi\|$ and $\|\pa_Y \psi \|$. Estimate \eqref{est.prop.Airy.6} follows as above from the equality 
\begin{align*}
\| U_s \psi \|^2 + \eps \Im \langle \pa_Y \psi, U_s' \psi \rangle = \eps \Re \langle f, \frac{U_s}{Y} \psi \rangle\,,
\end{align*}
and by using the bound $\|\frac{U_s}{Y} \|_{L^\infty}<\infty$.
The details are omitted. The uniqueness of the solution follows from the a priori estimates.

As for the existence, we first consider the problem $Airy[\psi] - i l \psi = f$ for $l>0$. When $l$ is large the operator $Airy - i l$ is clearly invertible in $L^2(\R_+)$, while all of the a priori estimates are valid uniformly in $l>0$ by applying the same argument as above. This implies the existence of the solution for the case $l=0$ by the standard continuity method. The proof is complete. 
\end{proof}

For later use we consider the Airy equation under the Neumann boundary condition{:}
\begin{align}\label{eq.Airy.N}
\begin{cases}
& Airy[\psi] =   \pa_Y f \,, \qquad Y>0\,,\\
& \pa_Y \psi|_{Y=0} = 0\,.
\end{cases}
\end{align}
\begin{proposition}\label{prop.Airy.N}
Let $f\in H^1_0 (\R_+)$. Then there exists a unique solution $\psi \in H^2 (\R_+)$ to \eqref{eq.Airy.N} such that 
\begin{align}
\| U_s \psi \| & \leq \frac{C}{\eps^\frac13} \| f\| + \frac{C}{\eps^\frac23} \| U_s f \| \,, \label{est.prop.Airy.N.0}
\end{align}
and 
\begin{align}
\eps^\frac12 \| \sqrt{U_s} \psi \|  + \eps^\frac23 \| \psi \| + \eps \big ( \| \pa_Y \psi \| + \alpha \| \psi \| \big ) & \leq C \| f\|\,. \label{est.prop.Airy.N.1}
\end{align}
Moreover, if $(1+Y)^2 f\in H^1 (\R_+)$ in addition, then 
\begin{align}
\| Y \psi \| & \leq  \frac{C}{\eps^\frac23} \| Y f\| + \frac{C}{\eps^\frac13} \| f\| \label{est.prop.Airy.N.2}\,, \\
\| Y^2 U_s \psi \| & \leq \frac{C}{\eps^\frac23} \| U_s Y^2 f\| + \frac{C}{\eps^\frac13} \| U_s Y f \| + C \| Y f \| + C \eps^\frac13\| f\|{\,,}   \label{est.prop.Airy.N.3}
\end{align}
and also
\begin{align}
\| \psi \|_{L^1} & \leq \frac{C}{\epsilon^{\frac56}} \| Y f\| + \frac{C}{\eps^{\frac12}} \| f\| \,,\label{est.prop.Airy.N.4}\\
|\int_0^\infty U_s \psi \, d Y| & \leq C \alpha^2 \Big ( \eps^{\frac16} \| Y f \| +  \epsilon^{\frac12} \| f \|  \Big )\,.\label{est.prop.Airy.N.5}
\end{align}
Finally, {when $0 \le \alpha \le 1$,} the function $\sigma[ U_s \psi] (Y) = \int_Y^\infty U_s \psi \, d Y_1$ satisfies 
\begin{align}
\| (1+Y) \sigma [U_s \psi] \|\leq C \| (1+Y)^2 f \|\,.\label{est.prop.Airy.N.6}
\end{align}
\end{proposition}

\begin{proof} As in the proof of the previous proposition{,} we have 
\begin{align}
\|\sqrt{U_s} \psi \|^2 = - \Re \langle f, \pa_Y \psi \rangle{\,,} \label{proof.prop.Airy.N.1}\\
\| \pa_Y \psi \|^2 + \alpha^2 \| \psi \|^2 = \eps^{-1} \Im \langle f, \pa_Y \psi \rangle\,,\label{proof.prop.Airy.N.2}
\end{align}
and also 
\begin{align}
\eps \| (\pa_Y^2-\alpha^2) \psi \|^2 = \Im \langle U_s' \psi, \pa_Y \psi \rangle + \langle f, (\pa_Y^2-\alpha^2) \psi \rangle\,.\label{proof.prop.Airy.N.3}
\end{align}
The interpolation inequality 
\begin{align*}
\| \psi \|^2 \leq C \|\sqrt{U_s} \psi\|^\frac43 \| \pa_Y \psi \|^\frac23 + C \| \sqrt{U_s} \psi \|^2 
\end{align*}
is valid for $\psi \in H^1 (\R_+)$ and thus, {\eqref{proof.prop.Airy.N.1} and \eqref{proof.prop.Airy.N.2} imply} 
\begin{align*}
\| \psi \|^2\leq \frac{C}{\eps^\frac13} \| f\| \|\pa_Y \psi \| +  \| f\| \| \pa_Y \psi \| \leq \frac{C}{\eps^\frac43} \| f\|^2\,.
\end{align*}
Estimate \eqref{est.prop.Airy.N.1} has been proved. 
Next by taking the inner product with $U_s\psi$ in the equation $Airy[\psi]=\pa_Y f$ and by taking the real part, we see
\begin{align*}
\|U_s \psi \|^2 + \eps \Im \langle \pa_Y \psi, U_s' \psi \rangle = - \Re \langle f, \pa_Y (U_s \psi)  \rangle 
\end{align*}
Thus we have 
\begin{align*}
\| U_s \psi \|^2 & \leq \eps \| U_s' \|_{L^\infty} \| \pa_Y \psi \| \| \psi \| + C \| f\| \| \psi \| + \| U_s f \| \|\pa_Y \psi \| \\
& \leq \frac{C}{\eps^\frac23} \| f \|^2 + \frac{C}{\eps^\frac23} \| f\|^2 + \frac{C}{\eps} \|  U_s f \| \| f \|\,.
\end{align*}
Hence \eqref{est.prop.Airy.N.0} holds.
To obtain the weighted estimate we see
\begin{align*}
Airy [Y\psi] = Y \pa_Y f +  2 i \eps  \pa_Y \psi = \pa_Y (Y f) - f + 2 i \eps  \pa_Y \psi \,,
\end{align*}
and we have by applying Proposition \ref{prop.Airy},
\begin{align}
\eps^\frac13 \| \pa_Y (Y \psi ) \| + \| Y \psi \| & \leq \frac{C}{\eps^\frac23} \| Y f \| + \frac{C}{\eps^\frac13} \| f\| + C \eps^{\frac13} \| \psi \| \nonumber \\
& \leq \frac{C}{\eps^\frac23} \| Y f \| + \frac{C}{\eps^\frac13} \| f\| {\,.} \label{proof.prop.Airy.N.5}
\end{align}
{Hence \eqref{est.prop.Airy.N.2} holds.
Moreover, from the computation of the inner product 
\begin{align*}
 \langle Airy[Y\psi], U_s Y \psi \rangle & = \langle Y \pa_Y f + 2i\eps \pa_Y \psi, U_s Y \psi \rangle \\
& = -\langle f, \pa_Y (Y U_s Y \psi) \rangle + 2i \eps \langle \pa_Y \psi, U_s Y \psi \rangle \\
& = -\langle U_s Y f, \pa_Y (Y \psi) \rangle - \langle f, U_s Y\psi\rangle -\langle U_s' Y f, Y \psi \rangle +  2i \eps \langle \pa_Y \psi, U_s Y \psi \rangle 
\end{align*}
and by taking the real part of it, we finally achieve from \eqref{proof.prop.Airy.N.5}, 
\begin{align}\label{proof.prop.Airy.N.6}
\| U_s Y \psi \| \leq \frac{C}{\eps^\frac13} \| Y f\| + \frac{C}{\eps^\frac23} \| U_s Y f\|  + C \| f\| \,. 
\end{align}
The details are omitted here.}
Next we see
\begin{align*}
Airy[Y^2 \psi] = Y^2 \pa_Y f + 2i \eps \psi +4i\eps Y \pa_Y \psi & = \pa_Y (Y^2 f) -2 Y f - 2i \eps \psi +4i\eps \pa_Y (Y \psi ) \,.
\end{align*}
Thus from \eqref{est.prop.Airy.N.0} and Proposition \ref{prop.Airy},
\begin{align}
\| U_s Y^2 \psi \| & \leq \frac{C}{\eps^\frac13} \| Y^2 f\| + \frac{C}{\eps^\frac23} \| U_s Y^2 f \| + C \| Y f\| + C \eps \| \psi \| \nonumber \\
& \quad + C \eps ( \frac{1}{\eps^\frac13} \| Y \psi \| + \frac{1}{\eps^\frac23} \| U_s Y \psi\| ) \nonumber \\
& \leq \frac{C}{\eps^\frac13} \| Y^2 f\| + \frac{C}{\eps^\frac23} \| U_s Y^2  f \|+ C \| Y f\| + C \eps^\frac13 \| f \| \nonumber \\
& \quad + C \eps^\frac23 \| Y \psi \| + C \eps^{\frac13} \| U_s Y \psi\|  \,.\label{proof.prop.Airy.N.4}
\end{align}
Next we have from the interpolation $\| \psi \|_{L^1} \leq C \| Y \psi \|^\frac23 \| \psi \|_{L^\infty}^\frac13 \leq C \| Y \psi \|^\frac23 \| \pa_Y \psi \|^\frac16 \| \psi \|^\frac16$,
\begin{align*}
\| \psi \|_{L^1} & \leq C ( \frac{1}{\eps^\frac23} \| Y f\| + \frac{1}{\eps^\frac13} \| f\| )^\frac23 \frac{1}{\eps^\frac16}\|  f\|^\frac16 \frac{1}{\eps^\frac19} \| f\|^\frac16  \\
& \leq C ( \frac{1}{\eps^\frac23} \| Y f\| + \frac{1}{\eps^\frac13} \| f\| )^\frac23 \frac{1}{\eps^\frac{5}{18}} \| f\|^\frac13\,.
\end{align*}
This implies \eqref{est.prop.Airy.N.4}.
We observe that from the integration by parts,
\begin{align*}
\int_0^\infty U_s \psi \, d Y & = -i\eps \int_0^\infty (\pa_Y^2-\alpha^2) \psi \, d Y + \int_0^\infty \pa_Y f \, d Y \\
&  = i\eps \alpha^2 \int_0^\infty \psi\, d Y\,.
\end{align*}
Then \eqref{est.prop.Airy.N.5} follows from the $L^1$ estimate of $\psi$ in \eqref{est.prop.Airy.N.4}.
{Finally, we observe that $\sigma[U_s \psi] (Y)= \int_Y^\infty U_s \psi \, d Y_1$ satisfies 
\begin{align*}
\sigma [U_s \psi] = i\eps \pa_Y \psi + i\eps \alpha^2 \int_Y^\infty \psi \, d Y_1 - f = i\eps \pa_Y \psi + i\eps \alpha^2 \sigma[\psi] - f \,.
\end{align*}
Thus we have 
\begin{align*}
\| (1+Y) \sigma [U_s \psi] \| & \leq \eps \| (1+Y) \pa_Y \psi\| + \eps \alpha^2 \| (1+Y) \sigma[\psi]\| + \| (1+Y) f\| \,,
\end{align*}
and then,  it follows from Proposition \ref{prop.pre.op} (1) that }
\begin{align*}
\| (1+Y) \sigma [U_s \psi] \| & \leq \eps \| (1+Y) \pa_Y \psi\| + C \eps \alpha^2 \| Y (1+Y)  \psi \| + \| (1+Y) f\| \,.
\end{align*}
Hence, {for $0 \le \alpha \le 1$},  \eqref{est.prop.Airy.N.1}, \eqref{proof.prop.Airy.N.5}, \eqref{proof.prop.Airy.N.5}, and \eqref{proof.prop.Airy.N.4} yield $\| (1+Y) \sigma [U_s \psi]\|\leq C \| (1+Y)^2 f \|$, as desired.
The proof is complete.
\end{proof}

\section{Orr-Sommerfeld equation}\label{sec.OS}

Set $OS[\phi] = Ray[\phi] + i \eps (\pa_Y^2-\alpha^2)^2 \phi$. 
This aim of this section is to solve the Orr-Sommerfeld equation
\begin{align}\label{eq.OS}
\begin{cases}
& OS[\phi] =   f \,, \qquad Y>0\,,\\
& \phi|_{Y=0} = \pa_Y \phi|_{Y=0} = 0\,.
\end{cases}
\end{align}
We assume that $\tilde 1 \sqrt{\nu} \leq \alpha \ll \eps^{-\frac13}$ with $\eps=1/\tilde n$, which means $\tilde 1 \leq \tilde n \ll \nu^{-\frac34}$.  
To simplify the statement we focus on the case when $f$ decays fast enough, though this condition can be relaxed to some extent.
\begin{theorem}\label{thm.OS} There exist positive numbers $\delta_0, \eps_0$ such that if $0<\eps\leq \eps_0$ and $0<\eps^\frac13 \alpha \leq \delta_0$ then for any $f\in L^2 (\R^2_+)$ with $(1+Y)^2 f \in L^2 (\R_+)$ there exists a unique solution $\phi\in H^4 (\R_+)\cap H^2_0 (\R_+)$ to \eqref{eq.OS} satisfying 

\noindent {\rm (i)} when {$\alpha \geq 1$}:
\begin{align}
\| \pa_Y \phi \| + \alpha \| \phi\|  & \leq  C {\| (1+Y) f\|} \,, \label{est.thm.OS.1} \\
\| (\pa_Y^2-\alpha^2)\phi \| & \leq {\frac{C}{\eps^\frac13} \| (1+Y) f \|} \,,\label{est.thm.OS.2}
\end{align}

\

\noindent {\rm (ii)} when {$0<\alpha \leq 1$}:
\begin{align}
\| \pa_Y \phi \| + \alpha \| \phi\|  & \leq  \frac{\alpha + \eps^\frac16}{\alpha+\eps^\frac13}  \Big ( \| (1+Y)^2 f \| + \frac{1}{\alpha} |\int_0^\infty f\, d Y | \Big ) \,, \label{est.thm.OS.3} \\
\| \pa_Y \phi \|_{L^\infty}  & \leq   \frac{C}{\eps^\frac16}  \Big ( \| (1+Y)^2 f \| + \frac{1}{\alpha} |\int_0^\infty f\, d Y | \Big ) \,,  \label{est.thm.OS.4}\\
\| (\pa_Y^2-\alpha^2)\phi \| & \leq   \frac{C}{\eps^\frac13}  \Big ( \| (1+Y)^2 f \| + \frac{1}{\alpha} |\int_0^\infty f\, d Y | \Big ) \,.\label{est.thm.OS.5}
\end{align}

\end{theorem}

\begin{remark}{\rm (i) When $\alpha$ is small, the singular factor $\alpha^{-1}$ in front of  $|\int_0^\infty f \, d Y|$ is due to the Rayleigh equation and Proposition \ref{prop.Ray}. It can not be dropped in general.

\noindent (ii) Note that the $L^\infty$ estimate \eqref{est.thm.OS.4} is not a consequence of the interpolation between \eqref{est.thm.OS.3} and \eqref{est.thm.OS.5}.
{The loss of the factor $\eps^{-\frac16}$ in \eqref{est.thm.OS.3} appearing in the case $0<\alpha\ll 1$ comes from the slow mode of the boundary corrector. However, we can recover the estimates of $\|\pa_Y \phi\|_{L^\infty}$ and $\| (\pa_Y^2-\alpha^2) \phi\|$ as in \eqref{est.thm.OS.4} and \eqref{est.thm.OS.5}, that are considered to be optimal  in view of scaling.}}
\end{remark}

\subsection{Rayleigh-Airy iteration}\label{subsec.iteration}
In this subsection we consider the modified Orr-Sommerfeld equation
\begin{align}\label{eq.mOS}
\begin{cases}
& OS[\Phi] =   f \,, \qquad Y>0\,,\\
& \Phi|_{Y=0} = \pa_Y H_\alpha  \Phi |_{Y=0} = 0 {\,,}
\end{cases}
\end{align} 
{where the self-adjoint operator $H_\alpha = \pa_Y^2 -\alpha^2$ is realized in $L^2(\R_+)$ with the domain $H^2(\R_+)\cap H^1_0 (\R_+)$. 
That is, the original boundary condition on $\pa_Y\Phi|_{Y=0}=0$ is replaced by $\pa_Y H_\alpha  \Phi|_{Y=0} = 0$. 
To be rigorous our aim is to construct the solution $\Phi\in H^2(\R_+) \cap H^1_0 (\R_+)$ to the problem in the weak formulation  
\begin{align}\label{eq.weak.mOS}
\langle U_s H_\alpha \Phi - U_s'' \Phi, q\rangle + i \eps \langle H_\alpha \Phi, (\pa_Y^2-\alpha^2) q\rangle = \langle f, q\rangle\,, \qquad q\in H^2 (\R_+)~{\rm with}~\pa_Y q (0)=0\,.
\end{align}}
To state the main result of this subsection it is convenient to introduce the functions $\varphi_1, \psi_0 \in H^2(\R_+)\cap H^1_0 (\R_+)$, which are respectively the solutions to 
\begin{align*}
Ray[ \varphi_1 ] = f\,, \qquad  Airy [\psi_0] = -i\eps \frac{f}{U_s}\,, \qquad \frac{f}{U_s}\in L^2 (\R_+)\,.
\end{align*}

\begin{proposition}\label{prop.mOS} There exists a positive number $\eps_1$ such that the following statement holds for any $0<\eps\leq \eps_1$ and $\alpha>0$. Let $f/U_s\in L^2 (\R_+)$. Then there exists a solution $\Phi \in H^4 (\R_+)\cap H^1_0 (\R_+)$ to {\eqref{eq.mOS} satisfying the following estimates.} 

\noindent {\rm (i)} when $\alpha \geq 1$,
\begin{align}
\| \pa_Y (\Phi -\varphi_1-\psi_0 ) \| & \leq C \eps^\frac13 \| \pa_Y \varphi_1 \|  + \frac{C}{\eps^\frac13} \big ( \| \psi_0 \| + \frac{1}{\eps^\frac13}  \| U_s \psi_0 \|\big )  \,, \label{est.prop.mOS.1} \\
\alpha \| \Phi -\varphi_1 -\psi_0  \|  & \leq C \eps^\frac13 ( 1+\alpha\eps^\frac13) \| \pa_Y \varphi_1 \|  + C (\alpha \eps^\frac13 +\frac{1}{\eps^\frac13} ) \big ( \| \psi_0 \| + \frac{1}{\eps^\frac13} \| U_s \psi_0 \| \big )\,, \label{est.prop.mOS.2} \\
\| (\pa_Y^2-\alpha^2) \big ( \Phi -\varphi_1-\psi_0 \big ) \| & \leq C \| \pa_Y \varphi_1 \| + \frac{C}{\eps^\frac23} \| \psi_0 \| \,.\label{est.prop.mOS.3} 
\end{align}

\noindent {\rm (ii)} when $0<\alpha\leq 1$,
\begin{align}
\| \pa_Y (\Phi - \varphi_1 -\psi_0) \| & \leq C \eps^\frac13 \| \pa_Y \varphi_1 \| + \frac{C}{\eps^\frac13} \big ( \| \psi_0 \| + \frac{1}{\eps^\frac13}  \| U_s \psi_0 \| \big ) \,, \label{est.prop.mOS.1'} \\
\alpha \| \Phi -\varphi_1-\psi_0 \| & \leq  C \alpha  \eps^\frac13 \| U_s'' \varphi_1 \| + C \alpha \| \psi_0 \| \,,\label{est.prop.mOS.2'} \\
\| (\pa_Y^2-\alpha^2) \big ( \Phi -\varphi_1-\psi_0 \big ) \| & \leq C \| \pa_Y \varphi_1 \|  + \frac{C}{\eps^\frac23}  \| \psi_0 \| \,.\label{est.prop.mOS.3'} 
\end{align}

\end{proposition}


\begin{proof} We apply the iteration argument, called the Rayleigh-Airy iteration. 
Let $\varphi_1\in H^2(\R_+)\cap H^1_0(\R_+)$ be the solution to the Rayleigh equation \eqref{eq.Ray}.
Then we have 
\begin{align}\label{proof.prop.mOS.-1} 
\langle U_s H_\alpha \varphi_1 - U_s'' \varphi_1, q\rangle  + i \eps \langle H_\alpha \varphi_1, (\pa_Y^2-\alpha^2)  q\rangle = \langle f, q\rangle + i \eps \langle H_\alpha \varphi_1, (\pa_Y^2-\alpha^2) q\rangle \,, \qquad q\in H^2 (\R_+)\,.
\end{align}
Note that the identity \eqref{proof.prop.mOS.-1} holds for any $q\in H^2(\R_+)$ rather than $q\in H^2 (\R_+)$ with $\pa_Y q(0)=0$.
To correct the error term $i \eps H_\alpha^2 \varphi_1$, we observe the identity
\begin{align*}
H_\alpha [\varphi] & = \Big ( (\pa_Y^2-\alpha^2) -\frac{U_s''}{U_s} \Big )  \varphi + \frac{U_s''}{U_s} \varphi  \\
& = \frac{1}{U_s} Ray [\varphi] + \frac{U_s''}{U_s} \varphi \,.
\end{align*} 
Hence we have from $Ray[\varphi_1] = f$,
\begin{align*}
\langle U_s H_\alpha \varphi_1 - U_s'' \varphi_1, q\rangle  + i \eps \langle H_\alpha \varphi_1, (\pa_Y^2-\alpha^2) q\rangle & = \langle f, q\rangle + i \eps \langle \frac{f}{U_s} + \frac{U_s''}{U_s} \varphi_1, (\pa_Y^2-\alpha^2)  q\rangle \,,\\
& \qquad  \qquad q\in H^2 (\R_+)\,.
\end{align*}
Our next task is to recover $(\pa_Y^2-\alpha^2)$ regularity in the error term.
To this end we observe the relation
\begin{align*}
[U_s, H_\alpha] h - U_s'' h = -2 \pa_Y (U_s' h )\,,
\end{align*}
and then formally we have 
\begin{align}\label{proof.prop.mOS.1.formal} 
OS[\varphi] = (\pa_Y^2-\alpha^2) \big ( U_s \varphi  + i \eps (\pa_Y^2-\alpha^2) \big ) \varphi - 2 \pa_Y (U_s' \varphi) = (\pa_Y^2-\alpha^2 ) Airy[\varphi] - 2 \pa_Y (U_s' \varphi) \,.
\end{align}
Thus we take $\psi_1\in H^2(\R_+) \cap H^1_0(\R_+)$ as the solution to the Airy equation
\begin{align*}
Airy[\psi_1]= - i \eps \Big ( \frac{f}{U_s}  +  \frac{U_s''}{U_s} \varphi_1\Big )\,,
\end{align*}
and then $\varphi_1 + \psi_1$ satisfies 
\begin{align}\label{proof.prop.mOS.-2} 
& \langle U_s H_\alpha (\varphi_1 + \psi_1) - U_s'' (\varphi_1+\psi_1), q\rangle  + i \eps \langle H_\alpha (\varphi_1+\psi_1), (\pa_Y^2-\alpha^2)  q\rangle \nonumber \\
& = \langle f, q\rangle + i \eps \langle \frac{f}{U_s} + \frac{U_s''}{U_s}\varphi_1, (\pa_Y^2-\alpha^2)  q\rangle \nonumber \\
& \qquad  +  \langle H_\alpha U_s \psi_1, q \rangle  + \langle -2 \pa_Y (U_s' \psi_1), q \rangle  + i \eps \langle H_\alpha \psi_1,(\pa_Y^2-\alpha^2)  q\rangle \nonumber \\
& = \langle f, q\rangle + i \eps \langle \frac{f}{U_s} + \frac{U_s''}{U_s}\varphi_1, (\pa_Y^2-\alpha^2)  q\rangle  \nonumber \\
& \qquad +  \langle U_s \psi_1, (\pa_Y^2-\alpha^2)  q \rangle  + i \eps \langle H_\alpha \psi_1, (\pa_Y^2-\alpha^2)  q\rangle + \langle -2 \pa_Y (U_s' \psi_1), q \rangle   \nonumber  \\
& = \langle f, q\rangle + \langle -2 \pa_Y (U_s' \psi_1), q \rangle \,, \qquad q\in H^2 (\R_+)\,.
\end{align}
Note that, due to the fact that $\pa_Y^k (U_s \psi_1)|_{Y=0}$ for $k=0,1$, the equality \eqref{proof.prop.mOS.-2} is valid for $q\in H^2(\R_+)$ rather than $q\in H^2 (\R_+)$ with $\pa_Y q(0)=0$. The new error term is $-2\pa_Y (U_s' \psi_1)$, but this is not compatible in solving the Rayleigh equation since $-2\pa_Y (U_s' \psi_1)/U_s$ does not belong to $L^2 (\R_+)$ in general. 
Hence, we next take $\tilde \psi_1\in H^2 (\R_+)$ as the solution to $Airy[\tilde \psi_1] = 2\pa_Y (U_s' \psi_1)$ under the Neumann boundary condition $\pa_Y \tilde \psi_1|_{Y=0}=0$, and set $\varphi_2$ as the solution to the Rayleigh equation $Ray[\varphi_2] = U_s \tilde \psi_1$ with $\varphi_2|_{Y=0}=0$.
Then $\varphi_2 \in  H^2(\R_+) \cap H^1_0 (\R_+)$ and from the formal identity 
\begin{align*}
OS [\varphi] & = \Big ( U_s + i\eps (\pa_Y^2-\alpha^2) \Big ) \Big ( (\pa_Y^2-\alpha^2) -\frac{U_s''}{U_s} \Big )  \varphi +i \eps (\pa_Y^2-\alpha^2) \frac{U_s''}{U_s} \varphi  \\
& = Airy \big [\frac{1}{U_s} Ray [\varphi] \big ] + i \eps  (\pa_Y^2-\alpha^2) \frac{U_s''}{U_s} \varphi \,,
\end{align*} 
we have 
\begin{align*}
OS[\varphi_2] & = Airy[\frac{1}{U_s} U_s \tilde \psi_1]+  i \eps  (\pa_Y^2-\alpha^2) \frac{U_s''}{U_s}  \varphi_2 = 2\pa_Y (U_s' \psi_1) +  i \eps  (\pa_Y^2-\alpha^2) \frac{U_s''}{U_s}  \varphi_2 \,.
\end{align*}
This formal identity is rigorously justified in the weak formulation as follows. Let $q\in H^2 (\R_+)$ with $\pa_Y q (0)=0$.
Then $\varphi_1+\psi_1+\varphi_2$ solves 
\begin{align*}
& \langle U_s H_\alpha (\varphi_1 + \psi_1 + \varphi_2) - U_s'' (\varphi_1+\psi_1+\varphi_2), q\rangle  + i \eps \langle H_\alpha (\varphi_1+\psi_1+ \varphi_2), (\pa_Y^2-\alpha^2)   q\rangle \\
& = \langle f, q\rangle + \langle -2 \pa_Y (U_s' \psi_1), q \rangle  + \langle U_s \tilde \psi_1 , q\rangle \\
& \quad +  i\eps \langle \frac{1}{U_s} \Big (  U_s (\pa_Y^2-\alpha^2)  - U_s'' \Big ) \varphi_2 , (\pa_Y^2-\alpha^2)   q\rangle + i\eps \langle \frac{U_s''}{U_s} \varphi_2,  (\pa_Y^2-\alpha^2)  q \rangle  \\
& = \langle f, q\rangle + \langle -2 \pa_Y (U_s' \psi_1), q \rangle  + \langle U_s \tilde \psi_1 , q\rangle  +  i\eps \langle \frac{1}{U_s} U_s \tilde \psi_1 , (\pa_Y^2-\alpha^2)  q\rangle + i\eps \langle \frac{U_s''}{U_s} \varphi_2, (\pa_Y^2-\alpha^2)  q \rangle  \\
& = \langle f, q\rangle + \langle -2 \pa_Y (U_s' \psi_1), q \rangle  + \langle Airy[\tilde \psi_1], q\rangle + i\eps \langle \frac{U_s''}{U_s} \varphi_2,  (\pa_Y^2-\alpha^2)  q \rangle \\
& \qquad   \qquad ({\rm since }~q\in H^2 (\R_+) ~{\rm with}~\pa_Y q (0)=0)\\
& = \langle f, q\rangle + \langle -2 \pa_Y (U_s' \psi_1), q \rangle + \langle 2 \pa_Y (U_s'\psi_1), q\rangle + \langle i \eps  \frac{U_s''}{U_s} \varphi_2  , (\pa_Y^2-\alpha^2)  q\rangle \\
& =  \langle f, q\rangle + i \eps \langle \frac{U_s''}{U_s} \varphi_2  , (\pa_Y^2-\alpha^2)  q\rangle \,, \qquad q\in H^2 (\R_+)~{\rm with}~\pa_Y q(0)=0 \,.
\end{align*}
The error term $i \eps  (\pa_Y^2-\alpha^2) \frac{U_s''}{U_s} \varphi_2$ (in the weak form) is then handled by solving the Airy equation with the source $-i\eps \frac{U_s''}{U_s} \varphi_2$ by using the formal relation \eqref{proof.prop.mOS.1.formal}, which creates the next error terms. 
We iterate this process, namely, for $k\geq 1$ we set

\

(i) $\varphi_{k+1}\in H^2 (\R_+)\cap H^1_0 (\R_+)$ as the solution to the Rayleigh equation $Ray[\varphi_{k+1}] = U_s \tilde \psi_{k}$, 

(ii) $\psi_{k+1} \in H^2 (\R_+)\cap H^1_0 (\R_+)$ as the solution to the Airy equation $Airy[\psi_{k+1}]= -i\eps \frac{U_s''}{U_s} \varphi_{k+1}$,

(iii) $\tilde \psi_{k+1}\in H^2(\R_+)$ as the solution to $Airy[\tilde \psi_{k+1}] = 2 \pa_Y (U_s' \psi_{k+1})$ under the Neumann boundary condition $\pa_Y \tilde \psi_{k+1}|_{Y=0}=0$.

\

\noindent Then $\Phi_{m} = \varphi_1 +  \psi_1 + \sum_{k=2}^{m} \varphi_k + \sum_{k=2}^m \psi_k$, $m \geq 2$, solves 
\begin{align}\label{proof.prop.mOS.2} 
\begin{split}
\langle U_s H_\alpha \Phi_m - U_s'' \Phi_m, q\rangle  + i \eps \langle H_\alpha \Phi_m, (\pa_Y^2-\alpha^2) q\rangle  
& =  \langle f, q\rangle +  \langle -2 \pa_Y ( U_s' \psi_m) , q\rangle\,,\\
&  \qquad q\in H^2 (\R_+) ~{\rm with}~\pa_Y q(0)=0\,.
\end{split}
\end{align}
Our next aim is to show that $\Phi_m$ converges in $H^2(\R_+)\cap H^1_0(\R_+)$. Then the limit $\Phi=\displaystyle \lim_{m\rightarrow \infty}\Phi_m$ solves the weak formulation \eqref{eq.weak.mOS}, and then the regularity $H_\alpha \Phi\in H^2 (\R_+)$ is recovered from the weak formulation by regarding the term $U_s H_\alpha \Phi-U_s'' \Phi\in L^2 (\R_+)$ as the source term; indeed, \eqref{eq.weak.mOS} implies that the limit $H_\alpha \Phi$ is the very weak solution to the Poisson equation $i\eps (\pa_Y^2-\alpha^2) H_\alpha \Phi=  - U_s H_\alpha \Phi + U_s'' \Phi + f\in L^2 (\R_+)$ subject to the zero Neumann boundary condition.
{
\begin{remark}
{\rm The additional boundary condition $\pa_Y H_\alpha \Phi\vert_{h=0}{=0}$, that is derived and understood in a weak sense through the variational formulation \eqref{eq.weak.mOS}, can also be recovered at a formal level by manipulating the strong formulation  of the equations. More precisely, one can derive the identity  $\pa_Y H_\alpha (\varphi_k + \psi_k)\vert_{Y=0}{=0}$ for all $k \ge 1$ as follows. For $k \ge 2$, we differentiate the Airy equation satisfied by $\psi_{k}$ and take its trace to find 
$$ i \eps \pa_Y H_\alpha \psi_k\vert_{Y=0} = - i \eps  \pa_Y \bigl( \frac{U''_s}{U_s} \varphi_{k}\bigr)\vert_{Y=0}. 
$$   
We have used here that $\pa_Y (U_s \psi_{k})\vert_{Y=0} = 0$ due to the Dirichlet boundary condition on $\psi_{k}$. 
Similarly, we divide the Rayleigh equation satisfied by $\varphi_{k}$ by $U_s$, differentiate it and take its trace to find (using the Neumann condition satisfied by $\tilde \psi_{k-1}$):  
$$ \pa_Y H_\alpha \varphi_k\vert_{Y=0} =  \pa_Y \bigl( \frac{U''_s}{U_s} \varphi_{k}\bigr)\vert_{Y=0}. $$
The identity $\pa_Y H_\alpha (\varphi_k + \psi_k)\vert_{Y=0}{=0}$ follows for all $k \ge 2$. The same result holds for $k=1$ taking into account the additional source term $f$. }
\end{remark}
}

\mspace
To show the convergence we divide into two cases $\alpha\geq 1$ and $0<\alpha{\le} 1$.

\noindent {\bf (1) The case $\alpha\geq 1$:}  We have from (i) above and \eqref{est.prop.Ray.1}, {for $k \ge 1$}
\begin{align}\label{proof.prop.mOS.-0} 
\| \pa_Y \varphi_{k+1} \|  + \alpha \| \varphi_{k+1} \| & \leq C \| Y \tilde \psi_k \|\,,
\end{align}
while from \eqref{est.prop.Ray.2},
\begin{align}\label{proof.prop.mOS.3} 
\| (\pa_Y^2-\alpha^2) \varphi_{k+1} \| \leq C \| Y \tilde \psi_k \| + \| \tilde \psi_k \| \,.
\end{align}
Here $C$ is independent of $\alpha\geq 1$. 
Then Proposition \ref{prop.Airy.N} implies, {for $k \ge 1$},   
\begin{align}
\begin{split}
\| Y \tilde \psi_k \| & \leq \frac{C}{\eps^\frac23} \| Y U_s' \psi_k \| + \frac{C}{\eps^\frac13} \|  U_s' \psi_k \| \leq \frac{C}{\eps^\frac23} \| U_s \psi_k \| + \frac{C}{\eps^\frac13} \| \psi_k \| \,, \\
\|\tilde \psi_k\|  & \leq \frac{C}{\eps^\frac23} \| U_s' \psi_{k} \| \leq \frac{C}{\eps^\frac23} \| \psi_k \| \label{proof.prop.mOS.4} \,.
\end{split}
\end{align}
Here $C$ is independent of $\alpha$. 
This gives 
\begin{align}\label{proof.prop.mOS.1}
\| \pa_Y \varphi_{k+1} \|  + \alpha \| \varphi_{k+1} \| & \leq \frac{C}{\eps^\frac23} \| U_s \psi_k \| +  \frac{C}{\eps^\frac13} \| \psi_k \|\,, \qquad \| (\pa_Y^2-\alpha^2) \varphi_{k+1} \|  \leq \frac{C}{\eps^\frac23} \| \psi_k \| \,.
\end{align}
Next Proposition \ref{prop.Airy} shows, for $k\geq 2$,
\begin{align}\label{proof.prop.mOS.6} 
\| U_s \psi_k \| + \eps^\frac13 \| \psi_k \| + \eps^\frac23 \| \pa_Y \psi_k \| + \eps \| (\pa_Y^2-\alpha^2) \psi_k \| & \leq C \eps \| \frac{U_s''}{U_s} \varphi_k \| \leq C \eps \| \pa_Y \varphi_k \|\,.
\end{align}
Hence \eqref{proof.prop.mOS.1} and \eqref{proof.prop.mOS.6} yield
\begin{align}\label{proof.prop.mOS.7} 
\| \pa_Y \varphi_{k+1} \| +\alpha \| \varphi_{k+1} \| \leq C \eps^\frac13 \| \pa_Y \varphi_k \|\,,
\end{align}
and thus, $\sum_{k=2}^\infty \varphi_k$ converges in $H^1(\R_+)$ if $\eps>0$ is small enough and satisfies 
\begin{align}\label{proof.prop.mOS.8} 
\sum_{k=2}^\infty \| \pa_Y \varphi_k \| + \alpha \sum_{k=2}^\infty \| \varphi_k \| \leq C  \| \pa_Y \varphi_2 \| + \alpha \| \varphi_2 \|  \leq \frac{C}{\eps^\frac23} \| U_s \psi_1 \| + \frac{C}{\eps^\frac13} \| \psi_1 \| \,.
\end{align}
Here we have used \eqref{proof.prop.mOS.1} in the last line.
Then \eqref{proof.prop.mOS.1} and \eqref{proof.prop.mOS.6} show that $\sum_{k=2}^\infty \varphi_k$ converges in $H^2(\R_+)$ and satisfies 
\begin{align}\label{proof.prop.mOS.8'} 
\sum_{k=2}^\infty \| (\pa_Y^2-\alpha^2) \varphi_k \| \leq C  \| \pa_Y \varphi_2 \| + \frac{C}{\eps^\frac23} \| \psi_1 \| \leq \frac{C}{\eps^\frac23} \| \psi_1\| \,.
\end{align}
As for the convergence and the estimate of $\sum_{k=2}^\infty \psi_k$, we have from \eqref{proof.prop.mOS.6} and \eqref{proof.prop.mOS.8},
\begin{align}\label{proof.prop.mOS.9} 
\sum_{k=2}^\infty \| \pa_Y  \psi_k \| \leq C \eps^\frac13  \sum_{k=2}^\infty \| \pa_Y \varphi_k \| \leq \frac{C}{\eps^\frac13} \| U_s \psi_1 \| + C  \| \psi_1 \|\,,
\end{align}
and similarly,
\begin{align}
\alpha \sum_{k=2}^\infty \| \psi_k \| & \leq C\alpha \eps^\frac23 \sum_{k=2}^\infty \| \pa_Y \varphi_k \| \leq C \alpha (\| U_s \psi _1 \| + \eps^\frac13 \|\psi_1 \|)\,, \label{proof.prop.mOS.10} \\
\sum_{k=2}^\infty \| (\pa_Y^2 - \alpha^2)  \psi_k \| & \leq C \sum_{k=2}^\infty \| \pa_Y \varphi_k \| \leq \frac{C}{\eps^\frac23} \| U_s \psi_1 \| + \frac{C}{\eps^\frac13} \| \psi_1 \|\,. \label{proof.prop.mOS.11} 
\end{align}  
Let us recall that $\psi_1$ is decomposed as $\psi_1=\psi_0 + \psi_{1,1}$, where $\psi_0, \psi_{1,1}\in H^2 (\R_+) \cap H^1_0 (\R_+)$ are respectively the solutions to 
\begin{align*}
Airy [\psi_0] = -i\eps \frac{f}{U_s}\,, \qquad Airy[\psi_{1,1}] = -i\eps \frac{U_s''}{U_s} \varphi_1\,.
\end{align*}
Then $\psi_{1,1}$ satisfies in virtue of {\eqref{est.prop.Airy.0} in} Proposition \ref{prop.Airy},
\begin{align}\label{proof.prop.mOS.9'} 
\| U_s \psi_{1,1} \| + \eps^\frac13 \| \psi_{1,1} \| + \eps^\frac23 \| \pa_Y \psi_{1,1} \|  + {\eps}\| (\pa_Y^2-\alpha^2) \psi_{1,1} \| {\le C \eps \| \frac{U_s''}{U_s} \varphi_1 \|} \leq C \eps \| \pa_Y \varphi_1 \|\,,
\end{align}
or we also have {from \eqref{est.prop.Airy.4} and \eqref{est.prop.Airy.6},}
\begin{align}\label{proof.prop.mOS.9''} 
\eps^\frac13 \| U_s \psi_{1,1} \| + \eps^\frac23 \| \psi_{1,1} \| + \eps \| \pa_Y \psi_{1,1} \|  \leq C \eps \| U_s'' \varphi_1 \|\,.
\end{align}
{Note that \eqref{proof.prop.mOS.9'} and \eqref{proof.prop.mOS.9''} are valid for all $\alpha>0$.}
Collecting \eqref{proof.prop.mOS.8}, \eqref{proof.prop.mOS.8'}, \eqref{proof.prop.mOS.9}, \eqref{proof.prop.mOS.10}, \eqref{proof.prop.mOS.11}, and \eqref{proof.prop.mOS.9'}, we obtain 
\begin{align}\label{proof.prop.mOS.12'} 
\begin{split}
\| \pa_Y \big ( \Phi -\varphi_1 - \psi_0 \big) \| & \leq  C \eps^\frac13 \| \pa_Y \varphi_1 \| + \frac{C}{\eps^\frac13} \| \psi_0 \| + \frac{C}{\eps^\frac23} \| U_s \psi_0 \| \,,\\
\alpha \| \Phi -\varphi_1-\psi_0  \|  & \leq C \eps^\frac13 ( 1+ \alpha \eps^\frac13 ) \| \pa_Y \varphi_1 \| +  C (\alpha \eps^\frac13 +\frac{1}{\eps^\frac13} ) \big ( \| \psi_0 \| + \frac{1}{\eps^\frac13} \| U_s \psi_0 \| \big ) \,,\\
\| (\pa_Y^2-\alpha^2) \big ( \Phi -\varphi_1 -\psi_0 \big ) \| & \leq  C \| \pa_Y \varphi_1 \| + \frac{C}{\eps^\frac23} \| \psi_0 \|\,.
\end{split}
\end{align}
The proof for the case $\alpha\geq 1$ is complete.

\noindent {\bf (2) The case $0<\alpha\leq 1$:} We first observe from Proposition \ref{prop.Airy.N} that 
\begin{align*}
\| (1+Y)^2 U_s \tilde \psi_k \| & \leq C \| U_s \tilde \psi_k \| + C \| Y^2 U_s \tilde \psi_k \| \\
& \leq \frac{C}{\eps^\frac13} \| U_s' \psi_k \| + \frac{C}{\eps^\frac23} \| U_s U_s' \psi_k \| \\
& \quad + \frac{C}{\eps^\frac23} \| U_s Y^2 U_s' \psi_k \|   + \frac{C}{\eps^\frac13} \| U_s Y U_s' \psi_k \| + C \| Y U_s' \psi_k \| + C \eps^\frac13 \| U_s' \psi_k \| \\
& \leq \frac{C}{\eps^\frac13} \big ( \| \psi_k \| + \frac{1}{\eps^\frac13} \| U_s \psi_k \| \big ) \,,
\end{align*}
and 
\begin{align*}
|\int_0^\infty U_s \tilde \psi_k \, d Y| & \leq C  \alpha^2 \Big ( \eps^\frac16 \| Y U_s' \psi_k \| + \eps^\frac12 \| U_s' \psi_k \| \Big )\,.
\end{align*}
Thus Proposition \ref{prop.Ray} and Proposition \ref{prop.pre.op} (1) yield, {for $k \ge 1$,}
\begin{align}\label{proof.prop.mOS.13} 
\| \pa_Y \varphi_{k+1} \| & \leq C \| (1+Y)^2 U_s \tilde \psi_k \| + \frac{C}{\alpha} \big | \int_0^\infty U_s \tilde \psi_k \, d Y \big | \nonumber \\
& \leq \frac{C}{\eps^\frac13}  \big ( \| \psi_k \| + \frac{1}{\eps^\frac13} \| U_s \psi_k \| \big ) \,.  
\end{align}
Then from \eqref{proof.prop.mOS.6}, which is valid also in the case $0<\alpha \leq 1$ with $k\geq 2$,  we have { for $k\geq 2$},
\begin{align}\label{proof.prop.mOS.14} 
\| \pa_Y \varphi_{k+1} \| & \leq C \eps^\frac13 \| \frac{U_s''}{U_s} \varphi_k \| \leq C \eps^\frac13 \| \pa_Y \varphi_k \|\,. 
\end{align}
Similarly, {for all $k \ge 1$,} 
\begin{align*}
\| (\pa_Y^2-\alpha^2) \varphi_{k+1} \| & \leq C \| (1+Y)^2 U_s \tilde \psi_k \| + \| \tilde \psi_k \| +  \frac{C}{\alpha} \big | \int_0^\infty U_s \tilde \psi_k \, d Y \big | \\
& \leq \frac{C}{\eps^\frac13} \| \psi_k \| + \frac{C}{\eps^\frac23} \| U_s \psi_k \| + \frac{C}{\eps^\frac23} \| \psi_k \| \\
& { \leq \frac{C}{\eps^\frac23} \| \psi_k \|\,,}
\end{align*}
{which implies when $k\ge2$:}
\begin{align}\label{proof.prop.mOS.16} 
\| (\pa_Y^2-\alpha^2) \varphi_{k+1} \|  & \leq C \| \pa_Y \varphi_k \|  \,.
\end{align}
Next we observe from \eqref{est.prop.Airy.N.6} that, for $\sigma [U_s \tilde \psi_k] (Y) = \int_Y^\infty U_s \tilde \psi_k \, d Y_1$, 
\begin{align*}
\| (1+Y) \sigma [U_s \tilde \psi_k] \| \leq C \| (1+Y)^2 U_s' \psi_k \| \leq C \| \psi_k \|\,.
\end{align*}
Thus we have again from Proposition \ref{prop.Ray}, {for all $k \ge 1$},
\begin{align*}
\alpha \| \varphi_{k+1} \| & \leq  C \alpha \| (1+Y)  \sigma [U_s \tilde \psi_k] \| + \frac{C}{\alpha^\frac12} \big | \int_0^\infty U_s \tilde \psi_k \, d Y \big | \nonumber \\
& \leq C \alpha \| \psi_k \| \nonumber 
\end{align*}
{and for $k \ge 2$, by \eqref{proof.prop.mOS.6}},
\begin{align}\label{proof.prop.mOS.15} 
\alpha \| \varphi_{k+1} \| & \leq C \alpha \epsilon^\frac23 \| \pa_Y \varphi_k \|
\end{align}
This ensure the convergence of $\sum_{k=2}^\infty \varphi_k$ in $H^2(\R_+)$ when $\eps$ is small enough, and we have 
\begin{align}\label{proof.prop.mOS.17} 
\begin{split}
\sum_{k=2}^\infty \| \pa_Y  \varphi_k \|  & \leq C \| \pa_Y \varphi_2 \| \leq \frac{C}{\eps^\frac13} \big ( \| \psi_1 \| + \frac{1}{\eps^\frac13} \| U_s \psi_1 \| \big ) \,,\\
 \sum_{k=2}^\infty \alpha \| \varphi_k \| & \leq \alpha \| \varphi_2 \| +  C \alpha \eps^\frac23 \| \pa_Y \varphi_2 \| \leq C \alpha \| \psi_1 \| \,,\\
\sum_{k=2}^\infty \| (\pa_Y^2-\alpha^2)  \varphi_k \| & \leq \| (\pa_Y^2-\alpha^2) \varphi_2\| + C \sum_{k=2}^\infty \| \pa_Y \varphi_k \| \leq \frac{C}{\eps^\frac23} \| \psi_1 \|\,.
\end{split}
\end{align}
Then, by applying Proposition \ref{prop.Airy}, $\sum_{k=2}^\infty \psi_k$ converges in $H^2(\R_+)$ as  in the case $\alpha\geq 1$, and we have 
\begin{align}\label{proof.prop.mOS.17'} 
\begin{split}
\sum_{k=2}^\infty \| \pa_Y \psi_k \| \leq C \eps^\frac13 \sum_{k=2}^\infty \| \pa_Y \varphi_k \| \leq C \big ( \| \psi_1 \| + \frac{1}{\eps^\frac13}  \| U_s \psi_1 \| \big )\,,\\
\sum_{k=2}^\infty \alpha \| \psi_k \| \leq C \alpha \eps^\frac23 \sum_{k=2}^\infty \| \pa_Y \varphi_k \| \leq C \alpha \eps^\frac13 \big ( \| \psi_1 \| + \frac{1}{\eps^\frac13}  \| U_s \psi_1 \| \big )\,,\\
\sum_{k=2}^\infty \| (\pa_Y^2-\alpha^2) \psi_k \| \leq  C \sum_{k=2}^\infty \| \pa_Y \varphi_k \| \leq \frac{C}{\eps^\frac13} \big ( \| \psi_1 \| + \frac{C}{\eps^\frac13}\| U_s \psi_1 \| \big )\,.
\end{split}
\end{align}
Recall that $\psi_1$ is decomposed as $\psi_1=\psi_0+\psi_{1,1}$ as in the case $\alpha \geq 1$.
Hence, from \eqref{proof.prop.mOS.17}, \eqref{proof.prop.mOS.17'},  and the estimates for $\psi_{1,1}$ in \eqref{proof.prop.mOS.9''} {(which is valid also for the case $0<\alpha \le 1$)}, we have 
\begin{align*}
\begin{split}
\| \pa_Y (\Phi - \varphi_1-\psi_0 ) \| & \leq C \eps^\frac13 \| \pa_Y \varphi_1 \| + \frac{C}{\eps^\frac13} \big ( \| \psi_0 \| + \frac{1}{\eps^\frac13}  \| U_s \psi_0 \| \big ) \,,\\
\alpha \| \Phi -\varphi_1-\psi_0 \| & \leq C \alpha \eps^\frac13  \| U_s'' \varphi_1 \| + C\alpha \| \psi_0 \|  \,,\\
\| (\pa_Y^2-\alpha^2) \big ( \Phi  - \varphi_1-\psi_0 \big )  \| & \leq  C \| \pa_Y \varphi_1 \|  + \frac{C}{\eps^\frac23} \| \psi_0 \|\,.
\end{split}
\end{align*}
The proof is complete.
\end{proof}

Proposition \ref{prop.Airy}  gives  two kinds of the estimates for $\psi_0$: 
\begin{align*}
\| U_s \psi_0 \| + \eps^\frac13 \|\psi_0 \| + \eps^\frac23 \| \pa_Y \psi _0 \| & \leq C \eps \| \frac{f}{U_s}  \| \,,\\
\eps^\frac13 \| U_s \psi_0 \| + \eps^\frac23 \|\psi_0 \| + \eps \| \pa_Y \psi _0 \| & \leq C \eps \| \frac{Y}{U_s}  f\| \,.
\end{align*}
Then, by applying Proposition \ref{prop.Ray} {and} Proposition \ref{prop.pre.op} (1) for $\varphi_1$, Proposition \ref{prop.mOS} finally yields the following corollaries in the case $\alpha \eps^\frac13\leq 1$.
\begin{corollary}\label{cor.prop.mOS.1} Let $\eps_1>0$ be the number in Proposition \ref{prop.mOS}, and let $0<\eps\leq \eps_1$ and $0<\alpha \eps^\frac13 \leq 1$. Let $f/U_s\in L^2 (\R_+)$.  Then the solution $\Phi$ in Proposition \ref{eq.weak.mOS} satisfies the following 
estimates.

\noindent {\rm (i)} when $\alpha \geq 1$,
\begin{align}
\| \pa_Y \Phi \| +\alpha \| \Phi \| & \leq C \min \{ \|\frac{Y}{U_s} f\|, ~ \frac{1}{\alpha} \| \frac{f}{U_s} \| \}\,, \label{est.cor.prop.mOS.1.1}\\
\| (\pa_Y^2-\alpha^2) \Phi \| & \leq C \min \{ \|\frac{Y}{U_s} f\|, ~ \frac{1}{\alpha} \| \frac{f}{U_s} \| \} + C \| \frac{f}{U_s} \|\,.\label{est.cor.prop.mOS.1.2}
\end{align}

\noindent {\rm (ii)} when $0<\alpha\leq 1$, if $(1+Y)^2 f\in L^2 (\R_+)$ in addition,
\begin{align}
\alpha \| \Phi \| & \leq C \alpha \| (1+Y) \sigma[f] \| + C \alpha\eps^\frac13 \| (1+Y) f \|  + \frac{C}{\alpha^\frac12} \, | \int_0^\infty f \, d Y | \,, \label{est.cor.prop.mOS.1.3}\\
\| \pa_Y \Phi \|  & \leq C \| (1+Y)^2 f\| + \frac{C}{\alpha} \, | \int_0^\infty f \, d Y | \,,\label{est.cor.prop.mOS.1.4}\\
\| (\pa_Y^2-\alpha^2) \Phi \| & \leq C \| (1+Y)^2 f\| + \frac{C}{\alpha} \, | \int_0^\infty f \, d Y |+ C \| \frac{f}{U_s} \|\,.\label{est.cor.prop.mOS.1.5}
\end{align}
Here $\sigma[f](Y)= \int_Y^\infty f \, d Y_1$.
\end{corollary}

In Corollary \ref{cor.prop.mOS.1} the $H^2$ norm of the solution $\Phi$ is estimated uniformly in the small number $\eps>0$, but under the condition of $f/U_s\in L^2 (\R_+)$ which implicitly imposes that $f$ vanishes on the boundary. By using the weak formulation \eqref{eq.weak.mOS} and  the standard density argument, we have another $\epsilon$-dependent bound for the $H^2$ norm of $\Phi$ when $f$ does not necessarily vanish on the boundary. Precisely, the result is stated as follows.

\begin{corollary}\label{cor.prop.mOS.2} Let $\eps_1>0$ be the number in Proposition \ref{prop.mOS}, and let $0<\eps\leq \eps_1$ and $0<\alpha \eps^\frac13 \leq 1$. Let $Yf/U_s \in L^2 (\R_+)$.  Then there exists a solution $\Phi\in H^4(\R_+)\cap H^1_0 (\R_+)$ to {\eqref{eq.mOS}} satisfying the following estimates.

\noindent {\rm (i)} when $\alpha \geq 1$,
\begin{align}
\| \pa_Y \Phi \| +\alpha \| \Phi \| & \leq C \|\frac{Y}{U_s} f\| \,, \label{est.cor.prop.mOS.2.1}\\
\| (\pa_Y^2-\alpha^2) \Phi \| & \leq \frac{C}{\eps^\frac13} \Big (  \|\frac{Y}{U_s} f\| +  \| f \| \Big )\,.\label{est.cor.prop.mOS.2.2}
\end{align}

\noindent {\rm (ii)} when $0<\alpha\leq 1$, if $(1+Y)^2 f\in L^2 (\R_+)$ in addition, 
\begin{align}
\alpha \| \Phi \| & \leq C \alpha \| (1+Y) \sigma [f] \| + C \alpha \eps^\frac13 \| (1+Y) f \| + \frac{C}{\alpha^\frac12} \, | \int_0^\infty f \, d Y |\,, \label{est.cor.prop.mOS.2.3}\\
\| \pa_Y \Phi \| & \leq C \| (1+Y)^2 f\| + \frac{C}{\alpha} \, | \int_0^\infty f \, d Y |\,, \label{est.cor.prop.mOS.2.4}\\
\| (\pa_Y^2-\alpha^2) \Phi \| & \leq \frac{C}{\eps^\frac13} \Big  (\| (1+Y)^2 f\| + \frac{C}{\alpha} \, | \int_0^\infty f \, d Y | \Big ) \,.\label{est.cor.prop.mOS.2.5}
\end{align}
Here $\sigma[f](Y)= \int_Y^\infty f \, d Y_1$.
\end{corollary}

\mspace

\begin{remark}{\rm Corollary \ref{cor.prop.mOS.2} implies that $\varphi_1+\psi_0\in H^2 (\R_+)$ even when $f$ does not vanish on the boundary, though neither $\varphi_1$ or $\psi_0$ belongs to $H^2(\R_+)$ for such a case. }
\end{remark}

\begin{proof} Thanks to Corollary \ref{cor.prop.mOS.1},  it suffices to show the estimates \eqref{est.cor.prop.mOS.2.2} and \eqref{est.cor.prop.mOS.2.5} for the solution $\Phi$ obtained  in Corollary \ref{cor.prop.mOS.1}. Then Corollary \ref{cor.prop.mOS.2} follows from the density argument.
Let $\Phi\in H^4 (\R_+) \cap H^1_0 (\R_+)$, $\pa_Y H_\alpha \Phi |_{Y=0} =0$, be the solution to \eqref{eq.weak.mOS} obtained in Corollary \ref{cor.prop.mOS.1}. Then $\Psi=H_\alpha\Phi$ solves the Airy equation $Airy [\Psi] = U_s'' \Phi + f\in L^2 (\R_+)$ subject to the Neumann boundary condition $\pa_Y \Psi|_{Y=0} =0$. Thus we have from the integration by parts,
\begin{align}
\| \sqrt{U_s} \Psi \|^2  - i \eps \big ( \|\pa_Y \Psi \|^2 + \alpha^2 \| \Psi \|^2 \big )  =  \langle U_s'' \Phi + f, \Psi \rangle\,.
\end{align}
Hence the interpolation inequality 
\begin{align*}
\| \Psi \|^2 \leq C \| \sqrt{U_s} \Psi \|^\frac43 \| \pa_Y \Psi \|^\frac23 + C \| \sqrt{U_s} \Psi \|^2
\end{align*}
implies 
\begin{align}
\| \Psi \| \leq \frac{C}{\eps^\frac13} \| U_s'' \Phi + f \| 
\end{align}
as in the proof of Proposition \ref{prop.Airy}. Since $\|U_s'' \Phi \|\leq C \| \pa_Y \Phi \|$ by the Hardy inequality, estimates \eqref{est.cor.prop.mOS.2.2} and \eqref{est.cor.prop.mOS.2.5} follow from the estimate of $\|\pa_Y \Phi \|$ in Corollary \ref{cor.prop.mOS.1}. The proof is complete.
\end{proof}

\subsection{Construction of a boundary corrector}\label{subsec.corrector}

Let $\Phi_{slip}[f]$ be the solution to the modified Orr-Sommerfeld equation \eqref{eq.mOS} obtained by Proposition \ref{prop.mOS}.
Then the solution to the original problem \eqref{eq.OS} is obtained by solving the equation
\begin{align}\label{eq.OS_b}
\begin{cases}
& OS[\phi] =   0 \,, \qquad Y>0\,,\\
& \phi|_{Y=0} = 0\,, \quad \pa_Y \phi|_{Y=0} = - \pa_Y \Phi_{slip} [f]|_{Y=0} \,.
\end{cases}
\end{align}
{\begin{remark}
By interpolation of \eqref{est.cor.prop.mOS.2.1} and \eqref{est.cor.prop.mOS.2.2}, or \eqref{est.cor.prop.mOS.2.4} and \eqref{est.cor.prop.mOS.2.5}, one has
\begin{align}
\label{bound.derivative.Phi.slip1}
|\pa_Y \Phi_{slip} [f]|_{Y=0}| & \le \frac{C}{\eps^{\frac16}} \left( \| \frac{Y}{U_s}f\| + \|f\|\right) & \quad \text{if } \: \alpha \ge 1, \\
\label{bound.derivative.Phi.slip2}
|\pa_Y \Phi_{slip} [f]|_{Y=0}| & \le \frac{C}{\eps^{\frac16}} \left( \|(1+Y)^2f\| + \frac{1}{\alpha} | \int_0^\infty f \, d Y|  \right) & \quad \text{if } \:  0 < \alpha \le 1.
\end{align}
\end{remark}}

\subsection{Slow mode}\label{subsec.slow}

In this subsection we construct a solution $\phi_{slow}$ to the equation 
\begin{align}\label{eq.OS_slow}
\begin{cases}
& OS [\phi] =   0 \,, \qquad Y>0\,,\\
& \phi|_{Y=0} = 1 + {\rm small~order}\,,
\end{cases}
\end{align}
decaying as $Y\rightarrow \infty$ around the Rayleigh solution $\varphi_{slow,Ray}$ obtained in Corollary \ref{cor.prop.slow.Ray} and Proposition \ref{prop.slow.Ray'}.  This solution is called the slow mode, {and will take the form  
 $\phi_{slow} = \varphi_{slow,Ray} + \tilde \phi_{slow}$.}

For the moment we assume that $\frac{U_s''}{U_s}\in L^2 (\R_+)$, which ensures the $H^2$ regularity of $\varphi_{slow,Ray}$ and justifies the formal computation in various steps.
Later we shall recover the $H^2$ regularity of $\phi_{slow}$ which does not depend on the condition $\frac{U_s''}{U_s}\in L^2 (\R_+)$. Then the standard limiting process gives the result for the general case without the condition $\frac{U_s''}{U_s}\in L^2 (\R_+)$.
Firstly we observe that $\varphi_{slow,Ray}$ satisfies the identity 
\begin{align*}
& \langle Ray[\varphi_{slow,Ray}], q \rangle + i \eps \langle (\pa_Y^2-\alpha^2) \varphi_{slow,Ray}, (\pa_Y^2-\alpha^2) q\rangle \\
& = i\eps \langle (\pa_Y^2-\alpha^2) \varphi_{slow,Ray}, (\pa_Y^2-\alpha^2) q \rangle \\
& =  i\eps \langle \frac{1}{U_s} Ray[\varphi_{slow,Ray}] +\frac{U_s''}{U_s} \varphi_{slow,Ray}, (\pa_Y^2-\alpha^2) q \rangle\\
& = i\eps \langle \frac{U_s''}{U_s} \varphi_{slow,Ray}, (\pa_Y^2-\alpha^2) q \rangle\,, \qquad q\in H^2 (\R_+)\,.
\end{align*}
In this computation we are using the condition $\frac{U_s''}{U_s}\in L^2 (\R_+)$ so that each term makes sense for any $q\in H^2 (\R_+)$.
By the ansatz $\phi_{slow} = \varphi_{slow,Ray}+\tilde \phi_{slow}$ we shall construct $\tilde \phi_{slow}\in H^2 (\R_+) \cap H^1_0 (\R_+)$ as the solution to 
\begin{align*}
& \langle Ray[\tilde \phi_{slow}], q \rangle + i \eps \langle (\pa_Y^2-\alpha^2) \tilde \phi_{slow}, (\pa_Y^2-\alpha^2) q \rangle  = -i\eps \langle \frac{U_s''}{U_s} \varphi_{slow,Ray}, (\pa_Y^2-\alpha^2) q \rangle \,,\\
& \qquad  \qquad q\in H^2 (\R_+)\,, \quad \pa_Y q (0) =0\,.
\end{align*}
To this end, by using $OS = (\pa_Y^2-\alpha^2) Airy -2\pa_Y (U_s' \cdot)$ we take ${\psi_{slow}}\in H^2 (\R_+) \cap H^1_0 (\R_+)$ as the solution to 
\begin{align*}
Airy [\psi_{slow}] = i\eps \frac{U_s''}{U_s} \varphi_{slow,Ray}{\,.}
\end{align*}
Then we see 
\begin{align*}
& \langle Ray[\psi_{slow}], q \rangle + i \eps \langle (\pa_Y^2-\alpha^2)\psi_{slow}, (\pa_Y^2-\alpha^2) q \rangle \\
&  =  \langle Airy[\psi_{slow}], (\pa_Y^2-\alpha^2) q\rangle  -  2 \langle \pa_Y (U_s' \psi_{slow}),  q \rangle \\
& =  i\eps \langle \frac{U_s''}{U_s} \varphi_{slow,Ray}, (\pa_Y^2-\alpha^2) q \rangle  -  2 \langle \pa_Y (U_s' \psi_{slow}),  q \rangle  \,, \qquad q\in H^2 (\R_+)\,.
\end{align*}
Finally we take ${\Phi_{slow}}\in H^4 (\R_+) \cap H^1_0 (\R_+)$ as the solution to \eqref{eq.mOS} with the source term $2\pa_Y (U_s' \psi_{slow} )$, which is constructed in Corollary \ref{cor.prop.mOS.2}. Then $\tilde \phi_{slow}$ is constructed in the form $\tilde \phi_{slow} = \psi_{slow} + \Phi_{slow}\in H^2 (\R_+) \cap H^1_0 (\R_+)$.
Moreover, the function 
\begin{align*}
\phi_{slow} = \varphi_{slow,Ray} + \psi_{slow} + \Phi_{slow}\in H^2 (\R_+)
\end{align*}
satisfies $OS[\phi_{slow}]=0$ in the weak form:
\begin{align}\label{proof.slow.1}
\begin{split}
& \langle U_s (\pa_Y^2-\alpha^2) \phi_{slow} - U_s'' \phi_{slow}, q \rangle + i\eps \langle (\pa_Y^2-\alpha^2) \phi_{slow}, (\pa_Y^2-\alpha^2)  q\rangle =0 \,,\\
& \qquad  \qquad q\in H^2 (\R_+)\,, \quad \pa_Y q (0)=0\,.
\end{split}
\end{align}
The requirement $\pa_Y q(0)=0$ for the test function $q$ is due to the weak formulation of $\Phi_{slow}$ as in \eqref{eq.weak.mOS}.
The identity \eqref{proof.slow.1} implies that $\Psi_{slow}=(\pa_Y^2-\alpha^2) \phi_{slow}$ is the very weak solution to the Poisson equation $i\eps (\pa_Y^2-\alpha^2) \Psi_{slow} = - U_s \Psi_{slow} + U_s'' \phi_{slow}$ subject to the Neumann boundary condition $\pa_Y \Psi_{slow}|_{Y=0} =0$.
Hence $\Psi_{slow}$ belongs to $H^2(\R_+)$ by the elliptic regularity and satisfies the Airy equation $Airy[\Psi_{slow}]=U_s'' \phi_{slow}$ with the zero Neumann boundary condition.
Then we have 
\begin{align}\label{proof.slow.2}
\langle U_s \Psi_{slow} - U_s'' \phi_{slow}, \Psi_{slow}\rangle  -i\eps\big ( \| \pa_Y \Psi_{slow} \|^2 + \alpha^2 \| \Psi_{slow}\|^2 \big ) = 0\,,
\end{align}
which yields, as in the proof of Corollary \ref{cor.prop.mOS.2}, 
\begin{align}\label{proof.slow.3}
\| (\pa_Y^2-\alpha^2) \phi_{slow} \| = \| \Psi_{slow} \| \leq \frac{C}{\eps^\frac13} \| U_s'' \phi_{slow} \|\,.
\end{align} 
Hence, the $H^2$ regularity of $\phi_{slow}$ is estimated in terms of $\|U_s'' \phi_{slow}\|$, for which the condition $\frac{U_s''}{U_s}\in L^2 (\R_+)$ is not required.
Now it suffices to establish the estimates of $\phi_{slow}=\varphi_{slow,Ray}+ \tilde \phi_{slow} = \varphi_{slow,Ray}+ \psi_{slow} + \Phi_{slow}$, which will be considered below depending on the two cases {$0<\alpha \leq 1$ and $\alpha \geq 1$}.

\subsubsection{Estimates in the case {$0<\alpha \leq 1$}}\label{subsubsec.slow.1}
Let us estimate $\tilde \phi_{slow}$, which is equal to $\psi_{slow} + \Phi_{slow}\in H^2(\R_+)\cap H^1_0 (\R_+)$.
Since the function $\psi_{slow}$ satisfies $Airy[\psi_{slow}] = i\eps \frac{U_s''}{U_s} \varphi_{slow,Ray}$ we have from $\varphi_{slow,Ray} = \frac{c_E}{\alpha}  U_s e^{-\alpha Y} + \varphi_{sRay,1} + \varphi_{sRay,2}$, 
\begin{align*}
\psi_{slow} = \psi_{slow,0} + \psi_{slow,1}\,,
\end{align*}
where $\psi_{slow,0}{\in H^2(\R_+) \cap H^1_0 (\R_+)}$ is the solution to 
\begin{align*}
Airy[\psi_{slow,0}] = i \eps \frac{c_E}{\alpha} U_s'' e^{-\alpha Y}\,, \qquad \psi_{slow,0} |_{Y=0} =0\,,
\end{align*} 
while $\psi_{slow,1}{\in H^2 (\R_+) \cap H^1_0 (\R_+)}$ is the Airy solution with the source $i\eps \frac{U_s''}{U_s} (\varphi_{sRay,1} + \varphi_{sRay,2})$.
As for the estimates of $\psi_{slow,0}$ we have from Proposition \ref{prop.Airy},
\begin{align*}
\| \pa_Y \psi_{slow,0} \|  & \leq C \eps^\frac13 \| \frac{c_E}{\alpha} U_s'' e^{-\alpha Y}\| \leq \frac{C \eps^\frac13}{\alpha}\,, 
\end{align*}
and 
\begin{align*}
\alpha \| \psi_{slow,0} \| \leq C \alpha \eps^\frac23 \| \frac{c_E}{\alpha} U_s'' e^{-\alpha Y}\| \leq C  \eps^\frac23\,.
\end{align*}
As for the estimates of $\psi_{slow,1}$, in virtue of \eqref{est.prop.Airy.4} of Proposition \ref{prop.Airy}, we see
\begin{align*}
\| \pa_Y \psi_{slow,1} \| & \leq C  \| \frac{Y U_s''}{U_s} (\varphi_{sRay,1} + \varphi_{sRay,2})\| \leq C\,.
\end{align*}
and 
\begin{align*}
\alpha \|  \psi_{slow,1} \| & \leq C \alpha \eps^\frac13 \| \frac{Y U_s''}{U_s} (\varphi_{sRay,1} + \varphi_{sRay,2}) \| \leq C\alpha \eps^\frac13\,,
\end{align*}
In order to estimate $\Phi_{slow}$, the important point is that $\int_0^\infty \pa_Y (U_s' \psi_{slow}) \, d Y =0$ in virtue of $\psi_{slow}\in H^1_0(\R_+)$, and thus, Corollary \ref{cor.prop.mOS.2} yields the estimates of $\Phi_{slow}$. Let us decompose $\Phi_{slow}$ as $\Phi_{slow}=\Phi_{slow,0} + \Phi_{slow,1}$ according to the decomposition of the source term $2\pa_Y (U_s' \psi_{slow,0} + U_s' \psi_{slow,1})$. 
Then we have from Corollary \ref{cor.prop.mOS.2},
\begin{align}\label{proof.slow.4}
\| \pa_Y\Phi_{slow,0} \| + \alpha \|  \Phi_{slow,0} \| \leq C \| (1+Y)^2 \pa_Y (U_s'  \psi_{slow,0} ) \|  & \leq C\| \pa_Y \psi_{slow,0}  \| \leq \frac{C\eps^\frac13}{\alpha} \,.
\end{align}
Similarly, we have for $\psi_{slow,1}$,
\begin{align}\label{proof.slow.5}
\begin{split}
\| \pa_Y \Phi_{slow,1} \| +\alpha \| \Phi_{slow,1} \| & \leq C  \| (1+Y)^2 \pa_Y (U_s' \psi_{slow,1} )  \|  \\
& \leq C \| \pa_Y  \psi_{slow,1} \|  \leq C\,.
\end{split}
\end{align}
Collecting these above, we have for $\tilde \phi_{slow} = \psi_{slow} + \Phi_{slow}=\psi_{slow,0} +  \psi_{slow,1} + \Phi_{slow,0} +  \Psi_{slow,1}$,
\begin{align}\label{proof.slow.6}
\| \pa_Y \tilde \phi_{slow} \| + \alpha \| \tilde \phi_{slow} \| \leq C (\frac{\eps^\frac13}{\alpha} + 1 ) \,,
\end{align}
{The next step is to estimate $\| (\pa_Y^2-\alpha^2) \phi_{slow,re} \|$, where 
$$\phi_{slow,re}=\phi_{slow}- \frac{c_E}{\alpha} U_s e^{-\alpha Y}\,.$$
Combined with the estimate on $\| \pa_Y \phi_{slow,re}\|$, that can be deduced from the previous bounds, we will obtain by interpolation an $L^\infty$ bound on $\pa_Y \phi_{slow,re}$. The point is to show that  $\pa_Y \phi_{slow,re}(0) \ll \pa_Y \bigl(\frac{c_E}{\alpha} U_s e^{-\alpha Y}\bigr)(0)$ when $\alpha \ll 1$. Unfortunately,  \eqref{proof.slow.3} is not accurate enough for this purpose, and we shall rather make use of \eqref{proof.slow.2}.}

Set $\Psi_{slow,re} = (\pa_Y^2-\alpha^2) \phi_{slow,re}$.
We observe that 
\begin{align*}
\Psi_{slow} & = (\pa_Y^2-\alpha^2) \phi_{slow} = \frac{c_E}{\alpha} (U_s'' e^{-\alpha Y} -2 \alpha U_s' e^{-\alpha Y} ) +  \Psi_{slow,re}\,,\\
U_s \Psi_{slow} - U_s'' \phi_{slow} & = -2c_E U_s U_s' e^{-\alpha Y} + U_s \Psi_{slow,re} - U_s'' \phi_{slow,re}\,.
\end{align*}
Thus \eqref{proof.slow.2} gives 
\begin{align}\label{proof.slow.7}
\begin{split}
& \langle  -2c_E U_s U_s' e^{-\alpha Y} + U_s \Psi_{slow,re}-U_s'' \phi_{slow,re},   \frac{c_E}{\alpha} (U_s'' e^{-\alpha Y} -2 \alpha U_s' e^{-\alpha Y} ) +  \Psi_{slow,re} \rangle \\
& \quad - i \eps \Big ( \| \pa_Y (\frac{c_E}{\alpha} U_s e^{-\alpha Y}) \|^2 + \alpha^2 \| \frac{c_E}{\alpha} U_s e^{-\alpha Y} \|^2 + \| \pa_Y \Psi_{slow,re} \|^2 + \alpha^2 \| \Psi_{slow,re} \|^2 \\
& \qquad + 2 \Re \langle \pa_Y (\frac{c_E}{\alpha} {U_s}e^{-\alpha Y}) , \pa_Y {\Psi_{slow,re}} \rangle + 2 \alpha^2 \Re \langle \frac{c_E}{\alpha} e^{-\alpha Y}, {\Psi_{slow,re}} \rangle \Big ) \\
& =0\,.
\end{split}
\end{align}
The real part of this identity gives 
\begin{align*}
\| \sqrt{U_s} \Psi_{slow,re} \|^2 & = \Re \langle 2c_E U_s U_s' e^{-\alpha Y} + U_s'' \phi_{slow,re}, \Psi_{slow,re}\rangle  \\
& \quad - \Re \langle U_s \Psi_{slow,re}, \frac{c_E}{\alpha} (U_s'' e^{-\alpha Y} -2 \alpha U_s' e^{-\alpha Y} ) \rangle \\
& \qquad + \Re \langle 2c_E U_s U_s' e^{-\alpha Y} + U_s'' \phi_{slow,re},\frac{c_E}{\alpha} (U_s'' e^{-\alpha Y} -2 \alpha U_s' e^{-\alpha Y} ) \rangle \,,
\end{align*}
which implies 
\begin{align*}
\| \sqrt{U_s} {\Psi_{slow,re}} \|^2 & \leq C \| \sqrt{U_s} U_s' e^{-\alpha Y} \|^2 + \| U_s'' \phi_{slow,re} \|\, \| \Psi_{slow,re} \| \\
& \quad + \frac{C}{\alpha^2} \| \sqrt{U_s} \big ( U_s'' e^{-\alpha Y} -2 \alpha U_s' e^{-\alpha Y} \big )\|^2  \\
& \qquad + \frac{C}{\alpha} ( 1+ \| U_s'' \phi_{slow,re} \| ) \\
& \leq \frac{C}{\alpha^2}  +  \| U_s'' \phi_{slow,re} \|  ( \| \Psi_{slow,re} \|  + \| U_s'' \phi_{slow,re} \| ) \,.
\end{align*}
Recall that $\phi_{slow,re} = \varphi_{sRay,1} + \varphi_{sRay,2} + \tilde \phi_{slow}$, and thus, by Corollary \ref{cor.prop.slow.Ray} and \eqref{proof.slow.6},
\begin{align}\label{proof.slow.8}
\| U_s'' \phi_{slow,re} \| \leq C (\frac{\eps^\frac13}{\alpha} + 1 )\,.
\end{align} 
Then we have 
\begin{align}\label{proof.slow.9}
\| \sqrt{U_s} \Psi_{slow,re} \|^2 
& \leq \frac{C}{\alpha^2}  +  C (\frac{\eps^\frac13}{\alpha} + 1 )  \| \Psi_{slow,re} \| \,.
\end{align}
On the other hand, the imaginary part of \eqref{proof.slow.7} yields, since $\langle -2c_E U_s U_s' e^{-\alpha Y}, \frac{c_E}{\alpha} (U_s'' e^{-\alpha Y} -2 \alpha U_s' e^{-\alpha Y} )\rangle \in \R$,
\begin{align}\label{proof.slow.10}
& \eps \big ( \| \pa_Y \Psi_{slow,re} \|^2 + \alpha^2 \| \Psi_{slow,re} \|^2 ) \\
 \leq & \frac{C \eps}{\alpha^2}  + \frac{C}{\alpha} \| \sqrt{U_s} \Psi_{slow,re} \|  + \frac{C}{\alpha} \| U_s'' \phi_{slow,re} \| + C \| \Psi_{slow,re} \|  + {C  \| U_s'' \phi_{slow,re}  \| \| \Psi_{slow,re} \|}  \nonumber \\
\leq & \frac{C}{\alpha^2} + C (\frac{\eps^\frac13}{\alpha} + 1) \| \Psi_{slow,re} \|\,.
\end{align}
Thus the inequality
\begin{align*}
\| \Psi_{slow,re} \|^2 \leq C \| \sqrt{U_s} \Psi_{slow,re} \|^\frac43 \| \pa_Y \Psi_{slow,re} \|^\frac23 + C \| \sqrt{U_s} \Psi_{slow,re} \|^2\,, 
\end{align*}
combined with \eqref{proof.slow.9} and \eqref{proof.slow.10}, leads to 
\begin{align*}
\| \Psi_{slow,re} \|^2 \leq \frac{C}{\eps^\frac13} \Big ( \frac{1}{\alpha^2} + (\frac{\eps^\frac13}{\alpha} + 1 ) \| \Psi_{slow,re} \| \Big ) \leq \frac{C}{\eps^\frac13 \alpha^2} + \frac{C}{\eps^\frac23}  (\frac{\eps^\frac13}{\alpha} + 1 )^2 \,.
\end{align*}
Hence we have arrived at 
\begin{align}\label{proof.slow.11}
\| (\pa_Y^2-\alpha^2) \phi_{slow,re} \| = \| \Psi_{slow,re} \| \leq \frac{C}{\eps^\frac16 \alpha} + \frac{C}{\eps^\frac13}\,.
\end{align}
The above estimates are valid without the condition $\frac{U_s''}{U_s}\in L^2 (\R_+)$.
We summarize the above results as follows.
\begin{proposition}\label{prop.slow.small} Let $0<\eps \leq \eps_1$ and ${0<\alpha \leq 1}$. Then there exists a solution $\phi_{slow}\in H^4 (\R_+)$ to $OS [\phi_{slow}]=0$ satisfying the following properties: $\phi_{slow} = \frac{c_E}{\alpha} U_s e^{-\alpha Y} + \phi_{slow,re}$, where 
\begin{align}
{\phi_{slow} (0) = 1}\,, \label{est.prop.slow.small.1}
\end{align}
and 
\begin{align}
\| \pa_Y \phi_{slow, re} \| + \alpha \| \phi_{slow,re} \| & \leq C ( \frac{\eps^\frac13}{\alpha} + 1 ) \,, \label{est.prop.slow.small.2}\\
\| \pa_Y \phi_{slow,re} \|_{L^\infty} & \leq C ( \frac{\eps^\frac{1}{12}}{\alpha} + \frac{1}{\eps^\frac14} )\,,\label{est.prop.slow.small.3}\\
\| (\pa_Y^2-\alpha^2 ) \phi_{slow,re} \| & \leq C ( \frac{1}{\eps^\frac16 \alpha} + \frac{1}{\eps^\frac13} )\,.\label{est.prop.slow.small.4}
\end{align}
In particular, we have 
\begin{align} 
\pa_Y \phi_{slow} (0) = \frac{c_E U_s'(0)}{\alpha} + O (\frac{\eps^\frac{1}{12}}{\alpha} + \frac{1}{\eps^\frac14})\,.\label{est.prop.slow.small.5}
\end{align}
\end{proposition}

\begin{proof} It suffices to recall $\phi_{slow,re} = \varphi_{sRay,1} + \varphi_{sRay,2} + \tilde \phi_{slow}$. Note that {$\varphi_{sRay,2}, \tilde \phi_{slow}$ belong to $H^1_0 (\R_+)$} and that $U_s (0)=0$,
and thus, $\phi_{slow} (0) = \phi_{slow,re} (0) = {\varphi_{sRay,1} (0)  = 1}$. This proves \eqref{est.prop.slow.small.1}. Estimate \eqref{est.prop.slow.small.2} follows from Corollary \ref{cor.prop.slow.Ray} and \eqref{proof.slow.6}, while \eqref{est.prop.slow.small.4} is proved in \eqref{proof.slow.11}. Then \eqref{est.prop.slow.small.3} follows from the interpolation.
The proof is complete.
\end{proof}

\subsubsection{Estimates in the case {$\alpha \geq 1$}}\label{subsubsec.slow.2}
Let us estimate $\tilde \phi_{slow} = \tilde \psi_{slow} + \tilde \Phi_{slow}$.
Since $\tilde \psi_{slow}$ is the Airy solution with the source $f=i\eps \frac{U_s''}{U_s} \varphi_{slow,Ray}$, we apply \eqref{est.prop.Airy.4} of Proposition \ref{prop.Airy} to obtain 
\begin{align*}
\| \pa_Y \tilde \psi_{slow} \|  + \alpha \| \tilde \psi_{slow} \|  \leq  C \| \frac{Y U_s''}{U_s} \varphi_{slow,Ray} \| \leq C\,.
\end{align*}
Next we recall that $\tilde \Phi_{slow}$ is the solution to \eqref{eq.mOS} with the source $f=2\pa_Y (U_s' \tilde \psi_{slow})$, we have from Corollary \ref{cor.prop.mOS.2},
\begin{align*}
\| \pa_Y \tilde \Phi_{slow} \| + \alpha \| \tilde \Phi_{slow} \| \leq C \| (1+Y)^2 \pa_Y (U_s' \tilde \psi_{slow} ) \| \leq C \| \pa_Y \tilde \psi_{slow} \| \leq C\,.
\end{align*}
Collecting these, we obtain 
\begin{align}\label{proof.slow.12}
\| \pa_Y \tilde \phi_{slow} \|  + \alpha \| \tilde \phi_{slow} \| \leq C\,.
\end{align}
{ In the case $\alpha \geq 1$ we set $\phi_{slow,re}
$ as 
$$\phi_{slow,re} = \phi_{slow} - e^{-\alpha Y}\,.$$}
Combining \eqref{proof.slow.3} and \eqref{proof.slow.12} with Proposition \ref{prop.slow.Ray'}, we have 
\begin{align}\label{proof.slow.13}
{\| (\pa_Y^2-\alpha^2) \phi_{slow,re}  \| = }\| (\pa_Y^2-\alpha^2) \phi_{slow} \|  \leq \frac{C}{\eps^\frac13} (\| U_s'' \varphi_{slow,Ray}\| + \| U_s'' \tilde \phi_{slow} \| ) \leq \frac{C}{\eps^\frac13}\,.
\end{align}
Here $C$ is independent of the condition $\frac{U_s''}{U_s}\in L^2 (\R_+)$.
We summarize the above results as follows.
\begin{proposition}\label{prop.slow.large} Let $0<\eps \leq \eps_1$, ${\alpha \geq  1}$, and $\alpha \eps^\frac13 \leq 1$. Then there exists a solution $\phi_{slow}\in H^4 (\R_+)$ to $OS [\phi_{slow}]=0$ satisfying the following properties: $\phi_{slow} = e^{-\alpha Y} + \phi_{slow,re}$, where $\phi_{slow,re}\in H^4(\R_+)\cap H^1_0 (\R_+)$ and 
\begin{align}
\| \pa_Y \phi_{slow, re} \| + \alpha \| \phi_{slow,re} \| & \leq C \,, \label{est.prop.slow.large.1}\\
\| \pa_Y \phi_{slow,re} \|_{L^\infty} & \leq \frac{C}{\eps^\frac16}\,,\label{est.prop.slow.large.2}\\
\| (\pa_Y^2-\alpha^2 ) \phi_{slow,re} \| & \leq \frac{C}{\eps^\frac13}\,.\label{est.prop.slow.large.3}
\end{align}
In particular, $\phi_{slow}(0) =1$ and $|\pa_Y \phi_{slow} (0) | \leq \alpha + C \eps^{-\frac16}$.
\end{proposition}

\subsection{Fast mode}\label{subsec.fast}

In this subsection we construct a solution $\phi_{fast}$ to 
\begin{align}\label{eq.OS_fast}
\begin{cases}
& OS [\phi] =   0 \,, \qquad Y>0\,,\\
& { \phi|_{Y=0} = O(1)} \,,
\end{cases}
\end{align}
possessing the boundary layer structure, called the fast mode.
To this end we first aim to construct the approximate solution to the problem 
\begin{align}\label{eq.Airy_b}
\begin{cases}
& Airy[\psi] =  0 \,, \qquad Y>0\,,\\
& { \psi|_{Y=0} = O(1)} \,,
\end{cases}
\end{align}
possessing the boundary layer structure.
To this end we set 
\begin{align*}
\lambda = \frac{i U_s'(0)}{\eps}\,.
\end{align*}
Let $\Gamma(z)$ be the Gamma function and let ${\rm Ai} (z)$ be the Airy function which solves the equation $\frac{d^2 {\rm Ai}}{d z^2}  - z {\rm Ai}  =0$ for $z\in \C$ and satisfies ${\rm Ai} (0) = \frac{1}{3^\frac23 \Gamma (\frac23)}$, ${\rm Ai}'(0)=-\frac{1}{3^\frac13\Gamma (\frac13)}$, and 
\begin{align}\label{eq.airy.est}
{\rm Ai} (z) \sim z^{-\frac14} e^{-\frac23 z^\frac32}\,, \qquad |z|\gg 1\,, \quad |\arg z| \leq \pi-\theta\,, ~ \theta\in (0,\pi)\,.
\end{align}
See \cite[Chapter 10]{AbSt1964} for details about the Airy function.
Then we set 
\begin{align}\label{def.psi_0.f}
\psi_{fast,0} (Y ) =   {\rm Ai} (\lambda^\frac13 Y)\,, \qquad Y>0\,.
\end{align}
Here $\lambda^\frac13 = (\frac{U_s'(0)}{\eps})^\frac13 e^{\frac{\pi}{6} i}$. 
Then $\psi_{fast,0}$ solves the equation 
\begin{align}\label{eq.psi_0.f}
\begin{cases}
& i \eps \pa_Y^2 \psi_{fast,0} + Y U_s'(0) \psi_{fast,0} =0\,, \qquad Y>0\,,\\
& \psi_{fast,0}|_{Y=0} = \frac{1}{3^\frac23 \Gamma (\frac23)}\,.
\end{cases}
\end{align}

\

Let $0<\eps^\frac13 \alpha \leq 1$. Now we set 
\begin{align}\label{proof.prop.fast.1} 
\begin{split}
\phi_{app,fast} (Y) & =C_{\eps,\alpha} \eps^{-\frac23} \int_Y^\infty e^{\alpha (Y-Y')} \int_{Y'}^\infty e^{\alpha (Y''-Y')} \psi_{0,fast} (Y'') \, d Y''\, d Y' \,.
\end{split}
\end{align}
where the constant $C_{\eps,\alpha}$ is chosen so that:
\begin{align*} 
& \phi_{app,fast} (0)=1\quad  {\rm if}~~ \eps^{-\frac23} |\int_0^\infty e^{-\alpha Y'} \int_{Y'}^\infty e^{\alpha (Y''-Y')} \psi_{fast,0} (Y'') \, d Y''\, d Y'|\geq \frac{1}{10 \,  U_s'(0)^\frac23 3^\frac13 \Gamma (\frac13)}\\
& C_{\eps,\alpha} =1 \quad  {\rm if} ~~ \eps^{-\frac23} |\int_0^\infty e^{-\alpha Y'} \int_{Y'}^\infty e^{\alpha (Y''-Y')} \psi_{fast,0} (Y'') \, d Y''\, d Y' |<\frac{1}{10\, U_s'(0)^\frac23 3^\frac13 \Gamma (\frac13)}\,.
\end{align*}
This choice of $C_{\eps,\alpha}$ ensures the condition $|\phi_{app,fast} (0)|\leq 1+\frac{1}{10\, U_s'(0)^\frac23 3^\frac13 \Gamma (\frac13)}$. 

By \eqref{eq.airy.est}  and $0<\eps^\frac13 \alpha \leq 1$ the integral defining $\phi_{app,fast}$ converges absolutely when $0<\eps^\frac13 \alpha\leq 1$ and $\eps>0$ is small enough. Note that we may assume the smallness of $\eps$ by the condition $0<\eps \leq \eps_1$, where $\eps_1>0$ is the number in Proposition \ref{prop.mOS} (by taking $\eps_1$ even smaller if necessary).  
Note that {$0<|C_{\eps,\alpha}|\leq 1 + 10 \,  U_s'(0)^\frac23 3^\frac13 \Gamma (\frac13)$} by definition.
Moreover, from the boundary layer structure of $\psi_{fast,0}$ we can also show that 
\begin{align}\label{proof.prop.fast.2} 
{1 + 10 \,  U_s'(0)^\frac23 3^\frac13 \Gamma (\frac13)\geq |C_{\eps,\alpha}| \geq c>0\,,}
\end{align}
where $c$ is independent of $\eps$ and $\alpha$ when $\eps>0$ is small enough.
We see that $(\pa_Y^2-\alpha^2) \phi_{app,fast} = C_{\eps,\alpha} \eps^{-\frac23} \psi_{fast,0}$, and thus,
\begin{align*}
OS[\phi_{app,fast}] = C_{\eps,\alpha} \eps^{-\frac23} (U_s-U_s'(0) Y) \psi_{fast,0}  - i\eps \alpha^2 C_{\eps,\alpha} \eps^{-\frac23}  \psi_{fast,0}  -U_s''  \phi_{app,fast}   \,.
\end{align*}
The fast mode $\phi_{fast}$ is then constructed in the form $\phi_{fast} = \phi_{app,fast}+ \tilde \phi_{fast}$. To construct $\tilde \phi_{fast}$ we first solve 
\begin{align*}
OS_0 [\tilde \phi_{fast,1}] & = - C_{\eps,\alpha} \eps^{-\frac23} (U_s-U_s'(0) Y) \psi_{fast,0}  +i  C_{\eps,\alpha} \eps^{\frac13}  \alpha^2 \psi_{fast,0}   + U_s''  \phi_{app,fast}  \,, \\
\tilde \phi_{fast,1}|_{Y=0} & = \pa_Y \tilde \phi_{fast,1} |_{Y=0} =0 \,,
\end{align*}
where $OS_0 [\phi] = \pa_Y (U_s \pa_Y \phi ) - \alpha^2 {U_s} \phi + i \eps (\pa_Y^2-\alpha^2)^2 \phi$.
Note that we have the relation $OS=OS_0  - \pa_Y (U_s' \cdot)$.
The advantage of introducing $OS_0$ is as follows:

(i) $OS_0$ has a good symmetry so that it is easy to solve by a simple energy method

(ii) the new error term $-\pa_Y (U_s' \tilde \phi_{fast,1})$ has zero average and also vanishes on the boundary in virtue of the boundary condition imposed on $\tilde \phi_{fast,1}$.

With this in mind let us set $\tilde \phi_{fast,2}$ as the solution to \eqref{eq.mOS} with the source term $f= \pa_Y (U_s' \tilde \phi_{fast,1})$, for which  Corollary \ref{cor.prop.mOS.1} can be applied.
Then $\phi_{fast} = \phi_{app,fast} + \tilde \phi_{fast,1}+ \tilde \phi_{fast,2}$ is our fast mode.
The main result of this section is stated as follows.
\begin{proposition}\label{prop.fast} Let $\eps_1>0$ be the number in Proposition \ref{prop.mOS}. There exists a positive number  $\delta_1$ such that if $0<\eps\leq \eps_1$ and $\eps^\frac13 \alpha\leq \delta_1$ then there exists a function $\phi_{fast}\in H^4 (\R_+)$ satisfying $OS[\phi_{fast}] =0$ and 
\begin{align}
\| \pa_Y \phi_{fast} \| + \alpha \| \phi_{fast} \| & \leq \frac{C}{\eps^\frac16} \,,\label{est.prop.fast.1}\\
\| (\pa_Y^2 -\alpha^2)\phi_{fast} \| & \leq \frac{C}{\eps^\frac12}\,,\label{est.prop.fast.2}
\end{align}
and also {
\begin{align}
\phi_{fast} (0) & = 1 \,,\label{est.prop.fast.4}\\
\pa_Y \phi_{fast} (0) & = \Big ( e^{\frac{\pi}{6}i} U_s'(0)^\frac13 3^{-\frac23} \Gamma (\frac13)  + O (\eps^\frac13\alpha) + O (\eps^\frac13) \Big ) \eps^{-\frac13}\,.\label{est.prop.fast.5}
\end{align}}
\end{proposition}
 
\begin{proof}
As stated above, we construct $\phi_{fast}$ of the form $\phi_{fast} = \phi_{app,fast} + \tilde \phi_{fast,1} + \tilde \phi_{fast,2}$, where $\phi_{app,fast}$ is given by 
\begin{align}\label{proof.prop.fast.5} 
\begin{split}
\phi_{app,fast} & = C_{\eps,\alpha} \eps^{-\frac23}  \int_Y^\infty e^{\alpha (Y-Y')} \int_{Y'}^\infty e^{\alpha (Y''-Y')} \psi_{fast,0} (Y'') \, d Y''\, d Y' \,.
\end{split}
\end{align}
From the definition of $\psi_{fast,0}$ in \eqref{def.psi_0.f} and the choice of $C_{\eps,\alpha}$, it is straightforward to see  
\begin{align}\label{proof.prop.fast.5'} 
\begin{split}
\| \pa_Y \phi_{app,fast} \| & \leq C \eps^{-\frac23} ( \alpha \eps^{\frac23+\frac16} + \eps^{\frac13+\frac16} )  = C ( \eps^\frac16 \alpha + \eps^{-\frac16} )\,,\\
\alpha \| \phi_{app,fast} \| & \leq C \eps^{-\frac23} \alpha \cdot \eps^{\frac23+\frac16} = C \alpha \eps^\frac16\,,\\
{ \| Y \phi_{app,fast} \|} & { \leq C \eps^{\frac13+\frac16} = C \eps^\frac12\,,}\\
\| (\pa_Y^2-\alpha^2 ) \phi_{app.fast} \| & \leq C \eps^{-\frac23 + \frac16} = C \eps^{-\frac12}\,,
\end{split}
\end{align}
{and} 
\begin{align*}
\pa_Y \phi_{app,fast} (0)  & = C_{\eps,\alpha} \eps^{-\frac23} \int_0^\infty e^{\alpha Y''} \psi_{0,fast} (Y'') \, d Y'' {\: + \:  \alpha \,  \phi_{app,fast} (0)}\\
& = C_{\eps,\alpha} \eps^{-\frac23}  \int_0^\infty e^{\alpha Y''}  {\rm Ai} (\lambda^\frac13 Y'') \, d Y'' {\: + \: O(\alpha)} \\
& = C_{\eps,\alpha}  \eps^{-\frac23} \lambda^{-\frac13} \int_0^\infty e^{\alpha \lambda^{-\frac13} Z} {\rm Ai} (Z) \, d Z\, {\: + \: O(\alpha)} .
\end{align*} 
Then let us recall that $\lambda^\frac13 = (\frac{U_s'(0)}{\eps})^\frac13 e^{\frac{\pi}{6}i}$ and $\int_0^\infty {\rm Ai} (z) d z = \frac13$.
Hence, if $\eps^\frac13 \alpha$ is small enough then the integral $\int_0^\infty e^{\alpha \lambda^{-\frac13} Z} {\rm Ai} (Z) \, d Z$ is away from zero uniformly and we have a lower bound of the form 
\begin{align}\label{proof.prop.fast.6} 
| \pa_Y \phi_{app,fast} (0) | \geq \frac{1}{C \eps^\frac13}\,,
\end{align}
where $C>0$ is uniform in $\eps$ and $\alpha$ as long as $\eps^\frac13 \alpha$ is small enough.
{In fact, we can show the expansions as in \eqref{est.prop.fast.4} and \eqref{est.prop.fast.5} as follows. By using $e^x-1\leq x e^x$ for $x\geq 0$ and the boundary layer structure of $\psi_{fast,0}$, we have 
\begin{align*}
\int_0^\infty { e^{-\alpha Y'}}\int_{Y'}^\infty e^{\alpha (Y''-Y')} \psi_{fast,0} (Y'') \, d Y''\, d Y' & = \int_0^\infty \int_{Y'}^\infty \psi_{fast,0} (Y'') \, d Y''\, d Y'  + O ({\eps} \alpha) \\
&  = \int_0^\infty Y \psi_{fast,0} (Y) \, d Y + O ({\eps} \alpha )\\
& = \int_0^\infty Y Ai (\lambda^\frac13 Y)\, d Y + O ({\eps} \alpha ) \\
& =\lambda^{-\frac23} \int_0^\infty Z Ai (Z) \, d Z + O ({\eps} \alpha ) \\
& = \lambda^{-\frac23} \int_0^\infty \pa_Z^2 Ai (Z) \, d Z + O ({\eps} \alpha ) \\
& = -\lambda^{-\frac23}  \pa_Z Ai (0) +  O ({\eps} \alpha ) \\
& = \eps^\frac23 e^{-\frac{\pi}{3}i}  \frac{1}{U_s'(0)^\frac23} \frac{1}{3^\frac13 \Gamma (\frac13)} + O ({\eps} \alpha ) \,.
\end{align*} 
Hence, {if $0 < \eps^\frac13 \alpha \le \delta_1$ is small enough}, then the constant $C_{\eps,\alpha}$ satisfies 
\begin{align*}
C_{\eps,\alpha} & = \Big ( \eps^{-\frac23} \int_0^\infty e^{-\alpha Y')}\int_{Y'}^\infty e^{\alpha (Y''-Y')} \psi_{fast,0} (Y'') \, d Y''\, d Y' \Big )^{-1} \\
& = e^{\frac{\pi}{3}i}  U_s'(0)^\frac23 3^\frac13 \Gamma (\frac13) + {O (\alpha \eps^\frac13)}\,,
\end{align*}
and thus we also have 
\begin{align}\label{proof.prop.fast.6.1} 
\phi_{app,fast} (0) = 1\,, 
\end{align}
and also 
\begin{align}\label{proof.prop.fast.6.2} 
\pa_Y \phi_{app,fast} (0)   & = C_{\eps,\alpha} \eps^{-\frac23} \lambda^{-\frac13} \Big ( \int_0^\infty Ai (Z) \, d Z + O (\alpha \lambda^{-\frac13} ) \Big ) + O (\alpha ) \nonumber \\
& = \Big ( e^{\frac{\pi}{3}i}  U_s'(0)^\frac23 3^\frac13 \Gamma (\frac13) + O (\alpha {\eps^\frac13}) \Big ) \eps^{-\frac13} U_s'(0)^{-\frac13} e^{-\frac{\pi}{6}i} \Big ( \frac13 + O (\alpha \lambda^{-\frac13} ) \Big )  + O(\alpha)  \nonumber \\
& = \Big ( e^{\frac{\pi}{6}i} U_s'(0)^\frac13 3^{-\frac23} \Gamma (\frac13)  + O (\alpha {\eps^\frac13} ) \Big ) \eps^{-\frac13}\,.
\end{align}
In particular, we have $\phi_{fast} (0)=1$, for $\phi_{fast} = \phi_{app,fast} + \tilde \phi_{fast}$ and $\tilde \phi_{fast} (0)=0$ by its construction.
Next let us turn to the estimate of $\tilde \phi_{fast,1}$.
The idea is to apply Proposition \ref{prop.OS_0} below, with $f_2=0$ and 
\begin{align*}
f_1 = -\frac{\alpha}{i} \sigma \Big [ - C_{\eps,\alpha} \eps^{-\frac23} (U_s-U_s'(0) Y) \psi_{fast,0}  +i  C_{\eps,\alpha} \eps^{\frac13}  \alpha^2 \psi_{fast,0}   + U_s''  \phi_{app,fast} \Big ]
\end{align*}
so that  
\begin{align*}
-\frac{i}{\alpha} \pa_Y f_1= - C_{\eps,\alpha} \eps^{-\frac23} (U_s-U_s'(0) Y) \psi_{fast,0}  +i  C_{\eps,\alpha} \eps^{\frac13}  \alpha^2 \psi_{fast,0}   + U_s''  \phi_{app,fast} \,.
\end{align*}
Thus we have from Proposition \ref{prop.OS_0} and also from Proposition \ref{prop.pre.op} (1),
\begin{align*}
& \| \pa_Y \tilde \phi_{fast,1} \| + \alpha \| \tilde \phi_{fast,1} \| \\
& \leq \frac{C}{\epsilon^\frac13\alpha } \| \frac{\alpha}{i} \sigma \Big [ - C_{\eps,\alpha} \eps^{-\frac23} (U_s-U_s'(0) Y) \psi_{fast,0}  +i  C_{\eps,\alpha} \eps^{\frac13}  \alpha^2 \psi_{fast,0}   + U_s''  \phi_{app,fast} \Big ] \| \\
& \leq \frac{C}{\epsilon^\frac13} \|  Y\Big ( - C_{\eps,\alpha} \eps^{-\frac23} (U_s-U_s'(0) Y) \psi_{fast,0}  +i  C_{\eps,\alpha} \eps^{\frac13}  \alpha^2 \psi_{fast,0}   + U_s''  \phi_{app,fast} \Big ) \| \\
& \leq \frac{C}{\eps^\frac13} \big ( \eps^{-\frac23} \| Y^3 \psi_{fast,0} \| + \eps^\frac13 \alpha^2 \| Y \psi_{fast,0} \| + \| Y \phi_{app,fast} \| \big )\,.
\end{align*}  
Hence,} the boundary layer structure of $\psi_{fast,0}$ and $\phi_{app,fast}$ {(e.g., see \eqref{proof.prop.fast.5'})} implies 
\begin{align}\label{proof.prop.fast.7} 
\| \pa_Y \tilde \phi_{fast,1} \| + \alpha \| \tilde \phi_{fast,1} \| \leq C \eps^\frac16 + C \eps^\frac12 \alpha^2\,.
\end{align}
Similarly, we have the $H^2$ bound such as 
\begin{align}\label{proof.prop.fast.8} 
\| (\pa_Y^2-\alpha^2) \tilde \phi_{fast,1} \| \leq \frac{C}{\eps^\frac13} (\eps^\frac16+\eps^\frac12 \alpha^2) \,.
\end{align} 
Finally let us estimate $\tilde \phi_{fast,2}$, which is the solution to \eqref{eq.mOS} with the source term $f= \pa_Y (U_s' \tilde \phi_{fast,1})$.
Corollary \ref{cor.prop.mOS.1} implies that 
\begin{align}\label{proof.prop.fast.3} 
& \| \pa_Y \tilde \phi_{fast,2} \| + \alpha \| \tilde \phi_{fast,2} \| & \leq C \| (1+Y)^2 \pa_Y (U_s' \tilde \phi_{fast,1})\| \leq C \| \pa_Y \tilde \phi_{fast,1} \| \leq C \eps^\frac16 + C \eps^\frac12 \alpha^2\,,
\end{align}
and 
\begin{align}\label{proof.prop.fast.3'} 
\| (\pa_Y^2-\alpha^2) \tilde \phi_{fast,2} \| & \leq C \| (1+Y)^2 \pa_Y (U_s' \tilde \phi_{fast,1})\|  + C \| \frac{\pa_Y (U_s' \tilde \phi_{fast,1})}{U_s}\| \nonumber \\
& \leq C \| \pa_Y \tilde \phi_{fast,1} \| + C \| \pa_Y^2\tilde \phi_{fast,1} \| \nonumber \\
& \leq \frac{C}{\eps^\frac13} (\eps^\frac16+\eps^\frac12 \alpha^2) \,.
\end{align}
Collecting \eqref{proof.prop.fast.7}, \eqref{proof.prop.fast.8}, and \eqref{proof.prop.fast.3'}, we have for $\tilde \phi_{fast} = \tilde \phi_{fast,1} + \tilde \phi_{fast,2}$,
\begin{align}\label{proof.prop.fast.4} 
\begin{split}
\| \pa_Y \tilde \phi_{fast} \| +\alpha \| \tilde \phi_{fast} \| & \leq C \eps^\frac16 + C \eps^\frac12 \alpha^2\,,\\
\| (\pa_Y^2-\alpha^2) \tilde \phi_{fast} \| & \leq \frac{C}{\eps^\frac13} (\eps^\frac16 + \eps^\frac12 \alpha^2)\,, \\
\| \pa_Y \tilde \phi_{fast} \|_{L^\infty} & \leq  \frac{C}{\eps^\frac16} (\eps^\frac16+\eps^\frac12 \alpha^2) \leq C ( 1+ \eps^\frac13 \alpha^2 )\,.
\end{split}
\end{align}
Then $\phi_{fast} =\phi_{app,fast}+ \tilde \phi_{fast}$ satisfies \eqref{est.prop.fast.1} and \eqref{est.prop.fast.2} when $0<\eps^\frac13 \alpha \leq \delta_1$, {and the expansions \eqref{est.prop.fast.4} and \eqref{est.prop.fast.5}  follow from \eqref{proof.prop.fast.6.1}, \eqref{proof.prop.fast.6.2}, and \eqref{proof.prop.fast.4}.} The proof is complete.
\end{proof}

\subsection{Proof of Theorem \ref{thm.OS} for $\eps^\frac13 \alpha \ll 1$}\label{subsec.proof.thm.OS}

In this subsection we prove Theorem \ref{thm.OS} when $\eps^\frac13 \alpha\ll 1$ and $0<\eps\leq \eps_1\ll 1$. As explained in the beginning of Subsection \ref{subsec.corrector}, we construct the solution $\phi$ to \eqref{eq.OS} of the form 
\begin{align*}
\phi = \Phi_{slip} [f] + a \phi_{slow} + b \phi_{fast}\,,
\end{align*}
where $\Phi_{slip}[f]$ is the solution to the modified Orr-Sommerfeld equation \eqref{eq.mOS} obtained in Proposition \ref{prop.mOS}, $\phi_{slow}$ and $\phi_{fast}$ are respectively the slow mode and the fast mode constructed in Propositions \ref{prop.slow.small}, \ref{prop.slow.large}, and \ref{prop.fast}. The coefficients $a$ and $b$ have to be chosen so that 
$$\tilde \phi = a \phi_{slow} + b \phi_{fast}$$
satisfies the boundary condition $\tilde \phi|_{Y=0}=0$, $\pa_Y \tilde \phi|_{Y=0} = - \pa_Y \Phi_{slip}[f]|_{Y=0}$, which is equivalent with the linear relation 
\begin{align*}
M 
\begin{pmatrix}
a\\
b
\end{pmatrix}
= 
\begin{pmatrix}
0\\
-\pa_Y \Phi_{slip}[f]|_{Y=0}
\end{pmatrix}\,,
\qquad 
M = 
\begin{pmatrix}
\phi_{slow}|_{Y=0} & \phi_{fast}|_{Y=0} \\
\pa_Y \phi_{slow}|_{Y=0} & \pa_Y \phi_{fast} |_{Y=0}
\end{pmatrix}\,.
\end{align*}
Hence we have 
\begin{align}\label{proof.thm.OS.1}
\begin{pmatrix}
a\\
b
\end{pmatrix}
& = \frac{1}{\det M} 
\begin{pmatrix}
\pa_Y \phi_{fast} |_{Y=0} & -\phi_{fast}|_{Y=0} \\
-\pa_Y \phi_{slow}|_{Y=0}  & \phi_{slow}|_{Y=0}
\end{pmatrix}
\begin{pmatrix}
0\\
-\pa_Y \Phi_{slip}[f]|_{Y=0}
\end{pmatrix} \nonumber \\
& = \frac{\pa_Y \Phi_{slip}[f]|_{Y=0}}{\det M}
\begin{pmatrix}
\phi_{fast}|_{Y=0}\\
-\phi_{slow}|_{Y=0}
\end{pmatrix}\,,
\qquad \det M = \Big (\phi_{slow} \pa_Y \phi_{fast} - \phi_{fast} \pa_Y \phi_{slow} \Big )|_{Y=0} \,.
\end{align}
From Propositions \ref{prop.slow.large} and \ref{prop.fast} we observe that 
\begin{align*}
|\det M| \geq \frac{1}{C\eps^\frac13} - C (\frac{1}{\eps^\frac16} + \alpha ) \geq \frac{1}{2C\eps^\frac13} \qquad {\rm if}~~ \eps^\frac13 \alpha \ll  1 ~~{\rm and}~~\alpha {\geq 1}\,.
\end{align*}
Hence, when $0<\eps^\frac13 \alpha\ll 1$ and $\alpha \geq 1$ it is easy to see from Proposition \ref{prop.slow.large} and Proposition \ref{prop.fast},
\begin{align} 
\| \pa_Y \tilde \phi \| + \alpha \| \tilde \phi \| & \leq C \eps^\frac13 |\pa_Y \Phi_{slip}[f]|_{Y=0} |\Big (  \|\pa_Y\phi_{slow} \| + \alpha \| \phi_{slow}  \| + \| \pa_Y \phi_{fast} \| + \alpha \| \phi_{fast} \| \Big ) \nonumber \\
& \leq C \eps^\frac13 |\pa_Y \Phi_{slip}[f]|_{Y=0} | \Big ( \alpha^\frac12 + \frac{1}{\eps^\frac16} \Big ) \nonumber \\
& \leq C \eps^\frac16 |\pa_Y \Phi_{slip}[f]|_{Y=0} | \,, \label{estim.tildephi.large1}
\end{align}
and similarly,
\begin{align}
\| (\pa_Y^2-\alpha^2) \tilde \phi \| & \leq C \eps^\frac13  |\pa_Y \Phi_{slip}[f]|_{Y=0} | \Big ( \| (\pa_Y^2-\alpha^2) \phi_{slow} \| + \| (\pa_Y^2-\alpha^2) \phi_{fast} \|\Big ) \nonumber \\
& \leq C \eps^\frac13  |\pa_Y \Phi_{slip}[f]|_{Y=0} | \Big({ \frac{1}{\eps^\frac13} + \frac{1}{\eps^\frac12}} \Big ) \nonumber \\
& \leq \frac{C}{\eps^\frac16}   |\pa_Y \Phi_{slip}[f]|_{Y=0} | \,. \label{estim.tildephi.large2}
\end{align}
{Combining the last two bounds with \eqref{bound.derivative.Phi.slip1}, we deduce the existence of a solution satisfying  \eqref{est.thm.OS.1}-\eqref{est.thm.OS.2}.}
If  $0 < \alpha \le {1}$ we use the asymptotic estimates \eqref{est.prop.slow.small.1}, \eqref{est.prop.slow.small.5}, \eqref{est.prop.fast.4}, and \eqref{est.prop.fast.5}, which yield
\begin{align*}
\det M & = \Big ( e^{\frac{\pi}{6}i} U_s'(0)^\frac13 3^{-\frac23} \Gamma (\frac13)  + {O (\eps^\frac13 \alpha ) + O(\eps^\frac13)} \Big ) \eps^{-\frac13} - \frac{c_E U_s'(0)}{\alpha}  + O (\frac{\eps^\frac{1}{12}}{\alpha} + \frac{1}{\eps^\frac14}) \\
& = \Big ( e^{\frac{\pi}{6}i} U_s'(0)^\frac13 3^{-\frac23} \Gamma (\frac13)  \frac{1}{\eps^\frac13} - \frac{c_E U_s'(0)}{\alpha} \Big ) \big ( 1 + o (1) \big ) \,,
\end{align*}
in the case  $0<\eps\leq \eps_1\ll 1$.
Hence we have the lower bound 
\begin{align}
|\det M| \geq \frac{1}{C} ( \frac{1}{\alpha} + \frac{1}{\eps^\frac13} )\,.
\end{align}
Thus, we have for $0<\eps^\frac13\alpha \ll 1$ and ${0 < \alpha \le 1}$, 
\begin{align}\label{proof.thm.OS.2}
\| \pa_Y \tilde \phi\| + \alpha \| \tilde \phi \| & \leq C \frac{\eps^\frac13\alpha}{\alpha+\eps^\frac13}  |\pa_Y \Phi_{slip}[f]|_{Y=0}| \bigg ( \| \pa_Y \phi_{slow} \| + \alpha \| \phi_{slow} \| + \| \pa_Y \phi_{fast} \| + \alpha \| \phi_{fast} \| \bigg ) \nonumber \\
& \leq C  \frac{\eps^\frac13\alpha}{\alpha+\eps^\frac13}   |\pa_Y \Phi_{slip}[f]|_{Y=0}| \bigg ( \frac{1}{\alpha} +  \frac{1}{\eps^\frac16} \bigg )  \,.
\end{align}
{Similarly, using \eqref{est.prop.slow.small.3}, \eqref{est.prop.fast.1}-\eqref{est.prop.fast.2}, we find
\begin{align}\label{proof.thm.OS.3-}
\| \pa_Y \tilde \phi\|_{L^\infty} & \leq C \frac{\eps^\frac13\alpha}{\alpha+\eps^\frac13}  |\pa_Y \Phi_{slip}[f]|_{Y=0}| \bigg ( \| \pa_Y \phi_{slow} \|_{L^\infty}  + \| \pa_Y \phi_{fast} \|_{L^\infty}  \bigg ) \nonumber \\
& \leq  C \frac{\eps^\frac13\alpha}{\alpha+\eps^\frac13}  \left( \frac{1}{\alpha} + C \left( \frac{\eps^{\frac{1}{12}}}{\alpha} + \frac{1}{\eps^{\frac14}} \right)+ \frac{1}{\eps^\frac13}\right) \nonumber  \\ 
& \leq C  |\pa_Y \Phi_{slip}[f]|_{Y=0}|\,. 
\end{align}}
As for the $H^2$ estimate, the similar argument shows 
\begin{align}\label{proof.thm.OS.3}
\| (\pa_Y^2 -\alpha^2) \tilde \phi \| & \leq  C \frac{\eps^\frac13\alpha}{\alpha+\eps^\frac13}  |\pa_Y \Phi_{slip}[f]|_{Y=0}| \bigg ( \| (\pa_Y^2 -\alpha^2) \phi_{slow} \| + \| (\pa_Y^2-\alpha^2 ) \phi_{fast} \| \bigg ) \nonumber \\
& \leq C \frac{\eps^\frac13\alpha}{\alpha+\eps^\frac13}  |\pa_Y \Phi_{slip}[f]|_{Y=0}|  \bigg ( \frac{1}{\alpha} +  \frac{1}{\eps^\frac16 \alpha} + \frac{1}{\eps^\frac13} +  \frac{1}{\eps^\frac12}  \bigg ) \nonumber \\
& \leq C \frac{\eps^\frac16\alpha}{\alpha+\eps^\frac13}  |\pa_Y \Phi_{slip}[f]|_{Y=0}|  \bigg ( \frac{1}{\alpha} + \frac{1}{\eps^\frac13}  \bigg )  \,.
\end{align}
{Combining with \eqref{bound.derivative.Phi.slip2}, we recover the existence of a solution with  the bounds  \eqref{est.thm.OS.3}-\eqref{est.thm.OS.4}-\eqref{est.thm.OS.5}.}

The uniqueness of the solution follows from the theory of ordinary differential equations.
Indeed, there exist four linearly independent solutions to the fourth order differential equation $OS[\phi]=0$, and two of which are taken as the slow mode and the fast mode constructed as above, while the other two grow as $Y\rightarrow \infty$. Hence, if $\phi$ is the solution to $OS[\phi]=0$ decaying as $Y\rightarrow \infty$ with $\phi|_{Y=0} =\pa_Y\phi|_{Y=0} =0$, then $\phi$ must be a linear combination of $\phi_{slow}$ and $\phi_{fast}$, and then, as we have seen that $\det M\ne 0$, $\phi$ must be trivial. The proof of Theorem \ref{thm.OS} is complete.

\subsection{Proof of Theorem \ref{thm.linear} for $|\tilde n| \leq \delta \nu^{-\frac34}$}\label{subsec.proof.linear}

In this section we prove Theorem \ref{thm.linear} for frequencies $\tilde n$ satisfying $0<|\tilde n| \nu^\frac34 \ll 1$. We may assume that $\tilde n>0$ and note that $0<\tilde n \nu^\frac34\ll 1$ corresponds to $0<\eps^\frac13 \alpha \ll 1$. We also assume that the source term $f_n$ belongs to $H^1(\R_+)^2$: the existence result  for the case $f_n\in L^2 (\R_+)^2$ follows from our uniform estimate and a standard density argument.
With this  in mind, we consider the Orr-Sommerfeld equations \eqref{eq.OS_start} with $f\in H^1(\R_+)^2$. These equations are expressed  in the rescaled variable $Y$, and are equivalent to the original problem \eqref{eq.linearized}.
An important point here is that the source term $-f_2-\frac{i}{\alpha} \pa_Y f_1$ is not compatible with a direct  application of Theorem \ref{thm.OS}. To fill this discrepancy we observe the identity
\begin{align}
OS[\phi]  & = \pa_Y (U_s \pa_Y \phi ) -\alpha^2 U_s \phi - U_s'\pa_Y \phi - U_s'' \phi + i \eps (\pa_Y^2-\alpha^2)^2 \phi \nonumber \\
& =   \pa_Y (U_s \pa_Y \phi ) -\alpha^2 U_s \phi + i \eps (\pa_Y^2-\alpha^2)^2 \phi \nonumber \\
& \quad - U_s'\pa_Y \phi - U_s''\phi   \nonumber \\
& = : OS_0[\phi] - U_s'\pa_Y \phi - U_s''\phi  \,,
\end{align}
and then we first consider the problem 
\begin{align}\label{eq.OS_0}
\begin{cases}
& \displaystyle OS_0 [\phi_0] = - f_2-\frac{i}{\alpha} \pa_Y f_1\,, \qquad Y>0\,,\\
& \phi_0|_{Y=0} = \pa_Y \phi_0 |_{Y=0} =0\,.
\end{cases}
\end{align}

\begin{proposition}\label{prop.OS_0} Let $\alpha>0$. Then for any $f=(f_1,f_2)\in H^1(\R_+)^2$ there exists a unique solution $\phi_0\in H^2_0 (\R_+)\cap H^4 (\R_+)$ to \eqref{eq.OS_0} such that 
\begin{align}
\| \pa_Y \phi_0 \| + \alpha \| \phi_0 \| & \leq  \frac{C}{\eps^\frac13 \alpha} \| f\| \,,\label{est.prop.OS_0.1}\\
\| (\pa_Y^2 -\alpha^2) \phi_0 \| & \leq \frac{C}{\eps^\frac23 \alpha} \| f\|\,.\label{est.prop.OS_0.2}
\end{align} 
\end{proposition}

\begin{proof} We focus on the a priori estimate.
For simplicity of notations we set 
\begin{align*}
E = \| \pa_Y\phi_0 \|^2 + \alpha^2 \| \phi_0 \|^2\,.
\end{align*}
By taking the inner product with $\phi_0$ in the equation \eqref{eq.OS_0} we obtain 
\begin{align*}
\| \sqrt{U_s} \pa_Y \phi_0 \|^2 + \alpha^2 \| \sqrt{U_s}\phi_0 \|^2 -i \eps \| (\pa_Y^2-\alpha^2) \phi_0 \| ^2 &  =  \langle f_2 + \frac{i}{\alpha} \pa_Y f_1, \phi_0 \rangle\,,
\end{align*}
and thus, the real part and the imaginary part of this identity give
\begin{align}\label{proof.prop.OS_0.1} 
\| \sqrt{U_s} \pa_Y \phi_0 \|^2 + \alpha^2 \| \sqrt{U_s}\phi_0 \|^2 = \Re  \langle f_2 + \frac{i}{\alpha} \pa_Y f_1, \phi_0 \rangle\,,
\end{align}
and 
\begin{align}\label{proof.prop.OS_0.2} 
\eps  \| (\pa_Y^2-\alpha^2) \phi_0 \| ^2 = - \Im 
\langle f_2 + \frac{i}{\alpha} \pa_Y f_1, \phi_0 \rangle\,.
\end{align}
Equality \eqref{proof.prop.OS_0.1} implies 
\begin{align}\label{proof.prop.OS_0.3} 
\| \sqrt{U_s} \pa_Y \phi_0 \|^2 + \alpha^2 \| \sqrt{U_s}\phi_0 \|^2  \leq \frac{2}{\alpha} \| f\|  E^\frac12\,.
\end{align} 
On the other hand, we have from \eqref{proof.prop.OS_0.2},
\begin{align}\label{proof.prop.OS_0.4} 
\eps \| (\pa_Y^2-\alpha^2) \phi_0 \| ^2  \leq 2 \frac{\| f\| }{\alpha}  E^\frac12 \,.
\end{align}
By using the interpolation inequality we have 
\begin{align*}
E & \leq C \| \sqrt{U_s} \pa_Y \phi_0 \|^\frac43 \| \pa_Y^2 \phi_0 \|^\frac23 + C \| \sqrt{U_s} \pa_Y \phi_0 \|^2 \\
& \quad + C \alpha^2 \| \sqrt{U_s} \phi_0\|^\frac43 \| \pa_Y \phi_0 \|^\frac23 + C \alpha^2 \| \sqrt{U_s} \phi_0 \|^2 \\
& \leq  C (\frac{\| f\| E^\frac12}{\alpha})^\frac23  \big ( \frac{1}{\eps} \frac{\| f\|}{\alpha}  E^\frac12 \big )^\frac13    + C \alpha^\frac23 (\frac{\| f\| E^\frac12}{\alpha})^\frac23 E^\frac13+ C\frac{\|f\|}{\alpha} E^\frac12\\
& \leq \frac{C}{\eps^\frac13 \alpha} \|f\| E^\frac{1}{2} + C \| f\|^\frac23  E^\frac23+  C\frac{\|f\|}{\alpha} E^\frac12  \,.
\end{align*}
This implies 
\begin{align}\label{proof.prop.OS_0.5}
E \leq C \big ( \frac{1}{\eps^\frac23 \alpha^2} + 1\big ) \| f\|^2 \leq \frac{C}{\eps^\frac23 \alpha^2}  \| f\|^2 \,.
\end{align}
Here we have used $0<\tilde n \leq \nu^{-\frac34}$. 
Estimate \eqref{est.prop.OS_0.1} is proved. Then we have from \eqref{proof.prop.OS_0.4},
\begin{align}\label{proof.prop.OS_0.6}
\| (\pa_Y^2-\alpha^2) \phi_0 \|^2  \leq \frac{C}{\eps \alpha}  \frac{1}{\eps^\frac13 \alpha}  \|f\|^2 
& \leq \frac{C}{\eps^\frac43 \alpha^2}  \| f\|^2\,,
\end{align}
which shows \eqref{est.prop.OS_0.2}.
We have proved the a priori estimates and the uniqueness.
The regularity $\phi_0\in H^4 (\R_+)$ follows from the elliptic regularity.
As for the existence, we introduce the operator $OS_{0,l}= OS_0+ i  l$ for a constant $l>0$, and then it is easy to see that $OS_{0,l}$ (under the same boundary condition as above) is invertible if $l$ is sufficiently large.
The argument above holds exactly in the same way even for $OS_{0,l}$ with $l>0$, and hence, the estimate is uniform in $l>0$. Then we can prove the existence of the solution for $l=0$ by the standard continuity method about $l$. The details are omitted here.
The proof is complete.
\end{proof}

\mspace
Let us prove Theorem \ref{thm.linear}. As mentioned in the beginning of this subsection, we first assume $f_n\in H^1 (\R_+)^2$.
The solution $\phi\in H^4 (\R_+)$ to \eqref{eq.OS_start} is constructed in the form $\phi=\phi_0 + \phi_1$, where $\phi_0$ is the solution to \eqref{eq.OS_0} obtained in Proposition \ref{prop.OS_0}, and $\phi_1$ is the solution to \eqref{eq.OS} with $f$ replaced by $U_s' \pa_Y \phi_0 + U_s''\phi_0 = \pa_Y (U_s' \phi_0)$, for which Theorem \ref{thm.OS} is applied without loss of $\alpha^{-1}$ when $\alpha$ is small, in virtue of the fact $\int_0^\infty \pa_Y (U_s' \phi_0) \, d Y =0$. 
Note that, in order to apply Theorem \ref{thm.OS}, we need the condition $0<\eps^\frac13 \alpha\leq \delta_0$, which is equivalent with $0<\tilde n\leq \delta_0^\frac32 \nu^{-\frac34}$. 
Let $\alpha\geq 1$ in addition, which is the case $\tilde n \geq  \nu^{-\frac12}$.
Theorem \ref{thm.OS} implies that 
{\begin{align}
\| \pa_Y \phi_1 \| + \alpha \| \phi_1 \|  & \leq  C \| (1+Y) \pa_Y (U_s' \phi_0) \|  \leq C \| \pa_Y \phi_0 \|  \,,\label{proof.thm.inear.1}\\
\| (\pa_Y^2-\alpha^2) \phi_1 \| & \leq \frac{C}{\eps^\frac13} \| (1+Y) \pa_Y ( U_s' \phi_0) \| \leq \frac{C}{\eps^\frac13} \| \pa_Y \phi_0 \| \,.\label{proof.thm.inear.1'}
\end{align}}
Therefore, Proposition \ref{prop.OS_0} and \eqref{proof.thm.inear.1}-\eqref{proof.thm.inear.1'} yield for $\phi=\phi_0+\phi_1$,
\begin{align}
\| \pa_Y \phi \| + \alpha \| \phi \| & \leq  \frac{C}{\eps^\frac13 \alpha} \| f\| \,,\label{proof.thm.inear.2}\\
\| (\pa_Y^2 -\alpha^2) \phi \| & \leq \frac{C}{\eps^\frac23 \alpha} \| f\|\,.\label{proof.thm.inear.2'}
\end{align} 
Rescaling back to the original variable and recalling \eqref{def.rescale}, we have from \eqref{proof.thm.inear.2},
\begin{align}\label{proof.thm.inear.3}
\| u_n \|_{L^2} \leq \frac{C}{\eps^\frac13 \tilde n} \| f_n \|_{L^2} \leq \frac{C}{\tilde n^\frac23} \| f_n \|_{L^2}\,,
\end{align}
and from \eqref{proof.thm.inear.2'},
\begin{align}\label{proof.thm.inear.4}
\| \pa_y u_n \|_{L^2} + \tilde n \| u_n \|_{L^2} \leq \frac{C}{\tilde n^\frac13 \nu^\frac12} \| f_n \|_{L^2}\,.
\end{align}
This proves \eqref{est.thm.linear.1} and \eqref{est.thm.linear.2} in the case $0<|\tilde n|\leq \delta_0^{\frac32}\nu^{-\frac34}$ and $\tilde n \geq  \nu^{-\frac12}$, where $\delta_0>0$ is the number in Theorem \ref{thm.OS}. 
Next we consider the case $0<\alpha\leq 1$, that is, $\tilde n \leq  \nu^{-\frac12}$.
In this case Theorem \ref{thm.OS} shows that 
\begin{align}
\| \pa_Y \phi_1 \| + \alpha \| \phi_1 \|  & \leq  C\frac{\alpha + \eps^\frac16}{\alpha + \eps^\frac13} \| (1+Y)^2 \pa_Y (U_s' \phi_0) \|  \leq C\frac{\alpha+\eps^\frac16}{\alpha + \eps^\frac13} \| \pa_Y \phi_0 \|  \,,\label{proof.thm.inear.5}\\
\| \pa_Y \phi_1 \|_{L^\infty} + \alpha \| \phi_1 \|_{L^\infty} & \leq  \frac{C}{\eps^\frac16} \| (1+Y)^2 \pa_Y (U_s' \phi_0) \|  \leq \frac{C}{\eps^\frac16} \| \pa_Y \phi_0 \|  \,,\label{proof.thm.inear.6}\\
\| (\pa_Y^2-\alpha^2) \phi_1 \| & \leq \frac{C}{\eps^\frac13} \| (1+Y)^2 \pa_Y ( U_s' \phi_0) \| \leq \frac{C}{\eps^\frac13} \| \pa_Y \phi_0 \| \,.\label{proof.thm.inear.7}
\end{align}
Hence, combining with Proposition \ref{prop.OS_0}, we obtain 
\begin{align}
\| \pa_Y \phi \| + \alpha \| \phi \| & \leq  C\frac{\alpha+\eps^\frac16}{(\alpha+\eps^\frac13) \eps^\frac13 \alpha} \| f\| \,,\label{proof.thm.inear.8}\\
\| \pa_Y \phi \|_{L^\infty} + \alpha \| \phi \|_{L^\infty} & \leq  \frac{C}{\eps^\frac12 \alpha} \| f\| \,,\label{proof.thm.inear.9}\\
\| (\pa_Y^2 -\alpha^2) \phi \| & \leq \frac{C}{\eps^\frac23 \alpha} \| f\|\,.\label{proof.thm.inear.10}
\end{align} 
In the original variable these estimates provide 
\begin{align}
\| u_n \|_{L^2} & \leq  C\frac{\tilde n\nu^\frac12 + \tilde n^{-\frac16}}{(\tilde n \nu^{\frac12} + \tilde n^{-\frac13})\tilde n^\frac23} \| f_n\|_{L^2} \,,\label{proof.thm.inear.8'}\\
\| u_n \|_{L^\infty} & \leq  \frac{C}{\tilde n^\frac12\nu^\frac14} \| f_n\|_{L^2} \,,\label{proof.thm.inear.9'}\\
\| \pa_y u_n\|_{L^2} + \tilde n \| u_n \|_{L^2} & \leq \frac{C}{\tilde n^\frac13 \nu^\frac12} \| f_n\|_{L^2}\,.\label{proof.thm.inear.10'}
\end{align} 
This proves \eqref{est.thm.linear.1'}, \eqref{est.thm.linear.2'}, and \eqref{est.thm.linear.3'},
by comparing the size of $\tilde n\nu^{\frac12}$ and of $\tilde n^{-\frac16}$ or $\tilde n^{-\frac13}$.
In virtue of the assumption $f_n\in  H^1(\R_+)^2$ we have the $H^3$ regularity of $u_n$. The bound of the $H^2$ norm of $u_n$ in terms of the $L^2$ norm of $f_n$ is then recovered from the elliptic regularity of the Stokes operator, for the $L^2$ norm of the term $-i\tilde n U_s^\nu u_n - u_{n,2}\pa_y U_s^\nu {\bf e}_1 + f_n$ in \eqref{eq.linearized} is bounded in terms of $\|f_n \|_{L^2}$ as already proved above. The details are omitted here. Hence, the existence and the estimates of the solution $u_n$ for the case $f_n\in L^2 (\R_+)^2$ follow from a standard density argument. As for the uniqueness, if $u_n\in H^2(\R_+)^2 \cap H^1_0 (\R_+)^2$ is the solution to \eqref{eq.linearized} with $f_n=0$, then $u_n$ is smooth and the associated streamfunction $\phi$ in the rescaled variable satisfies \eqref{eq.OS} with $f=0$. Hence Theorem \ref{thm.OS} implies $\phi=0$, and thus, $u_n=0$.
The proof of Theorem \ref{thm.linear} is complete in the case $0<|\tilde n|\leq \delta_0^{\frac32}\nu^{-\frac34}$.

\section{Analysis in high frequency $|\tilde n| \geq \delta \nu^{-\frac34}$}\label{sec.high}

In this section we consider the linearized problem in the regime $|\tilde n|\geq \delta \nu^{-\frac34}$.  If $|\tilde n|\gg \nu^{-\frac34}$ then the problem is easy since the dissipation is strong enough. Hence we must at least handle the regime $|\tilde n|\sim O(\nu^{-\frac34})$, in which the strength of the boundary layer is the same order as the dissipation, and this fact leads to an essential difficulty in constructing the boundary layer corrector, in particular the lower bound of the fast mode, which was the key in the previous section. 

\mspace
To overcome this difficulty we go back to the energy argument for the velocity in the original variables and try to gain the coercive estimate from the convection $U_s^\nu \pa_x$. This will be achieved from the  imaginary part of the energy identity, rather than the positive part of it.
To be precise let us recall the problem 
\begin{align}\label{eq.linearized.high}
\begin{cases}
& i \tilde n U_s^\nu u_n + u_{n,2} (\pa_y U_s^\nu ) {\bf e_1}  - \nu (\pa_y^2 - \tilde  n^2) u_n +
\begin{pmatrix}
i \tilde n p_n\\
\pa_y p_n
\end{pmatrix}
=  f_n  \,, \qquad  y>0\,,\\
& i\tilde n u_{n,1} + \pa_y u_{n,2} =0\,, \qquad y>0\,,\\
& u_n |_{y=0} = 0\,.
\end{cases}
\end{align}
Here $u_n= (\pa_y\phi_n, -i\tilde n \phi_n)^\top$ with the streamfunction $\phi_n=\phi_n (y)$.
The following proposition is valid for the regime $|\tilde n | \nu^\frac12 \gg 1$.

\begin{proposition}\label{prop.high} There exists a positive number $\delta_2>0$ such that if $|\tilde n | \geq \delta_2^{-1} \nu^{-\frac12}$ then there exists a unique solution $u_n \in H^2(\R_+)^2 \cap H^1_0(\R_+)^2$ to \eqref{eq.linearized.high} satisfying the following estimates:

\noindent {\rm (i)} if $|\tilde n|\geq \delta_2^{-1} \nu^{-\frac34}$ then 
\begin{align}
\| u_n \|_{L^2} \leq \frac{C}{|\tilde n|^2 \nu}  \| f_n \|_{L^2}\,, \label{est.prop.high.1} \\
\| \pa_y u_n \|_{L^2} + |\tilde n | \| u_n \|_{L^2} \leq \frac{C}{|\tilde n| \nu}  \| f_n \|_{L^2}\,. \label{est.prop.high.2}
\end{align}

\noindent {\rm (ii)} if $\delta_2^{-1} \nu^{-\frac12} \leq |\tilde n| \leq \delta_2^{-1} \nu^{-\frac34}$ then 
\begin{align}
\| u_n \|_{L^2} \leq  \frac{C}{|\tilde n|^\frac45 \nu^\frac{1}{10}}  \| f_n \|_{L^2}\,,\label{est.prop.high.3}  \\
\| \pa_y u_n \|_{L^2}+ |\tilde n | \| u_n \|_{L^2} \leq  \frac{C}{|\tilde n|^\frac35 \nu^\frac{7}{10}}  \| f_n \|_{L^2}\,. \label{est.prop.high.4} 
\end{align}
\end{proposition}

\begin{remark}{\rm When $|\tilde n| \sim \nu^{-\frac34}$ we have $|\tilde n|^{-\frac45} \nu^{-\frac{1}{10}} \sim |\tilde n|^{-\frac23}$ and $|\tilde n|^{-\frac35} \nu^{-\frac{7}{10}} \sim |\tilde n|^{-\frac13} \nu^{-\frac12}$.
}
\end{remark}

\begin{proof} We may assume that $\tilde n>0$.
We focus on the a priori estimate. By taking the inner product with $u_n$ in the first equation of \eqref{eq.linearized.high}, we have 
\begin{align}\label{proof.prop.high.1}
i \tilde n \langle U_s^\nu u_n, u_n \rangle + \langle u_{n,2} \pa_y U_s^\nu, u_{n,1}\rangle + \nu  ( \| \pa_y u_n \|^2_{L^2} + \tilde n^2 \| u_n \|^2_{L^2} ) = \langle f_n ,  u_n \rangle\,.
\end{align}
The real part of this identity gives 
\begin{align}\label{proof.prop.high.1'}
\nu  ( \| \pa_y u_n \|^2_{L^2} + \tilde n^2 \| u_n \|^2_{L^2} ) =  - \Re \langle u_{n,2} \pa_y U_s^\nu, u_{n,1}\rangle + \Re \langle f_n ,  u_n \rangle\,,
\end{align}
while the imaginary part of this identity gives 
\begin{align}\label{proof.prop.high.1''}
\tilde n \bigg ( \langle U_s^\nu u_n,u_n \rangle - \Re  \langle (\pa_y U_s^\nu) \phi_n, \pa_y \phi_n \rangle \bigg ) = \Im \langle f_n , u_n \rangle\,.
\end{align}
Below we take $\delta_2\in (0,1)$ small enough depending only on $U_s$.

\noindent (i) Case $\tilde n \geq \delta_2^{-1} \nu^{-\frac34}$: The first term in the right-hand side of \eqref{proof.prop.high.1'} is estimated as 
\begin{align*}
|- \Re \langle u_{n,2} \pa_y U_s^\nu, u_{n,1}\rangle| \leq \nu^{-\frac12} \| \pa_Y U_s \|_{L^\infty} \| u_n \|^2_{L^2} & = \nu \tilde n^2 \frac{\| \pa_Y U_s \|_{L^\infty} }{\nu^\frac32 \tilde n^2} \| u_n \|^2_{L^2}  \\
& \leq \frac{\nu\tilde n^2}{2}\| u_n \|^2_{L^2}
\end{align*}
if $\delta_2>0$ is small enough. Thus we have 
\begin{align*}
\nu ( \| \pa_y u_n \|^2_{L^2} + \tilde n^2 \| u_n \|^2_{L^2} )  \leq 2 \| f_n \|_{L^2}\, \| u_n \|_{L^2}\,,
\end{align*}
which proves \eqref{est.prop.high.1}-\eqref{est.prop.high.2}.

\noindent (ii) Case $\delta_2^{-1} \nu^{-\frac12} \leq \tilde n \leq \delta_2^{-1} \nu^{-\frac34}$:
We use the imaginary part of the energy identity \eqref{proof.prop.high.1''}, and then from the integration by parts,
\begin{align*}
\int_0^\infty U_s^\nu \big ( |\pa_y \phi_n|^2 + \tilde n^2 | \phi_n |^2 \big ) d y + \frac12 \int_0^\infty (\pa_y^2 U_s^\nu) |\phi_n|^2 d y  = \frac{1}{\tilde n} \Im \langle f_n, u_n \rangle
\end{align*}
To estimate the second term in the left-hand side of this equality, we use the condition $|\pa_y^2 U_s^\nu|\leq C \nu^{-\frac12}  \pa_y U_s^\nu$ for $0\leq y\leq 2 Y_0\nu^\frac12$, see Proposition \ref{prop.pre}. 
Let $\chi (Y)$ be a cut-off such that $\chi(Y)=1$ for $0\leq Y\leq Y_0$ and $\chi(Y)=0$ for $Y \geq 2Y_0$. Then we have
\begin{align*}
\big |  \frac12 \int_0^\infty (\pa_y^2 U_s^\nu) |\phi_n|^2 d y \big | & \leq \frac{C}{2} \nu^{-\frac12} \int_0^\infty \pa_y U_s^\nu \chi (\frac{y}{\sqrt{\nu}}) |\phi_n|^2 d y + C \nu^{-1} \int_{Y_0\nu^\frac12}^\infty |\phi_n|^2 d y \\
& = - \frac{C}{2} \nu^{-\frac12} \int_0^\infty  U_s^\nu \pa_y \Big ( \chi (\frac{y}{\sqrt{\nu}}) |\phi_n|^2 \Big ) d y + C \nu^{-1} \int_{Y_0\nu^\frac12}^\infty |\phi_n|^2 d y \\
& \leq C \nu^{-1} \int_{Y_0\nu^\frac12}^\infty U_s^\nu |\phi_n|^2 d y - C \nu^{-\frac12}  \Re  \int_0^\infty U_s^\nu \chi (\frac{y}{\sqrt{\nu}}) \pa_y \phi_n \, \overline{\phi_n} d y \\
& \leq \frac{C}{\nu\tilde n^2} \| \sqrt{U_s^\nu} \tilde n \phi_n \|^2_{L^2} + \frac{C}{\nu^\frac12 \tilde n} \| \sqrt{U_s^\nu} \pa_y \phi_n \|_{L^2} \, \| \sqrt{U_s^\nu} \tilde n \phi_n \|_{L^2}\,.
\end{align*}
Here we have used $|\frac{1}{U_s^\nu (y)}|\leq C$ for $y\geq Y_0 \nu^\frac12$. Hence if $\nu^\frac12\tilde n$ is large enough, the term $\frac12 \int_0^\infty (\pa_y^2 U_s^\nu) |\phi_n|^2 d y$ is absorbed by the first term in the left-hand side, resulting in
\begin{align}\label{proof.prop.high.2}
\| \sqrt{U_s^\nu} u_n \|^2_{L^2} \leq \frac{2}{\tilde n} \, \Im \langle f_n, u_n \rangle\,.
\end{align}
The estimate \eqref{proof.prop.high.2} is coercive but degenerate near the boundary $y=0$. 
To recover the estimate near the boundary we use the real part of the energy identity \eqref{proof.prop.high.1'}.
To this end we recall the interpolation inequality \eqref{proof.prop.Airy.5}, which is formulated  in the rescaled variable $Y$.
In the original variable, it can be written as 
\begin{align}\label{proof.prop.high.3}
\| g \|^2_{L^2} \leq C \nu^\frac13 \| \sqrt{U_s^\nu} g \|^\frac43_{L^2} \| \pa_y g \|^\frac23_{L^2} + C \| \sqrt{U_s^\nu} g \|^2_{L^2} \,, \qquad g = g(y) \in H^1_0 (\R_+)\,.
\end{align}
This yields from \eqref{proof.prop.high.1'} that
\begin{align*}
\| u_n \|^2_{L^2} & \leq C\nu^\frac13 \| \sqrt{U_s^\nu} u_n \|^\frac43_{L^2} \big ( - \nu^{-1} \Re \langle u_{n,2} \pa_y U_s^\nu, u_{n,1} \rangle + \nu^{-1} \Re \langle f_n, u_n\rangle \big ) ^\frac13 + C \| \sqrt{U_s^\nu} u_n \|^2_{L^2} \,.
\end{align*}
By using the Hardy inequality and the divergence free condition we have 
\begin{align}\label{proof.prop.high.4}
|\Re \langle u_{n,2} \pa_y U_s^\nu, u_{n,1} \rangle| & \leq \| \frac{y^\frac12 \pa_y U_s^\nu}{\sqrt{U_s^\nu}}\|_{L^\infty} \| \frac{u_{n,2}}{y^\frac12} \|_{L^2} \| \sqrt{U_s^\nu} u_{n,1} \|_{L^2} \nonumber \\
& \leq C \nu^{-\frac14} \| \frac{Y^\frac12 \pa_Y U_s}{\sqrt{U_s}} \|_{L_Y^\infty} \| u_{n,2} \|_{L^2}^\frac12 \| \pa_y u_{n,2} \|_{L^2}^\frac12 \| \sqrt{U_s^\nu} u_{n,1} \|_{L^2}\nonumber \\
& \leq \frac{C\tilde n^\frac12}{\nu^\frac14} \|u_n \|_{L^2} \| \sqrt{U_s^\nu} u_n\|_{L^2}  \,.
\end{align}
Therefore,
\begin{align*}
\| u_n \|^2_{L^2} & \leq C\nu^\frac13 \| \sqrt{U_s^\nu} u_n \|^\frac43_{L^2} \bigg (  \nu^{-\frac54} \ \tilde n^\frac12   \| \sqrt{U_s^\nu} u_{n}\|_{L^2} \| u_n \|_{L^2}  + \nu^{-1}   |\Re \langle f_n, u_n\rangle| \bigg )  ^\frac13    + C \| \sqrt{U_s^\nu} u_n \|^2_{L^2} \\
& \leq C \nu^{-\frac{1}{12}} \tilde n^\frac16 \| \sqrt{U_s^\nu} u_n \|^\frac{5}{3}_{L^2} \| u_n \|^\frac13_{L^2}
+ C \| \sqrt{U_s^\nu} u_n \|^\frac43_{L^2}  |\Re \langle f_n, u_n\rangle|^\frac13+ C \| \sqrt{U_s^\nu} u_n \|^2_{L^2}\,,
\end{align*}
which gives from \eqref{proof.prop.high.2} and $\tilde n\geq \nu^{-\frac12}$,
\begin{align}
\| u_n \|^2_{L^2} & \leq C  ( \nu^{-\frac{1}{10}} \tilde n^\frac15   + 1 )  \|\sqrt{U_s^\nu} u_n \|^2 _{L^2} + C \| \sqrt{U_s^\nu} u_n \|^\frac43_{L^2}  |\Re \langle f_n, u_n\rangle|^\frac13 \nonumber \\
& \leq C \nu^{-\frac{1}{10}} \tilde n^{\frac15-1} |\Im \langle f_n, u_n \rangle| + C \tilde n^{-\frac23} |\Im \langle f_n,u_n\rangle |^\frac23 |\Re \langle f_n, u_n\rangle|^\frac13 \nonumber \\
& \leq C ( \nu^{-\frac{1}{10}} \tilde n^{-\frac45}  + \tilde n^{-\frac23} ) \| f_n \|_{L^2} \| u_n \|_{L^2} \nonumber\,, 
\end{align}
that is,
\begin{align}
\| u_n \|_{L^2} \leq C ( \nu^{-\frac{1}{10}} \tilde n^{-\frac45}  + \tilde n^{-\frac23} ) \| f_n \|_{L^2} \leq C \nu^{-\frac{1}{10}} \tilde n^{-\frac45} \| f_n \|_{L^2} \,.
\end{align}
Here we have used the condition $\tilde n\leq \delta_2^{-1} \nu^{-\frac34}$.
Then \eqref{proof.prop.high.4} and \eqref{proof.prop.high.2} imply
\begin{align}
|\Re \langle u_{n,2} \pa_y U_s^\nu, u_{n,1} \rangle| & \leq C\nu^{-\frac14} \| f_n \|^\frac12_{L^2} \| u_n \|^\frac32_{L^2}\,.
\end{align}
Thus the estimate for the derivatives is obtained from \eqref{proof.prop.high.1'}:
\begin{align*}
\|\pa_y u_n \|^2_{L^2} + \tilde n^2 \| u_n \|^2_{L^2} & \leq C\nu^{-\frac54}  \| f_n \|^\frac12_{L^2} \| u_n \|^\frac32_{L^2} + \nu^{-1} \|f_n \|_{L^2} \, \| u_n \|_{L^2} \\
& \leq C \big (\tilde n^{-\frac{6}{5}} \nu^{-\frac{7}{5}} + \tilde n^{-\frac45} \nu^{-\frac{11}{10}} \big ) \| f_n \|^2_{L^2} \\
& \leq C\tilde n^{-\frac65} \nu^{-\frac75} \| f_n \|^2_{L^2} \,.
\end{align*}
The proof of the a priori estimates is complete, which also implies the uniqueness. As for the existence, we first  replace the operator $-\nu (\pa_y^2 - \tilde n^2)$ by $-\nu (\pa_y^2 - \tilde n^2) + l$ with $l>0$, and  if $l$ is large enough then it is easy to show the unique existence of the solution for any $f_n\in L^2 (\R_+)^2$. One can check that the argument above for the a priori estimates is valid even for $l>0$ without any changes, and all of the a priori estimates are uniform in $l>0$. Thus, the existence of the solution for the case $l=0$ holds by the continuity. The proof is complete.
\end{proof}

\noindent
Proposition \ref{prop.high} shows Theorem \ref{thm.linear} for $\tilde n \geq \delta_0^\frac32 \nu^{-\frac34}$ as long as $\delta_2^{-1}\nu^{-\frac12} \leq \delta_0^\frac32 \nu^{-\frac34}$, which is valid for sufficiently small $\nu>0$. The proof of Theorem \ref{thm.linear} is complete.

\section{Nonlinear problem}\label{sec.nonlinear}

In this section we prove Theorem \ref{thm.main} based on the linear result in the previous section, Theorems \ref{thm.zero} and \ref{thm.linear}.
The proof relies on the fixed point theorem for the contraction map.
Let us denote by $X_{\nu,\epsilon}$ the closed convex set defined by
\begin{align}\label{proof.thm.main.1}
\begin{split}
& X_{\nu,\epsilon}  = \big \{ u \in X ~|~  \\
& \quad \|u \|_{X_\nu} :=  \| u_{0,1} \|_{L^\infty} +  \nu^{\frac14} \| \pa_y u_{0,1} \|_{L^2} + \sum_{n\ne 0} \| u_n \|_{L^\infty}   + \nu^{-\frac14} \| \mathcal{Q}_0 u\|_{L^2} +  \nu^\frac14 \| \nabla \mathcal{Q}_0 u \|_{L^2} \leq \epsilon \frac{\nu^\frac{1}{2}}{|\log \nu|^\frac12}  \big \}\,.
\end{split}
\end{align}
Then for $w\in X_{\nu,\epsilon}$ we define the map $\Psi [w]=u$ as the solution to the linear problem 
\begin{align}\label{eq.pS}
\begin{cases}
& U_s^\nu \pa_x u + u_2 \pa_y U_s^\nu {\bf e}_1 - \nu \Delta u + \nabla p =- w \cdot \nabla w   + f^\nu \,, \quad (x,y) \in \T_\kappa \times \R_+\,,\\
& {\rm div}\, u^\nu =0\,, \quad (x,y) \in \T_\kappa\times \R_+\,,\\
& u^\nu|_{y=0} =0\,.
\end{cases}
\end{align}
We observe that 
\begin{align*}
-w\cdot \nabla w = -w_{0,1} \pa_x \mathcal{Q}_0 w - \mathcal{Q}_0 w_2  \, \pa_y w_{0,1} {\bf e}_1 - \mathcal{Q}_0 w \cdot \nabla \mathcal{Q}_0 w\,,
\end{align*}
and therefore,
\begin{align*}
-\mathcal{P}_0 (w\cdot \nabla w ) = - \mathcal{P}_0 \big ( \mathcal{Q}_0 w \cdot \nabla \mathcal{Q}_0 w\big )  = - \pa_y \mathcal{P}_0 \big ( \mathcal{Q}_0 w_2 \, \mathcal{Q}_0 w \big )\,.
\end{align*}
Hence the zero mode of the source term in \eqref{eq.pS} is written as 
\begin{align}\label{proof.thm.main.3}
\mathcal{P}_0 \Big ( - w \cdot \nabla w  + f^\nu \Big ) = \pa_y \Big ( - \mathcal{P}_0 \big ( \mathcal{Q}_0 w_2 \, \mathcal{Q}_0 w \big )  \Big ) =: \pa_y H[w] \,.
\end{align}
Thus, Theorem \ref{thm.zero} shows 
\begin{align}\label{proof.thm.main.5}
\begin{split}
\| u_{0,1} \|_{L^\infty}  \leq \frac{1}{\nu} \| H[w]_1\|_{L^1} & \leq \frac{1}{\nu} \| \mathcal{Q}_0 w\|_{L^2}^2  \,,\\
\|\pa_y u_{0,1} \|_{L^2} \leq \frac{1}{\nu} \| H[w]_1\|_{L^2} & \leq \frac{1}{\nu}  \| \mathcal{Q}_0 w\|_{L^\infty} \| \mathcal{Q}_0 w\|_{L^2} \,,  \\
u_{0,1}|_{y=0} & =0\,.
\end{split}
\end{align}  
As for the nonzero mode, we see 
\begin{align*}
& \mathcal{Q}_0 \Big ( - w \cdot \nabla w + f^\nu \Big )  =  -w_{0,1} \pa_x \mathcal{Q}_0 w - \mathcal{Q}_0 w_2  \, \pa_y w_{0,1} {\bf e}_1 - \mathcal{Q}_0 \big ( \mathcal{Q}_0 w \cdot \nabla \mathcal{Q}_0 w \big ) + f^\nu \,.
\end{align*}
Therefore, Theorem \ref{thm.linear} implies, by applying the Parseval equality,
\begin{align}\label{proof.thm.main.6}
& \| \mathcal{Q}_0 u \|_{L^2}  \nonumber \\
&  \leq C \| -w_{0,1} \pa_x \mathcal{Q}_0 w - \mathcal{Q}_0 w_2  \, \pa_y w_{0,1} {\bf e}_1 - \mathcal{Q}_0 \big ( \mathcal{Q}_0 w \cdot \nabla \mathcal{Q}_0 w \big ) + f^\nu \|_{L^2} \nonumber \\
& \leq  C\big ( \| w_{0,1}\|_{L^\infty} \| \nabla \mathcal{Q}_0 w\|_{L^2} + \| \mathcal{Q}_0 w \|_{L^\infty} \| \pa_y w_{0,1} \|_{L^2} + \| \mathcal{Q}_0  w \|_{L^\infty} \| \nabla \mathcal{Q}_0 w\|_{L^2} + \| f^\nu \|_{L^2} \big )\,,
\end{align}
and similarly,
\begin{align}\label{proof.thm.main.7}
& \| \nabla \mathcal{Q}_0 u \|_{L^2} \nonumber \\
& \leq \frac{C}{\nu^\frac12} \big ( \| w_{0,1}\|_{L^\infty} \| \nabla \mathcal{Q}_0 w\|_{L^2} + \| \mathcal{Q}_0 w \|_{L^\infty} \| \pa_y w_{0,1} \|_{L^2} + \| \mathcal{Q}_0  w \|_{L^\infty} \| \nabla \mathcal{Q}_0 w\|_{L^2} + \| f^\nu \|_{L^2} \big )\,.
\end{align}
For $n \ne 0$, by using the identity 
\begin{align*}
& \mathcal{P}_n \Big ( - w \cdot \nabla w + f^\nu \Big )  =  -w_{0,1} \pa_x \mathcal{P}_n w - \mathcal{P}_n w_2  \, \pa_y w_{0,1} {\bf e}_1 - \mathcal{P}_n \big ( \mathcal{Q}_0 w \cdot \nabla \mathcal{Q}_0 w \big ) + \mathcal{P}_n f^\nu \,,
\end{align*}
we have from Theorem \ref{thm.linear},
\begin{align}\label{proof.thm.main.8}
& \| \mathcal{P}_n u \|_{L^\infty} \nonumber \\
& \leq 
\begin{cases}
& \displaystyle \frac{C}{|\tilde n|^\frac12 \nu^\frac{1}{4}} \big ( \| w_{0,1}\|_{L^\infty} \| \nabla \mathcal{P}_n w\|_{L^2} + \| \mathcal{P}_n w \|_{L^\infty} \| \pa_y w_{0,1} \|_{L^2} + \|   \mathcal{P}_n \big ( \mathcal{Q}_0 w \cdot \nabla \mathcal{Q}_0 w \big )\|_{L^2} + \| \mathcal{P}_n f^\nu \|_{L^2} \big )\,,\\
& \qquad {\rm if}\quad 0<|\tilde n|\leq \delta_*\nu^{-\frac34}\,,\\
& \displaystyle \frac{C}{|\tilde n|^\frac32 \nu} \big ( \| w_{0,1}\|_{L^\infty} \| \nabla \mathcal{P}_n w\|_{L^2} + \| \mathcal{P}_n w \|_{L^\infty} \| \pa_y w_{0,1} \|_{L^2} + \|   \mathcal{P}_n \big ( \mathcal{Q}_0 w \cdot \nabla \mathcal{Q}_0 w \big )\|_{L^2} + \| \mathcal{P}_n f^\nu \|_{L^2} \big )\,,\\
& \qquad {\rm if}\quad |\tilde n| >\delta_* \nu^{-\frac34}\,.
\end{cases}
\end{align}
Note that the $L^\infty$ estimate in the case $|\tilde n| \geq  \nu^{-\frac12}$ follows from (ii) and (iii) in Theorem \ref{thm.linear} combined with the interpolation $\|\mathcal{P}_n u \|_{L^\infty} \leq C \| \pa_y u_n \|_{L^2_y}^\frac12 \| u_n \|_{L^2_y}^\frac12$.
Thus, by the Parseval equality,
\begin{align}\label{proof.thm.main.9}
\sum_{n\ne 0} \| \mathcal{P}_n u \|_{L^\infty}  & \leq \frac{C}{\nu^\frac{1}{4}} \big ( |\log \nu|^\frac12 \| w_{0,1}\|_{L^\infty} \| \nabla \mathcal{Q}_0 w\|_{L^2} + \sum_{n \ne 0} \| \mathcal{P}_n w \|_{L^\infty} \| \pa_y w_{0,1} \|_{L^2}  \nonumber \\
& \qquad +  |\log \nu| ^\frac12 \|   \mathcal{Q}_0 \big ( \mathcal{Q}_0 w \cdot \nabla \mathcal{Q}_0 w \big )\|_{L^2} + |\log \nu|^\frac12 \| \mathcal{Q}_0 f^\nu \|_{L^2} \big ) \nonumber \\
\begin{split}
& \leq \frac{C}{\nu^\frac{1}{4}} \big ( |\log \nu|^\frac12 \| w_{0,1}\|_{L^\infty} \| \nabla \mathcal{Q}_0 w\|_{L^2} + \sum_{n \ne 0} \| \mathcal{P}_n w \|_{L^\infty} \| \pa_y w_{0,1} \|_{L^2}  \\
& \qquad + |\log \nu|^\frac12 \| \mathcal{Q}_0  w \|_{L^\infty} \| \nabla \mathcal{Q}_0 w\|_{L^2} + |\log \nu|^\frac12 \| f^\nu \|_{L^2} \big )\,.
\end{split}
\end{align}
We note that $\| \mathcal{Q}_0 w\|_{L^\infty} \leq C \sum_{n\ne 0} \| \mathcal{P}_n w \|_{L^\infty}$ holds. Collecting these, we have 
\begin{align}\label{proof.thm.main.10}
\| \Psi [w] \|_{X_\nu} \leq \frac{C|\log \nu|^\frac12}{\nu^\frac{1}{2}} \| w \|_{X_\nu}^2 + \frac{C|\log\nu|^\frac12}{\nu^\frac{1}{4}} \| f^\nu \|_{L^2}\,,
\end{align}
and since the nonlinear term $-w\cdot \nabla w$ is quadratic, we can show a similar bound for the difference $\Psi[w]-\Psi[w']$:
\begin{align}\label{proof.thm.main.11}
\| \Psi [w]-\Psi[w']\|_{X_\nu} \leq \frac{C|\log \nu|^\frac12}{\nu^\frac{1}{2}} (\| w\|_{X_\nu} + \| w' \|_{X_\nu} ) \| w- w' \|_{X_\nu}\,.
\end{align}
Hence $\Psi$ is a contraction map from $X_{\nu,\epsilon}$ into itself if $\epsilon$ is small enough and $\| f^\nu \|_{L^2}$ is small enough compared with $\epsilon |\log \nu|^{-1} \nu^{\frac{1}{4}+\frac12}$. 
By the standard fixed point theorem there exists a unique fixed point $u^*$ of $\Psi$ in $X_{\nu,\epsilon}$, which is the unique solution to \eqref{eq.pNS} in $X_{\nu,\epsilon}$. 
Note that the solution $u^*$ satisfies $\| u^*\|_{X_\nu} \leq C |\log \nu|^\frac12 \nu^{-\frac{1}{4}}\| f^\nu \|_{L^2}$, which gives \eqref{est.thm.main.1}. Since 
$$g=-U_s^\nu \pa_x u^*- u_2^* \pa_y U_s^\nu {\bf e}_1 - u^* \cdot \nabla u^* + f^\nu$$
 belongs to $L^2 (\T_\kappa\times \R_+)^2$, the elliptic regularity of the Stokes equations also implies $\nabla^2 u^*\in L^2 (\T_\kappa \times \R_+)$ and $\nabla p^* \in L^2(\T_\kappa \times \R_+)^2$.The proof is complete.


{\small

\begin{thebibliography}{}
  
\bibitem{AbSt1964} Abramowitz, M.  and Stegun, I.; {\it Handbook of mathematical functions with formulas, graphs, and mathematical tables}, National Bureau of Standards Applied Mathematics Series, {\bf 55},1964, 1046 pp.

\bibitem{AlWaXuYa} Alexandre, R., Wang, Y.-G.,   Xu, C.-J., and  Yang, T.; Well-posedness of the Prandtl equation in Sobolev spaces. J. Amer. Math. Soc., 28(3):745-784, 2015.


\bibitem{Cowley} Cowley, S.; Laminar boundary layer theory: a 20th century paradox ? Mechanics for a new millenium, 389-412, 2001 Kluwer Academic Publisher. 


\bibitem{DalMas} Dalibard, A.-L., and Masmoudi, N.; Ph\'enom\`ene de s\'eparation pour l'\'equation de Prandtl. S\'eminaire Laurent Schwartz, Exp IX, 18. Ed. Ecole Polytechnique, 2016. 

\bibitem{DrRe}
Drazin, P. G., and  Reid,W.~H.;
\newblock {\em Hydrodynamic stability}.
\newblock Cambridge Mathematical Library. Cambridge University Press,
  Cambridge, second edition, 2004.
\newblock With a foreword by John Miles.

\bibitem{E} E, W.; Boundary layer theory and the zero-viscosity limit of the Navier-Stokes equation. Acta Math. Sin. 16, 2 (2000) 207-218. 

\bibitem{FeTaZha} Fei, M., Tao, T., and Zhang, Z.; On the zero-viscosity limit of the Navier-Stokes equations in $\R^3_+$ without analyticity. J. Math. Pures Appl., 112:170-229, 2018.

\bibitem{GaHiMa} Gallagher, I., Higaki, M., and Maekawa, Y.; On stationary two-dimensional flows around a fast rotating disk.  arXiv:1710.01029.


\bibitem{GaSaSci} Gargano, F., Sammartino, M., and Sciacca, V.; Singularity formation for Prandtl's equations. Phys. D, 238(19):1975-1991, 2009.

\bibitem{GeDo} Gerard-Varet, D., and  Dormy, E.; On the ill-posedness of the Prandtl equation. J. Amer. Math. Soc., 23(2):591-609, 2010.

\bibitem{GeMaMa} Gerard-Varet, D.,  Maekawa, Y., and Masmoudi, N.; Gevrey stability of Prandtl expansions for 2D Navier-Stokes flows. Accepted by Duke Math. J.  (available in ArXiv e-prints, July 2016) 

\bibitem{GeMa} Gerard-Varet, D., and Masmoudi, N.; Well-posedness for the Prandtl system without analyticity or monotonicity. Ann. Scient. Ec. Norm. Sup., 48(4):1273-1325, 2015.
 
 \bibitem{GePr} Gerard-Varet, D., and Prestipino, M.; Formal derivation and stability analysis of boundary layer models in MHD.  Z. Angew. Math. Phys. 68 (3), Art. 76 (2017). 
 
\bibitem{Gol} Goldstein, S.; On laminar boundary layer flow near a position of separation.  Quaterly J. Mech. Applied Math. 1, 43-69, 1948. 

\bibitem{Gre}
Grenier, E.;  On the nonlinear instability of {E}uler and {P}randtl equations. Comm. Pure Appl. Math., 53(9):1067-1091, 2000.

\bibitem{GGN2014}
Grenier, E., Guo, Y.,  and  Nguyen, T.;  Spectral instability of characteristic boundary layer flows. Duke Math. J., 165:3085-3146, 2016.

\bibitem{GrNg}
Grenier, E.,  and Nguyen, T.; Sharp bounds on linear semigroup of Navier-Stokes with boundary layer norms.
\newblock Preprint arXiv, 2017.

\bibitem{GuoIyer2018} Guo, Y., Iyer, S.; Validity of steady Prandtl layer expansions. arXiv:1805.05891, May 2018. 

\bibitem{GuNgu2017} Guo, Y., Nguyen, T.; Prandtl boundary layer expansions of steady Navier-Stokes flows over a moving plate. Ann. PDE 3 (2017), no. 1, Art. 10, 58 pp.

\bibitem{HoHu} Hong, L.,  and Hunter, J.; Singularity formation and instability in the unsteady inviscid and viscous Prandtl equations. Comm. Math. Sci., 1:293-316, 2003.

\bibitem{Iyer2017} Iyer, S.; Steady Prandtl boundary layer expansions over a rotating disk. Arch. Ration. Mech. Anal. 224 (2017), no. 2, 421-469.

\bibitem{KuVi}
Kukavica, I.,  and Vicol, V.;
\newblock On the local existence of analytic solutions to the {P}randtl boundary layer equations.
\newblock {\em Commun. Math. Sci.}, 11(1):269-292, 2013.

\bibitem{Tong1}
 Li, W.-X., and  Yang., T.; 
\newblock Well-posedness in Gevrey space for the Prandtl equations with non-degenerate critical points 
\newblock To appear in J. Eur. Math. Soc.  

\bibitem{Tong2} Liu, C.,  Xie, F.,   and Yang, T.; MHD  boundary  layers  in  Sobolev  spaces  without  monotonicity.  I.  Well-posedness 
theory,  to appear in CPAM. 

\bibitem{Tong3}  Liu, C.,  Xie, F.,   and Yang, T.; MHD  boundary  layers  in  Sobolev  spaces  without  monotonicity.  II.  Convergence theory,  preprint arXiv 2017.  

\bibitem{LoCaSa} Lombardo, M. C.,   Cannone, M., and  Sammartino, M.; Well-posedness of the boundary layer equations. SIAM J. Math. Anal., 35(4):987-1004, 2003.

\bibitem{LoMaNu} Lopes Filho, M. C.,   Mazzucato, A. L.,  and  Nussenzveig Lopes H. J.; Vanishing viscosity limit for incompressible flow inside a rotating circle. Phys. D, 237(10-12):1324-1333, 2008.

\bibitem{Ma} Maekawa, Y.; On the inviscid limit problem of the vorticity equations for viscous incompressible flows in the half plane, Comm. Pure and Applied Math., 67:1045-1128, 2014. 

\bibitem{MaWo}
Masmoudi, N., and Wong, T.~K.;
\newblock Local-in-time existence and uniqueness of solutions to the {P}randtl
  equations by energy methods.
\newblock {\em Comm. Pure Appl. Math.}, 68(10):1683-1741, 2015.

\bibitem{MaShi} Matsui, S., and Shirota, T.; On separation points of solutions to Prandtl boundary layer problem.
Hokkaido Math. J., 13(1): 92-108, 1984.

\bibitem{MaTa} Mazzucato, A. L.,  and Taylor, M. E.; Vanishing viscosity plane parallel channel flow and related singular perturbation problems. Analysis \& PDE, 1(1):35-93, 2008.

\bibitem{Ole}
Oleinik, O.~A., and  Samokhin.,V.~N.;
\newblock {\em Mathematical models in boundary layer theory}, volume~15 of 
  Applied Mathematics and Mathematical Computation.
\newblock Chapman \& Hall/CRC, Boca Raton, FL, 1999.

\bibitem{SaCa1}
Sammartino, M., and  Caflisch, R.~E.;
\newblock Zero viscosity limit for analytic solutions, of the {N}avier-{S}tokes
  equation on a half-space. {I}. {E}xistence for {E}uler and {P}randtl
  equations.
\newblock {\em Comm. Math. Phys.}, 192(2):433-461, 1998.

\bibitem{SaCa2}
Sammartino, M., and  Caflisch, R.~E.;
\newblock Zero viscosity limit for analytic solutions of the {N}avier-{S}tokes
  equation on a half-space. {II}. {C}onstruction of the {N}avier-{S}tokes
  solution.
\newblock {\em Comm. Math. Phys.}, 192(2):463-491, 1998.

\bibitem{Schlichting} Schlichting, H. {\em Boundary layer theory}.  Eighth edition. Springer-Verlag, Berlin, 2000.

\bibitem{WaWaZh} Wang, C., Wang, Y., and Zhang, Z.; Zero-viscosity Limit of the Navier-Stokes equations in the analytic setting. Arch. Rational. Mech. Anal. 224 (2017) 555-595. 

\bibitem{XiZa}  Xin, Z.,  and Zhang., L.;  On the global existence of solutions to the Prandtl's system. Adv. Math.,
181:88-133, 2004.
\end{thebibliography}
}
\end{document}